\documentclass[a4paper]{article}

\usepackage{natbib}

\usepackage{comment}

\usepackage[all]{xy}\usepackage[latin1]{inputenc}        
\usepackage[dvips]{graphics,graphicx}
\usepackage{amsfonts,amssymb,amsmath,xcolor,mathrsfs, amstext}
\usepackage{amsbsy, amsopn, amscd, amsxtra, amsthm,authblk, enumerate,dsfont}
\usepackage{upref}
\usepackage{geometry}
\usepackage[displaymath]{lineno}
\usepackage{float}

\usepackage[colorlinks,
            linkcolor=blue,
            anchorcolor=green,
            citecolor=blue
            ]{hyperref}

\numberwithin{equation}{section}

\def\e{\varepsilon}
\def\epsilon{\varepsilon}
\def\eps{\varepsilon}

\newcommand{\ol}{\overline}

\def\alb#1\ale{\begin{align*}#1\end{align*}}
\newcommand{\eqb}{\begin{equation}}
\newcommand{\eqe}{\end{equation}}

\DeclareMathOperator{\supp}{supp}
\DeclareMathOperator{\dist}{dist}

\let\Im\undefined
\DeclareMathOperator{\Im}{Im}

\newcommand{\bbC}{\mathbb{C}}
\newcommand{\bbD}{\mathbb{D}}
\newcommand{\bbE}{\mathbb{E}}
\newcommand{\bbH}{\mathbb{H}}

\newcommand{\bbR}{\mathbb{R}}
\newcommand{\bbP}{\mathbb{P}}

\newcommand{\cA}{\mathcal{A}}
\newcommand{\cB}{\mathcal{B}}
\newcommand{\cC}{\mathcal{C}}
\newcommand{\cD}{\mathcal{D}}
\newcommand{\cF}{\mathcal{F}}

\newcommand{\cL}{\mathcal{L}}
\newcommand{\cM}{\mathcal{M}}

\newcommand{\cS}{\mathcal{S}}
\newcommand{\cV}{\mathcal{V}}
\newcommand{\cW}{\mathcal{W}}

\newcommand{\QT}{\mathrm{QT}}
\newcommand{\LF}{\mathrm{LF}}
\newcommand{\SLE}{\mathrm{SLE}}
\newcommand{\Wd}{\mathrm{Weld}}
\newcommand{\Md}{{\mathcal{M}}^\mathrm{disk}}

\newtheorem{theorem}{Theorem}[section]
\newtheorem{lemma}[theorem]{Lemma}
\newtheorem{proposition}[theorem]{Proposition}
\newtheorem*{proposition*}{Proposition}

\newtheorem*{corollary*}{Corollary}
\newtheorem{definition}[theorem]{Definition}
\newtheorem*{definitions*}{Definitions}

\newtheorem*{example*}{\bf Example}
\theoremstyle{remark}

\numberwithin{equation}{section}

\title{Radial conformal welding in Liouville quantum gravity}
\author{Morris Ang   \qquad\quad   Pu Yu}

\date{}
\begin{document}

\maketitle

\begin{abstract}
The seminal work of Sheffield showed that when random surfaces called Liouville quantum gravity (LQG) are conformally welded, the resulting interface is Schramm-Loewner evolution (SLE). This has been proved for a variety of configurations, and has applications to the scaling limits of random planar maps and the solvability of SLE and Liouville conformal field theory. We extend the theory to the setting where two sides of a canonical three-pointed LQG surface are conformally welded together, resulting in a radial SLE curve which can be described by imaginary geometry.
\end{abstract}

\tableofcontents

\section{Introduction}
	Over the past two decades, Schramm-Loewner evolution (SLE) and Liouville quantum gravity (LQG)  have become central subjects in random conformal geometry as canonical models of random curves and surfaces. As discovered in the landmark work of~\cite{She16a}, SLE curves arise as the interfaces in the  conformal weldings of LQG surfaces. This is a fundamental ingredient for the mating-of-trees theory~\cite{DMS14}, a powerful framework relating LQG and the scaling limits of random planar maps \cite{shef-burger, gms-tutte,HS19, bgs-smith}; see also~\cite{GHS19, BP21}. {Conformal welding has also found numerous  applications in the derivations of exact formulas in SLE, Liouville conformal field theory (LCFT), and critical lattice models, see e.g.\ \cite{AHS21,ARS21, NQSZ-backbone, AHSY23, ASYZ24, AY24, ACSW24, LSYZ24}.}

	Though the mathematical literature on LCFT (e.g., \cite{DKRV16, HRV-disk, KRV20, GRV19, GKRV-bootstrap}) focuses on different aspects of Liouville theory compared to the conformal welding approach, it was shown that certain finite-volume LQG surfaces in \cite{DMS14} are described by LCFT   \cite{AHS17,cercle2021unit, AHS21}. This motivated the definition via LCFT of
	 quantum triangles~\cite{ASY22}, which are finite-volume LQG surfaces with three marked points on the boundary. As shown in~\cite{ASY22, SY23}, the chordal $\SLE_\kappa(\underline\rho)$, an important variant of the SLE, arises as the interface for {certain conformal weldings involving  quantum triangles.}
	
	In this paper, we consider the conformal welding of LQG surfaces in the radial setting. We show that welding a quantum triangle to itself gives a version of quantum disk with one bulk marked point and one boundary marked point decorated with an independent radial SLE curve (Figures~\ref{fig:radial-thick} and~\ref{fig:radial-thin}); see Theorem~\ref{thm:main} for more details. {Similarly as in the chordal case, this curve arises naturally in the imaginary geometry framework~\cite{MS16a}.} The radial imaginary geometry was {constructed} in~\cite{ig4} for radial SLE with a single force point; we extend the framework to  radial SLE with multiple force points. In Theorem~\ref{thm:main-nonsimple}, we prove a variant of Theorem~\ref{thm:main} for forested quantum surfaces which gives radial $\SLE_\kappa$ in the nonsimple regime $\kappa \in (4,8)$.
	
	\subsection{Quantum triangles}\label{subsec:intro-qt}
	In this section we give a brief introduction to the quantum triangles appearing in our main results. See Section~\ref{sec:pre-LQG}
	 for more details. 
	
	Fix $\gamma\in (0,2)$. A quantum surface in $\gamma$-LQG is a surface with an area measure~\cite{DS11} and a   metric structure~\cite{DDDF19,GM19metric}  induced by a variant of the Gaussian free field (GFF). A quantum surface with the disk topology can be represented as a pair $(D,h)$ where $D \subset \bbC$ is a simply connected domain and $h$ is a variant of the GFF on $D$. For such surfaces there is also a notion of $\gamma$-LQG length measure on the disk boundary~\cite{DS11}.   Two pairs $(D,h)$  and $(D',h')$ represent the same quantum surface if there is a conformal map between $D$ and $D'$ preserving the $\gamma$-LQG geometry, and a particular pair $(D,h)$ is called a (conformal) embedding of the quantum surface.

	 For $W>0$, the two-pointed quantum disk of weight $W$, {whose law is denoted by $\Md_{2}(W)$}, is a quantum surface having two marked points on the boundary~\cite{DMS14,AHS20}. When $W \ge \frac{\gamma^2}2$ the quantum disk is simply connected and is called \emph{thick}, whereas when $W\in(0,\frac{\gamma^2}{2})$ the quantum disk is a chain of countably many thick quantum disks and is called \emph{thin}. 
	
	We now use the Liouville field coming from LCFT~\cite{HRV-disk} to define quantum triangles; see Section~\ref{subsec:pre-lf} for the definition of Liouville field. 
	The quantum triangle has three weight parameters $W_1,W_2,W_3>0$.   For $W_1,W_2,W_3>\frac{\gamma^2}{2}$, set $\beta_i = \gamma+\frac{2-W_i}{\gamma}<Q$ and let $\textup{LF}_{\bbH}^{(\beta_1, \infty), (\beta_2, 0), (\beta_3, 1)}$ be the law of the Liouville field on the upper half plane $\bbH$ 
	with insertions $\beta_1,\beta_2,\beta_3$ at points  $\infty,0$ and $1$, respectively. 
	Sample $\phi$ from 
	  \[
	  \frac{1}{(Q-\beta_1)(Q-\beta_2)(Q-\beta_3)}\textup{LF}_{\bbH}^{(\beta_1, \infty), (\beta_2, 0), (\beta_3, 1)}.
	  \]
	  We define  $\textup{QT}(W_1, W_2, W_3)$ to be the law of the quantum surface {$(\bbH, \phi)$ having three marked points $(\infty, 0, 1)$,} and call a sample from {$\QT(W_1,W_2,W_3)$} a quantum triangle of weight $(W_1,W_2,W_3)$. This definition can be extended to the parameter range $W_1, W_2, W_3 \geq \frac{\gamma^2}2$ by a limiting procedure; see~\cite[Section 2.5]{ASY22} for more details.
	  Finally, if $W_1, W_2, W_3 > 0$ and $I = \{ i : W_i < \frac{\gamma^2}2\}$, we can define $\QT(W_1, W_2, W_3)$ by first sampling a quantum triangle $\mathcal T$ from $\QT( \tilde W_1, \tilde W_2, \tilde W_3)$ where $\tilde W_i := \max ( W_i, \gamma^2 - W_i)$, then independently sampling for each $i \in I$ a weight $W_i$ quantum disk $\cD_i$, and finally attaching each $\cD_i$ to the corresponding vertex of $\mathcal T$. 
	We call the three marked points of a quantum triangle its vertices. Given a sample of  $\QT(W_1,W_2,W_3)$, the geometry near the vertex of weight $W_i$ looks like a neighborhood of a marked point on a weight-$W_i$ quantum disk; we call a vertex thick if $W_i \geq \frac{\gamma^2}2$ and thin otherwise.  See Figure~\ref{fig-qt} for an illustration. 
	
	{
	The quantum triangles with weights $W_1, W_2, W_3 > \frac{\gamma^2}2$ are random surfaces whose geometric observables (i.e., area, lengths), by definition, encode the \emph{three-point structure constants} of boundary LCFT. 
	An analogous statement holds when some weights satisfy $W_i \in (0, \frac{\gamma^2}2)$, due to the reflection principle for the three-point structure constants \cite[Lemma 3.11]{ARSZ23} (in fact, this is one motivation for the definition via adding independent quantum disks).
	}
	
	\begin{figure}[ht]
		\centering
		\includegraphics[scale=0.43]{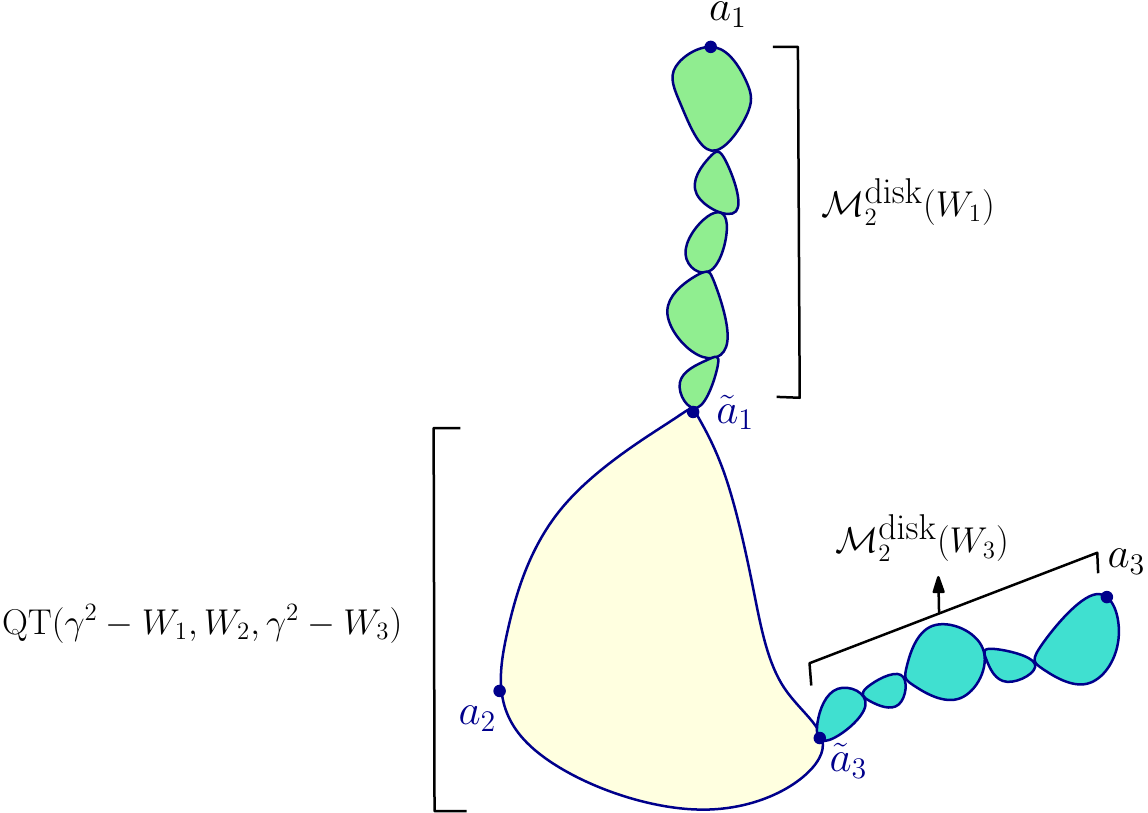}
		\caption{A sample of $\QT(W_1,W_2,W_3)$ with a thick vertex $a_2$   and  two thin vertices $a_1,a_3$, i.e.\ $W_2\geq\frac{\gamma^2}{2}$ and $W_1,W_3<\frac{\gamma^2}{2}$. The yellow surface is a quantum triangle with thick vertices $\tilde{a}_1,a_2,\tilde{a}_3$.  The two thin two-pointed quantum disks (colored green) are concatenated with the  yellow triangle at  $\tilde{a}_1$ and $\tilde{a}_3$.} \label{fig-qt}
	\end{figure}
	
	\subsection{Radial conformal welding for quantum triangles}\label{subsec:intro-main-thm}
	
	Before stating our main result Theorem~\ref{thm:main}, we discuss related results for the conformal welding for quantum disks and quantum triangles.  
	
	For $W>0$, consider the disintegration  $\mathcal{M}_2^{\textup{disk}}(W)=\iint_0^\infty \mathcal{M}_2^{\textup{disk}}(W;\ell,r) d\ell d r$, where each measure $\mathcal{M}_2^{\textup{disk}}(W;\ell,r)$ is supported on the space of quantum surfaces with left boundary length $\ell$ and right boundary length $r$. 
    Given a pair of quantum surfaces sampled from $\mathcal{M}_2^{\textup{disk}}(W_1;\ell_1,\ell)\times \mathcal{M}_2^{\textup{disk}}(W_2;\ell,\ell_2)$, we can conformally weld them together along the boundary arcs having length $\ell$ to obtain a quantum surface decorated with a curve. We denote its law by $\mathrm{Weld}(\mathcal{M}_2^{\textup{disk}}(W_1;\ell_1,\ell),\mathcal{M}_2^{\textup{disk}}(W_2;\ell,\ell_2))$. 
	For $\kappa>0$, $\rho_->-2$ and $\rho_+>-2$, chordal  $\SLE_\kappa(\rho_-;\rho_+)$ is a classical variant of SLE$_\kappa$; see Section~\ref{subsec:pre-sle-ig} for a precise definition. Let $(D,h,a,b)$ be an embedding of a two-pointed quantum disk sampled from $\mathcal{M}_2^{\textup{disk}}(W_1+W_2)$ with $a,b$ being the two boundary marked points. We  draw a chordal $\SLE_\kappa(W_1-2;W_2-2)$ curve $\eta$ in $D$ from $a$ to $b$ independent of $h$ where $\kappa=\gamma^2\in(0,4)$, and write $\mathcal{M}_2^{\textup{disk}}(W_1+W_2)\otimes  {\SLE}_\kappa(W_1-2;W_2-2)$ for  the law of the curve-decorated surface described by $(D,h,\eta,a,b)$.    As shown in~\cite[Theorem 2.2]{AHS20}, there is a constant $c>0$ such that
	\begin{equation}\label{eq:2pt-weld}
	    \mathcal{M}_2^{\textup{disk}}(W_1+W_2)\otimes  {\SLE}_\kappa(W_1-2;W_2-2)=c\textup{Weld} (\mathcal{M}_2^{\textup{disk}}(W_1),\mathcal{M}_2^{\textup{disk}}(W_2)),
	\end{equation}
	 where $\textup{Weld} (\mathcal{M}_2^{\textup{disk}}(W_1),\mathcal{M}_2^{\textup{disk}}(W_2)) := \iiint_0^\infty \textup{Weld}( \mathcal{M}_2^{\textup{disk}}(W_1;\ell,\ell_1),\mathcal{M}_2^{\textup{disk}}(W_2;\ell_1,\ell_2))d\ell d\ell_1d\ell_{2}$ is called the conformal welding of $\mathcal{M}_2^{\textup{disk}}(W_1)$ and $\mathcal{M}_2^{\textup{disk}}(W_2)$. 
	 {A similar result was obtained in \cite[Theorem 1.1]{ASY22} for the conformal welding of a quantum disk and quantum triangle; see Theorem~\ref{thm:disk+QT} for details.}
	
	Our main theorem is a radial analog of these results, where we consider the welding of a quantum triangle to itself.  For $W>0$ and $W'>\frac{\gamma^2}{2}$, let $\beta = \gamma+\frac{2-W'}{\gamma}$ and $\alpha = Q-\frac{W}{2\gamma}$. Let $\Md_{1,1}(W,W')$ be the law of the quantum surface $(\bbH,\phi,i,0)$ where $\phi$ is sampled from $\LF_{\bbH}^{(\alpha,i),(\beta,0)}$. For $\rho_-,\rho_+>-2$ with $\rho_-+\rho_+>-2$, write $\mathrm{raSLE}_\kappa(\rho_-;\rho_+)$ for the law of radial SLE with two force points of weights $\rho_-,\rho_+$ lying infinitesimally close to the root; see Section~\ref{subsec:radial-ig} for a definition. {Finally, let $\QT(W_1,W_2,W_3;\ell,\ell)$ denote the law of the quantum triangle with weights $W_1, W_2, W_3$ whose boundary arcs adjacent to the weight $W_1$ vertex both have quantum length $\ell$, as defined in Section~\ref{subsec:pre-qs}.} We are now ready to state our main result; see Figures~\ref{fig:radial-thick} and~\ref{fig:radial-thin}.
	\begin{theorem}\label{thm:main}
	Let $\gamma\in (0,2)$, $\kappa=\gamma^2$ and suppose $W_1, W_2, W_3>0$ satisfy $W_2+W_3=2+W_1$. Suppose that $W_1 \neq \frac{\gamma^2}2$. Then there exists some constant $c=c_{W_1,W_2,W_3}\in (0,\infty)$ such that, 
	\begin{equation}\label{eqn:thm-main-0}
	    \Md_{1,1}(W_1,W_1+2)\otimes \mathrm{raSLE}_\kappa(W_3-2;W_2-2) = c\int_0^\infty \Wd(\QT(W_1,W_2,W_3;\ell,\ell))\,d\ell.
	\end{equation}
	In other words, for a sample from  $\Md_{1,1}(W_1,W_1+2)$ embedded as $(\bbH,\phi,0,i)$, if one draws an independent radial $\SLE_\kappa(W_3-2;W_2-2)$ targeted at $i$ with force points at $0^-,0^+$,  the law of $(\bbH\backslash\eta,\phi,i,0^-,0^+)$ viewed as a marked quantum surface is given by  
	\begin{equation}\label{eqn:thm-main}
	    c\int_0^\infty \QT(W_1,W_2,W_3;\ell,\ell)\, d\ell.
	\end{equation}
	\end{theorem}
	\begin{figure}[ht]
		\centering
		\includegraphics[scale=0.69]{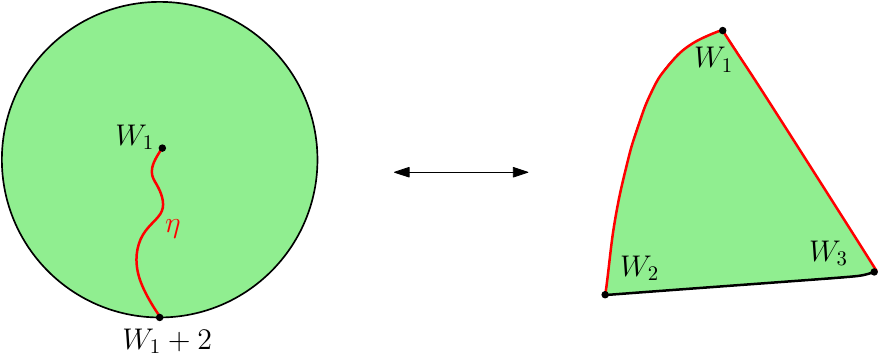}
		\caption{An illustration of Theorem \ref{thm:main} in the case $W_1,W_2,W_3>\frac{\gamma^2}{2}$. Conformally welding two edges of the weight $(W_1,W_2,W_3)$ quantum triangle gives a quantum disk with one bulk insertion of weight $W_1$ and one boundary insertion of weight $W_1+2$, and the interface is an independent radial $\SLE_\kappa(W_3-2;W_2-2)$ curve independent of the field. }\label{fig:radial-thick}
	\end{figure}
	\begin{figure}[ht]
		\centering
		\includegraphics[scale=0.69]{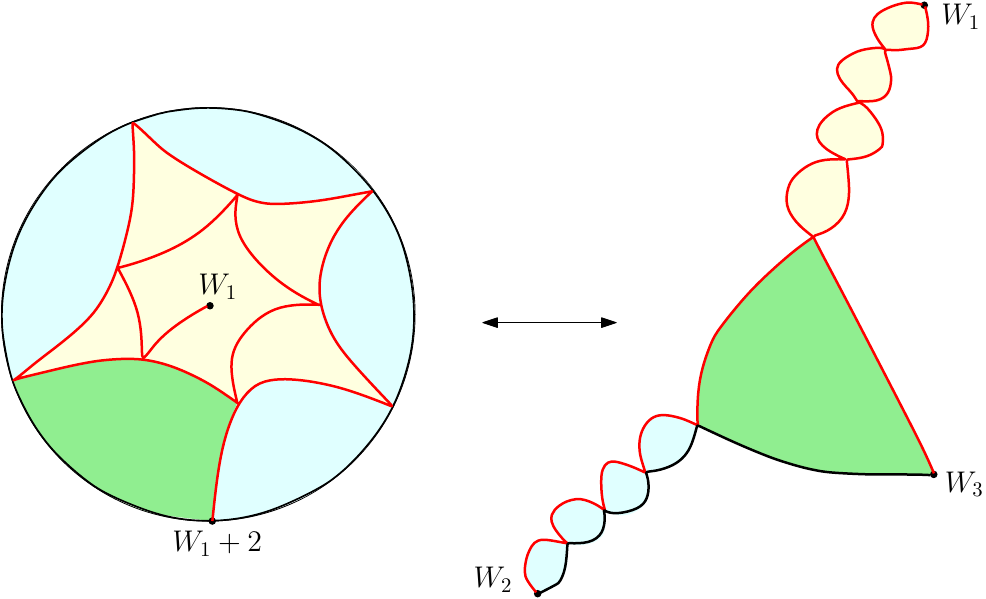}
		\caption{An illustration of Theorem \ref{thm:main} in the case $W_1,W_2<\frac{\gamma^2}{2}$ and $W_3>\frac{\gamma^2}{2}$. The radial $\SLE_\kappa(W_3-2;W_2-2)$ curve is now both self-hitting and boundary-hitting, producing the welding of a thin quantum triangle. }\label{fig:radial-thin}
	\end{figure}
	
	The full statement of  Theorem~\ref{thm:main} is obtained from bootstrapping the following special case:
	\begin{theorem}\label{thm:w2w}
	For $W \in (0, \frac{\gamma^2}2) \cup (\frac{\gamma^2}2,\infty)$, Theorem~\ref{thm:main} holds for $W_1=W_2=W$ and $W_3= 2$. 
	\end{theorem}
	In Section~\ref{subsec:weld-qt}, we will explain how Theorem~\ref{thm:main} follows from Theorem~\ref{thm:w2w}. We expect that the $W = \frac{\gamma^2}2$ case of Theorem~\ref{thm:w2w} can be obtained by a limiting argument (and so Theorem~\ref{thm:main} holds without the condition $W_1 \neq \frac{\gamma^2}2$), but do not pursue this in the present work.
	
	We now discuss the proof of Theorem~\ref{thm:w2w}, beginning with the identification of the random field obtained from conformal welding. For $W>\frac{\gamma^2}{2}$, \cite{AHS21} identifies the uniform embedding of a two-pointed quantum disk via the Liouville field, and \cite{Ang23} gives a description of the field obtained from conformally welding a Liouville field to itself; combining these gives the law of the field. For $W\in(0,\frac{\gamma^2}{2})$, the Poissonian structure of thin quantum disks allows us to identify conformal weldings of various parts via Liouville fields, yielding resampling properties of the field. These resampling properties characterize the Liouville field as proposed in~\cite{ARS22,ASY22}. 
	
	Next, in the proof of Theorem~\ref{thm:w2w}, we need to determine the conditional law of the interface given the field. We explain a radial version of imaginary geometry where the flow lines are radial $\SLE_\kappa(\underline\rho)$ curves; see Theorem~\ref{thm:radial-ig}. Although not explicitly stated there, the arguments in~\cite{ig4}  readily imply Theorem~\ref{thm:radial-ig}, and we include the statement and a (slightly different) proof for completeness.  Then we extend the resampling property for flow lines in~\cite[Theorem 4.1]{MS16b} and \cite[Theorem 5.1]{ig4} to the radial setting, from which we prove that the interface indeed agrees in law with the radial flow lines, completing the proof of Theorem~\ref{thm:w2w}.
	
	Finally, using a similar approach, we will prove an analog of Theorem~\ref{thm:main} with $\kappa \in (4,8)$. It involves \emph{forested} quantum surfaces, which we introduce in Section~\ref{sec-forested}. See Figure~\ref{fig:thm:main-non-simple} for an illustration for the case $W_1,W_2,W_3<\frac{\gamma^2}{2}$.
	
	\begin{theorem}\label{thm:main-nonsimple}
		Let $\gamma\in (\sqrt2,2)$, $\kappa'=16/\gamma^2$ and suppose $W_1, W_2, W_3>0$ satisfy $W_2+W_3=\gamma^2-2+W_1$. Suppose that $W_1 \neq \frac{\gamma^2}2$. Then there exists some constant $c=c_{W_1,W_2,W_3}\in (0,\infty)$ such that, 
		\begin{equation}\label{eqn:thm-main-nonsimple}
			\mathcal{M}^\mathrm{f.d.}_{1,1}(W_1 + 2 - \frac{\gamma^2}2, W_1+\frac{\gamma^2}2)\otimes \mathrm{raSLE}_{\kappa'}(W_3-(\gamma^2-2);W_2-(\gamma^2-2)) = c\int_0^\infty \Wd(\QT^f(W_1,W_2,W_3;\ell,\ell))\,d\ell.
		\end{equation}
	\end{theorem}

    \begin{figure}
        \centering
        \includegraphics[scale=0.6]{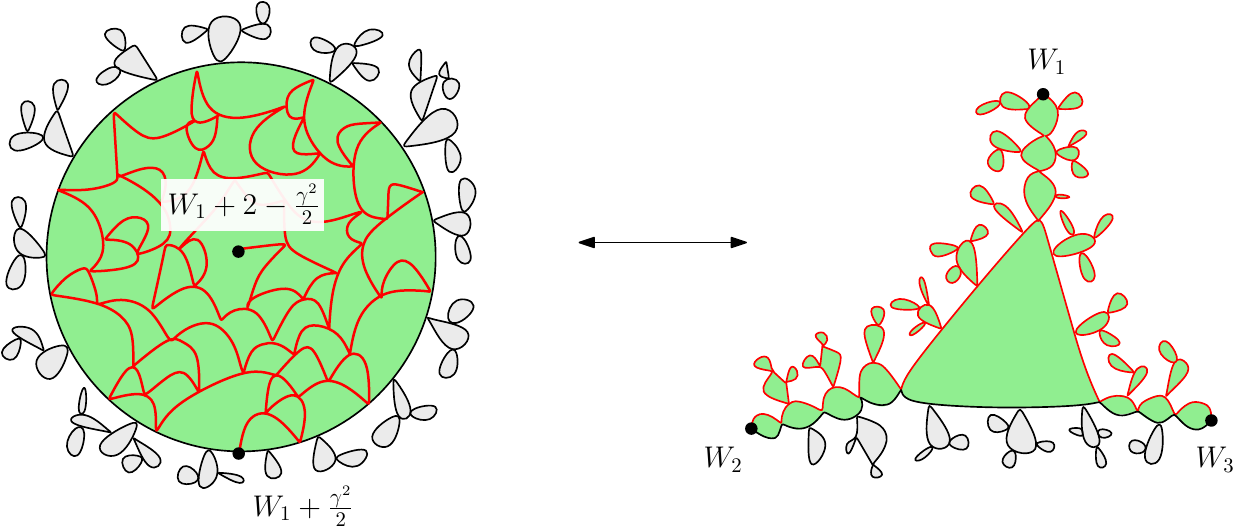}
        \caption{An illustration of Theorem~\ref{thm:main-nonsimple} when $W_1,W_2,W_3<\frac{\gamma^2}{2}$. Conformally welding two boundary arcs of a weight $(W_1, W_2, W_3)$ forested quantum triangle gives a forested quantum disk with a bulk insertion of weight $W_1 + 2 - \frac{\gamma^2}2$ and a boundary insertion of weight $W_1 + \frac{\gamma^2}2$, and the interface is an independent radial $\SLE_{\kappa'} (W_3 - (\gamma^2 - 2); W_2 - (\gamma^2-2))$ curve.
      }
        \label{fig:thm:main-non-simple}
    \end{figure}
	
	We note that during the process of writing the present work, a number of variants and special cases of Theorem~\ref{thm:main} and Theorem~\ref{thm:main-nonsimple} were proved and used in other works. These  were motivated from our work, and have applications that we now describe. In~\cite{ASYZ24} with Sun and Zhuang, 
	we proved the $(W_1, W_2, W_3) = (2-\frac{\gamma^2}{2},2-\frac{\gamma^2}{2},\gamma^2-2)$ case of Theorem~\ref{thm:main-nonsimple}
	 using a mating-of-trees description for quantum surfaces~\cite{DMS14,AG19}, from which we determined the {clockwise/counterclockwise} winding probability for radial $\SLE_{\kappa'}(\kappa'-6)$ and further the boundary touching probability for a non-simple $\mathrm{CLE}_{\kappa'}$ loop; this special case of Theorem~\ref{thm:main-nonsimple} was later also applied in \cite{sxz-annulus-perc} to derive annulus crossing formulae for critical planar percolation, and in \cite{cfsx-annulus-bm} to solve for the probability Brownian motion on an annulus disconnects the two boundary components of the annulus. In the recent work~\cite{LSYZ24}, with Liu, Sun and Zhuang the second named author proved a variant of Theorem~\ref{thm:main} with $W_2+W_3=\gamma^2-2$ and the interface being radial $\SLE_\kappa(\rho;\kappa-6-\rho)$ (this falls outside the range of Theorem~\ref{thm:main} as the interface a.s.\ hits the continuation threshold before reaching the target). Based on this, they determined the {clockwise/counterclockwise} winding probability for the boundary CLE~\cite{MSW2017} and hence the one arm exponent for the fuzzy Potts model. Finally, in forthcoming work of the first named author and Mithal, Theorems~\ref{thm:main} and~\ref{thm:main-nonsimple} are applied to construct a quantum natural time version of the \emph{quantum Loewner evolution (QLE)} introduced in \cite{qle}, and establish that this variant of QLE has three different phases of behavior analogous to the simple, swallowing, and space-filling regimes of SLE. The analogous result on phases for the original QLE remains open \cite[Section 8.2]{qle}.
	
	\medskip \noindent \textbf{Outline.}
	In Section~\ref{sec:radial-ig}, we give a background for SLE and the imaginary geometry, and establish the theory in the radial setting. In Section~\ref{sec:pre-LQG}, we provide background on LCFT and LQG, and prove Theorem~\ref{thm:main} based on Theorem~\ref{thm:w2w} in Section~\ref{subsec:weld-qt}. We prove Theorem~\ref{thm:w2w} in Section~\ref{sec:pf-thick} for $W>\frac{\gamma^2}{2}$ and Section~\ref{sec:pf-thin} for $W\in(0,\frac{\gamma^2}{2})$.
	Finally, we prove Theorem~\ref{thm:main-nonsimple} in Section~\ref{sec-nonsimple} by adapting the arguments of Sections~\ref{sec:pf-thick} and~\ref{sec:pf-thin}.

    \medskip 
    \noindent \textbf{Acknowledgments.} We thank Xin Sun for helpful discussions.  We also thank the anonymous referees for their careful reading and many valuable comments. M.A. was partially supported by the Simons Foundation as a Junior Fellow at the Simons Society of Fellows, NSF grants DMS-1712862 and DMS-2348201, and a start-up grant from UC San Diego.  P.Y. was partially supported by NSF grant DMS-1712862.

	\section{Radial SLE and imaginary geometry}\label{sec:radial-ig}

	In this section, we review the imaginary geometry theory~\cite{MS16a}, and develop the coupling between radial $\SLE_\kappa(\underline\rho)$ and the Gaussian free field. The single force point case is treated in~\cite{ig4}; we extend the results to the multiple force point regime.

	Throughout this section, we set 
	\begin{equation}
	    \kappa\in(0,4), \ \ \lambda = \frac{\pi}{\sqrt{\kappa}}, \ \ \kappa' = \frac{16}{\kappa}, \ \ \lambda' = \frac{\pi}{\sqrt{\kappa'}}, \ \ \chi = \frac{2}{\sqrt{\kappa}}-\frac{\sqrt{\kappa}}{2}.
	\end{equation}
	In this paper we work with non-probability measures and extend the terminology of ordinary probability to this setting. For a finite or $\sigma$-finite  measure space $(\Omega, \mathcal{F}, M)$, we say $X$ is a random variable if $X$ is an $\mathcal{F}$-measurable function with its \textit{law} defined via the push-forward measure $M_X=X_*M$. In this case, we say $X$ is \textit{sampled} from $M_X$ and write $M_X[f]$ for $\int f(x)M_X(dx)$. \textit{Weighting} the law of $X$ by $f(X)$ corresponds to working with the measure $d\tilde{M}_X$ with Radon-Nikodym derivative $\frac{d\tilde{M}_X}{dM_X} = f$.  \textit{Conditioning} on some event $E\in\mathcal{F}$ (with $0<M[E]<\infty$) refers to the probability measure $\frac{M[E\cap \cdot]}{M[E]} $  over the space $(E, \mathcal{F}_E)$ with $\mathcal{F}_E = \{A\cap E: A\in\mathcal{F}\}$,   {while \emph{restricting} to $E$ refers to the measure $M[E\cap\cdot]$. }
	
	\subsection{The Gaussian free field}\label{subsec:GFF}
	
	Let $D\subset \mathbb{C}$ be a domain with $\partial D =  \partial^D\cup\partial^F$, $\partial^D\cap\partial^F=\emptyset$. We construct the GFF on $D$ with \textit{Dirichlet} \textit{boundary conditions} on $\partial^D$ and \textit{free boundary conditions} on $\partial^F$ as follows. Suppose first that $\partial ^D$ is harmonically nontrivial. Consider the space of smooth functions on $D$ with finite Dirichlet energy and zero value near $\partial^D$, and let $H(D)$ be its closure with respect to the inner product $(f,g)_\nabla=(2\pi)^{-1}\int_D(\nabla f\cdot\nabla g)\ dx\ dy$. Then our GFF is defined by
	\begin{equation}\label{eqn-def-gff}
	h = \sum_{n=1}^\infty \xi_nf_n
	\end{equation}
	where $(\xi_n)_{n\ge 1}$ is a collection of i.i.d. standard Gaussians and $(f_n)_{n\ge 1}$ is an orthonormal basis of $H(D)$. One can show that the sum \eqref{eqn-def-gff} a.s.\ converges to a random distribution independent of the choice of the basis $(f_n)_{n\ge 1}$. If $\partial^F$ is harmonically trivial, then we call $h$ a Dirichlet GFF on $D$ with zero boundary value; if $h = h^0+f$ where $h^0$ is a  Dirichlet GFF on $D$ with zero boundary value and $f$ is harmonic, then we say $h$ is a Dirichlet GFF on $D$ with boundary condition $f|_{\partial D}$. 
    
    Now suppose $\partial^D$ is harmonically trivial. We consider the space of smooth functions modulo additive constant, let $H(D)$ be its closure with respect to $(\cdot, \cdot)_\nabla$, and define the GFF $h$ via~\eqref{eqn-def-gff}. Thus $H(D)$ is a space of functions modulo additive constant, and $h$ is a distribution modulo additive constant. We now specify a way to fix the additive constant for two cases, so $h$ is a distribution. 
    If $D = \mathcal{S}$ is the horizontal strip $\mathbb{R}\times (0, \pi)$ and $\partial^D = \emptyset$, we fix the constant by requiring that every function in $H(\mathcal{S})$ has average zero on $\{0\}\times [0, i\pi]$. If $D=\mathbb{H}$ is the upper half plane $\{z:\text{Im} z>0\}$  and $\partial^D = \emptyset$, every function in $H(\mathbb{H})$ should have average zero on the semicircle $\{e^{i\theta}:\theta\in(0, \pi)\}$. 
	We denote the corresponding laws of $h$ by $P_{\mathcal{S}}$ and $P_{\mathbb{H}}$, and the samples from $P_{\mathcal{S}}$ and $P_\bbH$ are referred as $h_{\mathcal{S}}$ and $h_{\bbH}$. We call these the free boundary GFFs on $\cS$ and $\bbH$.
	
	For $D\in\{\bbH,\cS\}$, the covariance of the GFF with free boundary conditions and the above normalization is 
	$$\bbE[ h_D(z)h_D(w)] = G_{D}(z,w),$$ where $G_{D}$ is the Green's function 
	\eqb
	\begin{aligned}
	    &G_{\bbH}(z,w) = -\log|z-w|-\log|z-\bar{w}|+2\log|z|_++2\log|w|_+; \ \ \quad G_{\cS}(z,w) = G_{\bbH}(e^z,e^w); 
	\end{aligned}
	\eqe
    here we write $\log r_+ := \max (\log r, 0)$.

	We now state the Markov property of the GFF.
	\begin{proposition}
	\label{prop:Markov}
		Let $D\subset \mathbb{C}$ be a domain with $\partial D =  \partial^D\cup\partial^F$, $\partial^D\cap\partial^F=\emptyset$, and $U\subset D$ open. Let $h$ be the GFF on $D$ with  \textit{Dirichlet} (resp. free) \textit{boundary conditions} on $\partial^D$ (resp. $\partial^F$). Then we can write $h=h_1+h_2$ where: 
		\begin{enumerate}
			\item $h_1$ and $h_2$ are independent;
			\item $h_1$ is a GFF on $U$ with Dirichlet boundary condition on $\partial U\backslash \partial^F$ and free on $\partial U\cap \partial^F$;
			\item $h_2$ is the same as $h$ outside $U$ and harmonic inside $U$.
		\end{enumerate}
	\end{proposition}
	Note that if $\partial^D=\emptyset$ (i.e., $h$ is free) then $h$ and $h_2$ are both defined modulo additive constants. See \cite[Section 4.1.5]{DMS14} for more details.  {The above property can also be extended to random sets. Following \cite{SS13}, we say that a (random) closed set $A\subset D$ coupled with $h$ and satisfying $\partial D \subset A$ is \textit{local}, if one can find a law on pairs $(A, h_2)$ such that $h_2|_{D\backslash A}$ is harmonic, and given $(A, h_2)$, the conditional law of $h -h_2$ is that of the zero boundary GFF on $D\backslash A$.

	\subsection{Overview of chordal $\SLE_\kappa(\underline\rho)$ and imaginary geometry}\label{subsec:pre-sle-ig}
	The SLE$_\kappa$ curves are introduced in \cite{Sch00}. For a curve $\eta$ on the upper half plane starting from 0,  let $H_t$ be the unbounded connected component of $\bbH\backslash \eta([0,t])$, and we call $K_t:=\bbH\backslash H_t$ the \emph{hull} of $\eta$ at time $t$. For $\kappa>0$, $\SLE_\kappa$ is a conformally invariant measure on continuously growing   curves $\eta$   with the Loewner driving function $W_t=\sqrt{\kappa}B_t$ (where $B_t$ is the standard Brownian motion).   Then  the $\SLE_\kappa$ curve from 0 to $\infty$ on the upper half plane  can be described by 
	\begin{equation}\label{eqn-def-sle}
	g_t(z) = z+\int_0^t \frac{2}{g_s(z)-W_s}ds, \ z\in\mathbb{H},
	\end{equation} 
	and $g_t$ is the unique conformal transformation from $H_t$ to $\mathbb{H}$ such that $\lim_{|z|\to\infty}|g_t(z)-z|=0$. $\SLE_\kappa$ is scale invariant, and hence its definition can be extended to other simply connected domains via conformal maps. 
	
	SLE$_\kappa$ curves also has a natural variant called SLE$_\kappa(\underline{\rho})$, which first appeared in \cite{LSW03} and studied in \cite{Dub05,MS16a}. 
	Fix $\rho_1,...,\rho_n\in\bbR$ and $x_1,...,x_n\in\ol\bbH$. The  {chordal} SLE$_\kappa(\underline{\rho})$ process is a measure on continuously growing curves $\eta$ with the  Loewner driving function $(W_t)_{t\ge 0}$  characterized by 
	\begin{equation}\label{eqn-def-sle-rho}
	\begin{split}
	&W_t = \sqrt{\kappa}B_t+\sum_{j=1}^n \int_0^t \mathrm{Re}\big(\frac{\rho_j}{W_s-V_s^j}\big)ds; \\
	& g_t(z) = z+\int_0^t \frac{2}{g_s(z)-W_s}ds, \ \text{for}\ z\in\ol{\bbH},\\
    & V_t^j = x_j+\int_0^t \frac{2}{V_s^j-W_s}ds, \ \ \text{for } j=1,...,n.
	\end{split}
	\end{equation}
	where $g_t$ is the unique conformal transformation from $H_t$ to $\mathbb{H}$ such that $\lim_{|z|\to\infty}|g_t(z)-z|=0$. The points $x_j$ with $x_j\in\bbH$ are referred as interior force points, while the points $x_j$ with $x_j\in\bbR$ are referred as boundary force points. The number $\rho_j$ is referred to as the weight of the force point $x_j$. The \emph{continuation threshold} for~\eqref{eqn-def-sle-rho} is the first time $t$ when  the sum of the weights of the force points which are immediately to the left of $W_t$ is not more than $-2$ or the sum of the weights of the force points which are immediately to the right of $W_t$ is not more than $-2$.  As explained in~\cite[Section 2]{MS16a} and~\cite[Section 3.3.1]{DMS14}, ~\eqref{eqn-def-sle-rho} admits a unique solution until the first time that either the continuation threshold is hit or the imaginary part of one of the interior force points is equal to 0. When all of the force points  lie on $\bbR$, we write 
	$x^{k,L}<...<x^{1,L}\le 0^-  \le 0^+ \le  x^{1,R}<...<x^{\ell, R}$ for the rearrangement of $x_1,...,x_n$ and write $\rho^{i,q}$ (where $q\in\{L,R\}$) for the  weight associated to $x^{i,q}$.
	
	 Now we recall the notion of \emph{the  GFF flow lines}. Heuristically, given a GFF $h$, $\eta(t)$ is a flow line of angle $\theta$ if
	\begin{equation}
	{\partial_t\eta(t)} = e^{i(\frac{h(\eta(t))}{\chi}+\theta)}\ \text{for}\ t>0.
	\end{equation} 
	Rigorously, the GFF flow lines are defined via the following result from~\cite[Theorem 1.1, Theorem 1.2]{MS16a}. Also recall the notion of GFF local sets after Proposition~\ref{prop:Markov}.
	
	\begin{theorem}\label{thm-ig1}
	 	Fix $\kappa>0$, a vector $\underline{\rho}$ of weights and a vector $\underline{x}$ of force points. Let $(K_t)_{t\ge 0}$ be the hull at time $t$ of the SLE$_\kappa(\underline{\rho})$ process $\eta$ described by the Loewner flow \eqref{eqn-def-sle} with $(W_t, V_t^{i,q})$ solving \eqref{eqn-def-sle-rho}. Let $\mathfrak{h}_t^0$ be the harmonic function on $\mathbb{H}$ with boundary values 
	 	$$
	 	-\lambda(1+\sum_{i=0}^j \rho^{i,L})\ \ \text{on} \ [V_t^{j+1, L}-W_t, V_t^{j,L}-W_t)\ \  \text{and}\ \ \lambda(1+\sum_{i=0}^j \rho^{i,R})\ \text{on}\ \ [V_t^{j, R}-W_t, V_t^{j+1,R}-W_t)
	 	$$
	 	where $\rho^{0,R} = \rho^{0,L}=0$, $x^{0,L} = 0^-, x^{0,R} = 0^+, x^{k+1, L} = -\infty, x^{\ell+1, R} = +\infty$, $V_t^{0,L} = g_t(0^-)$ and $V_t^{0,R} = g_t(0^+)$. Set $\mathfrak{h}_t(z) = \mathfrak{h}_t^0(g_t(z))-\chi\arg g_t'(z)$. Let $\mathcal{F}_t$ be the filtration generated by $(W, V^{i,q})$. Then there exists a coupling $(K,h)$ where $h = \tilde{h}+ \mathfrak{h}_0$ with $\tilde{h}$ being a zero boundary GFF on $\mathbb{H}$ such that the following is true. For any $\mathcal{F}_t$-stopping time $\tau$ before the continuation threshold, $K_\tau$ is a local set for $h$ and the conditional law of $h|_{\mathbb{H}\backslash K_\tau}$ given $\mathcal{F}_\tau$ is the same as the law of $\mathfrak{h}_\tau+\tilde{h}\circ g_\tau$. Moreover, the curve $\eta$ is measurable with respect to $h$.
	 \end{theorem}
	 Following~\cite{MS16a}, a chordal $\SLE_\kappa(\underline\rho)$ curve $\eta$ with $\kappa\in(0,4)$ coupled with the GFF $h$ as above is called a \emph{flow line} of $h$, while a chordal $\SLE_{\kappa'}(\underline\rho)$ curve $\eta$ with $\kappa'>4$ coupled with the GFF $-h$ is called a \emph{counterflow line} of $h$.
	 One important consequence  is that, as argued in \cite[Section 6]{MS16a}, given $\eta$, the field $h$ in each component of $\bbH\backslash\eta$ is a Dirichlet GFF with the \emph{flow line boundary conditions}; see e.g.\ \cite[Figure 1.10,1.11]{MS16a}.
	
	 We can also extend the notion of GFF flow lines to other simply connected domains as explained in~\cite{MS16a}. For a conformal map $f:D\to\bbH$ and a GFF $h$ on $\bbH$, if $\eta$ is a flow line of $h$, then $f^{-1}\circ \eta$ is a flow line of $h\circ f - \chi\arg f'$. Thus we define an imaginary surface to be an equivalence class of pairs $(D,h)$ under the equivalence relation
	 \begin{equation}\label{eq:ig-change-coord}
	     (D,h)\to (f^{-1}(D), h\circ f - \chi\arg f') = (\tilde D, \tilde h)
	 \end{equation}
	where $f:\tilde D\to D$ is a conformal map.
	
	\subsection{Radial $\SLE_\kappa(\underline\rho)$ and imaginary geometry}\label{subsec:radial-ig}
	
	We begin with the radial $\SLE_\kappa$ processes. For a curve $\eta$ from 1 targeted at 0 in $\bbD$, we write $D_t$ for the connected component of $\bbD\backslash\eta([0,t])$ containing 0. 
	Let
	$$\Psi(u,z)=\frac{u+z}{u-z},\ \Phi(u,z) = z\Psi(u,z),\ \text{and}\ \hat{\Phi}(u,z) = \frac{\Phi(u,z)+\Phi(1/\bar{u},z)}{2}   .$$
	For $\kappa>0$, the radial $\SLE_\kappa$ curve $\eta$ in $\bbD$ from 1 to 0 is defined by 
	\eqb\label{eq-rad-sle}
	dg_t(z) = \Phi(U_t, g_t(z))\,dt\quad \text{ for } z\in\mathbb{D}_t,  
	\eqe
	where $U_t = \exp(i\sqrt\kappa B_t)$ and $g_t$ is the unique conformal transformation $D_t\to \mathbb{D}$ fixing 0 with $g_t'(0)>0$ and $\log g_t'(0)=t$. 
	
	For $\rho_1,...,\rho_n\in\bbR$, $x_1,..., x_n\in\partial{\mathbb{D}}$ and $\mu\in\bbR$, the radial $\SLE_\kappa^\mu(\underline{\rho})$ in $\mathbb{D}$ targeted at 0 is the curve $\eta$ in $\mathbb{D}$ characterized by a random family of conformal maps $(g_t)$ solving 
	\begin{equation}\label{eqn:def-radial-sle}
	    \begin{split}
	        &dU_t = (i\kappa\mu-\frac{\kappa}{2})U_tdt+i\sqrt{\kappa}U_tdB_t+\sum_{j=1}^n\frac{\rho_j}{2}\hat{\Phi}(V^j_t,U_t)dt;\\
            & V_t^j = x_j+\int_0^t \Phi(U_t, g_s(z))ds\ \text{  for } j=1,...,n; \\
	        &dg_t(z) = \Phi(U_t, g_t(z))dt\ \text{for}\ z\in\ol{\mathbb{D}}.
	    \end{split}
	\end{equation}
	Again $g_t$ is the unique conformal transformation $D_t\to \mathbb{D}$ fixing 0 with $g_t'(0)>0$ and $\log g_t'(0)=t$. If $n=2$, $x_1=e^{i0^-}$ and $x_2=e^{i0^+}$, then we write $\SLE_\kappa^\mu(\rho_1;\rho_2)$ for the corresponding process.  If $\mu=0$, then the process is referred to as radial $\SLE_\kappa(\underline{\rho})$.  The definition of radial $\SLE_\kappa^\mu(\underline{\rho})$ can easily be extended to other simply connected domains via conformal maps.  When $\mu = 0$, by invoking~\cite[Theorem 3]{SW05} along with the results on chordal $\SLE_\kappa(\underline\rho)$ with a single interior force point, one can see that~\eqref{eqn:def-radial-sle} admits a unique solution and the  radial $\SLE_\kappa(\underline{\rho})$ process is well-defined up until time $\tau$ when the continuation threshold is hit or the curve $\eta$ disconnects 0 from the boundary $\partial\bbD$. For the latter case, one can continue the curve as a radial $\SLE_\kappa(\rho_1+...+\rho_n)$ in $D_\tau$, with the single force point lying immediately to the left (resp.\ right) of $\eta(\tau)$ if $\eta([0,\tau])$ closes a clockwise (resp.\ counterclockwise) loop around 0. Existence and uniqueness of solutions for the $\mu\neq 0$ case follows by noting that, by the Girsanov theorem, solutions to~\eqref{eqn:def-radial-sle}  correspond to those of the $\mu=0$ case reweighted by $\exp(\mu\sqrt\kappa B_t - \mu^2\kappa t/2)$. When $\mu > 0$ (resp.\ $\mu < 0$) the curve tends to spiral in a counterclockwise (resp.\ clockwise) direction as it approaches $0$.
	
	\begin{figure}
	    \centering
	    \begin{tabular}{cc}
	      \includegraphics[scale=0.6]{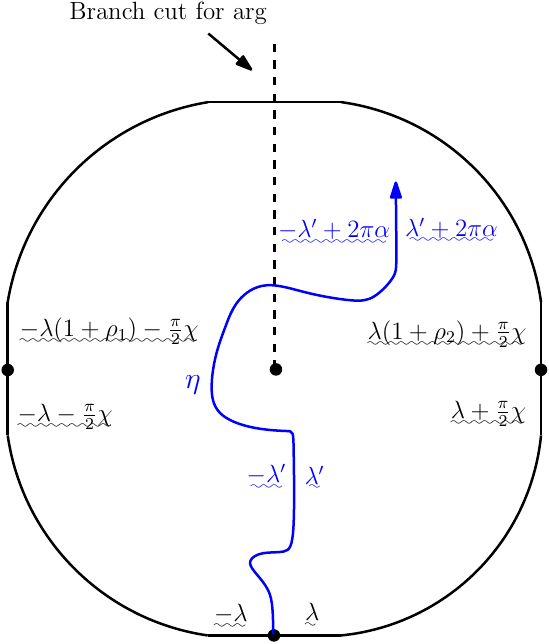}   &  \includegraphics[scale=0.6]{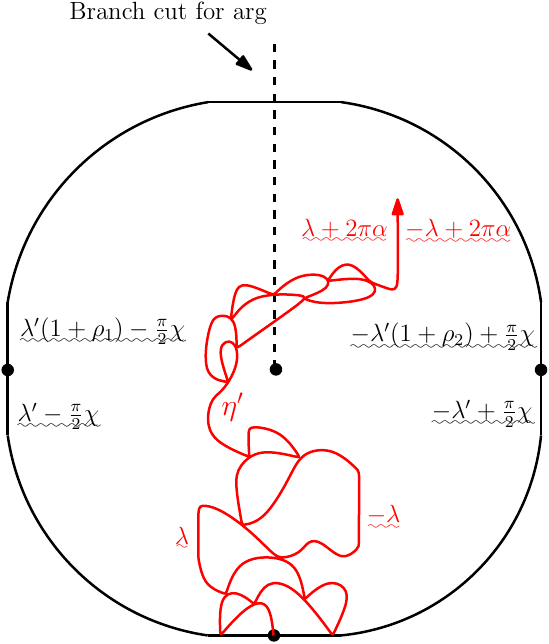}
	    \end{tabular}
	    
	    \caption{The $-\alpha$-flow line boundary conditions in Theorem~\ref{thm:radial-ig} for $n=2$.}
	    \label{fig:radialboundary}
	\end{figure}
	
	Before stating the coupling between radial $\SLE_\kappa(\underline\rho)$ with the GFF, we briefly recall the notion of $-\alpha$-flow line boundary condition as in~\cite[Figure 1.10]{ig4}. Choose some branch cut for the argument function. Given a non-crossing curve $\eta$ in $\bbD$ targeted at 0 stopped at some $\tau$, let $f$ be the harmonic function on $\bbD\backslash\eta([0,\tau])$ recording $\chi$ times the winding of $\eta$ as in the chordal case in~\cite[Figure 1.9, 1.10]{MS16a}, except that when $\eta$ crosses the branch cut clockwise (resp.\ counterclockwise) the boundary data on $\eta$ jumps $2\pi\alpha$ (resp.\ $-2\pi\alpha$). We say that a GFF $h$ on $\bbD\backslash\eta([0,\tau])$ has $-\alpha$-flow line boundary conditions along $\eta([0,\tau])$, if the boundary data of $h$ agrees with $f$ on $\eta([0,\tau])$ up to a global additive constant in $2\pi(\chi-\alpha)\mathbb{Z}$. We also say $h$ has $-\alpha$-flow line boundary conditions with angle $\theta$ along $\eta([0,\tau])$ if $h+\theta\chi$ has $-\alpha$-flow line boundary conditions   along $\eta([0,\tau])$.
	
	For a number of force points $x_1,...,x_n\in\partial\bbD$, we assume they are aligned in a clockwise way when started from $-i$, with $x_1$ (resp.\ $x_n$) possibly being $(-i)^-$ (resp.\ $(-i)^+$). If $\sum_{i=j}^n\rho_i>-2$ and $\sum_{i=1}^j\rho_i>-2$ for every $j=1,...,n$, then the continuation threshold of the radial $\SLE_\kappa^\mu(\underline\rho)$ is never hit and the corresponding curve continues all the way to the origin.  
	We set
	\begin{equation}\label{eq:radial-alpha}
	    \alpha = \frac{6-\kappa+\sum_{j=1}^n\rho_j}{2\sqrt{\kappa}}\ \  \text{for} \ \ \kappa\in(0,4)\ \ \text{and}\ \ \alpha = \frac{\kappa'-6-\sum_{j=1}^n\rho_j}{2\sqrt{\kappa'}}\ \ \text{for} \ \ \kappa'>4.
	\end{equation}
	 Fix a branch cut for $\arg(\cdot)$ (e.g., let the branch cut be the ray from 0 to $i\infty$). Let $h$ be a GFF on $\bbD$ whose boundary conditions are described as follows. For $\kappa\in(0,4)$, the boundary data of the field $h+\alpha\arg(\cdot)$ is equal to $-\lambda$ minus $\chi$ times the winding along the clockwise arc started from $-i$, except there is a jump of  $-\lambda \rho_j$   when passing the force point $x_j$ for each $j$ and a jump of $2\pi\alpha$ when passing the branch cut.  For $\kappa'>4$, the boundary data of the field $h+\alpha\arg(\cdot)$ is equal to $\lambda'$ minus $\chi$ times the winding along the clockwise arc started from $-i$, except there is a jump of  $\lambda' \rho_j$   when passing the force point $x_j$ for each $j$ and a jump of $2\pi\alpha$ when passing the branch cut.  See Figure~\ref{fig:radialboundary} for the case when $n = 2$.
	
	The aim of this section is to prove the following theorem, which is an extension of~\cite[Proposition 3.1 and 3.18]{ig4} to the multiple force point case. 
	\begin{theorem}\label{thm:radial-ig}
	    Let $\beta,\rho_1,...,\rho_n\in\bbR$, define $\alpha$ as in~\eqref{eq:radial-alpha},  and let $h$ be the GFF on $\bbD$ described  above. Then there exists a unique coupling between $h+\alpha\arg(\cdot)+\beta\log|\cdot|$ and a radial $\SLE_\kappa^\beta(\underline\rho)$ process in $\bbD$ from $-i$ targeted at 0 with force points $x_1,...,x_n$ such that the following holds. For every $\eta$-stopping time $\tau$, the conditional law of $h+\alpha\arg(\cdot)+\beta\log|\cdot|$ given $\eta|_{[0,\tau]}$ is that of $\tilde h+\alpha\arg(\cdot)+\beta\log|\cdot|$ where $\tilde h$ is a GFF on $\bbD\backslash\eta([0,\tau])$ such that $\tilde h+\alpha\arg(\cdot)+\beta\log|\cdot|$ has the same boundary conditions as $h+\alpha\arg(\cdot)+\beta\log|\cdot|$ on $\partial \mathbb D$. If $\kappa\in(0,4)$, then $\tilde h+\alpha\arg(\cdot)+\beta\log|\cdot|$ has $-\alpha$-flow line boundary conditions on $\eta([0,\tau])$. If $\kappa'>4$, then $\tilde h+\alpha\arg(\cdot)+\beta\log|\cdot|$ has $-\alpha$-flow line boundary conditions with angle $\frac{\pi}{2}$ (resp.\ $-\frac{\pi}{2}$) on the left (resp.\ right) side of $\eta([0,\tau])$. 
	
	    In this coupling, $\eta([0,\tau])$ is a local set of $h$, and $\eta$ is measurable with respect to $h$.
	\end{theorem}
	\begin{definition}\label{def:radial-flow-line}
	    For $\kappa\in(0,4)$, we call a radial $\SLE_\kappa^\beta(\underline\rho)$ curve $\eta$ coupled with the GFF $h_{\alpha\beta}:=h+\alpha\arg(\cdot)+\beta\log|\cdot|$ as in Theorem~\ref{thm:radial-ig} the \emph{flow line} of $h_{\alpha\beta}$. For $\theta\in\bbR$, we call $\eta$ \emph{the flow line of} $h_{\alpha\beta}$ \emph{with angle $\theta$} if $\eta$ is the flow line of  $h_{\alpha\beta}+\chi\theta$. For $\kappa'>4$, we call a  $\SLE_{\kappa'}^\beta(\underline\rho)$ curve $\eta'$ coupled with $h_{\alpha\beta}$ as in Theorem~\ref{thm:radial-ig} the \emph{counterflow line} of $h_{\alpha\beta}$. 
	\end{definition}
	
	Using~\eqref{eq:ig-change-coord}, one can also define the flow lines of $h_{\alpha\beta}$ started at other points on $\partial \bbD$. We will not need to consider flow lines started at $0$ in this paper, but these can be defined similarly as in~\cite[Theorem 1.4]{ig4}. 
	
	Theorem~\ref{thm:radial-ig} is essentially proved in~\cite[Proposition 3.1]{ig4}, except that it is not proved that the flow line $\eta$ is measurable with respect to $h$. Here we present a different proof based on coordinate change from Theorem~\ref{thm-ig1} and the Girsanov theorem. In particular, we will use the following absolute continuity for GFF, where one may see  e.g.~\cite[Proposition 3.4]{MS16a} for a proof. Recall that for a Dirichlet GFF $h$ in $D$ and a function $g\in H(D)$, one can define $(h,g)_\nabla = -\frac{1}{2\pi}(h,\Delta g)$ (see e.g.~\cite[Section 3.1]{MS16a}).
	\begin{lemma}\label{lem:Girsanov-Dirichlet}
	    Let $h$ be a Dirichlet GFF $h$ in $D$ and  $g\in H(D)$. If we weight the law of $h$ by $\exp((h,g)_{\nabla} - \frac{1}{2}(g,g)_\nabla)$, then the law of $h$ under the reweighted measure agrees with the law of $h+g$ under the unreweighted measure.
	\end{lemma}
	Now let $\eta$ be a chordal $\SLE_\kappa(\underline\rho)$ curve from 0 to $\infty$ as in Theorem~\ref{thm-ig1}, and let $f_t$ be its centered Loewner flow. For $z,w\in\ol \bbH$, we set $G(z,w) = -\log|z-w|+\log|z-\ol w|$ and $G_t(z,w) = G(f_t(z),f_t(w))$. Let $\cL_i$ be the infinite half  line $\cL_i:=\{z: \mathrm{Re} z=0,\   \mathrm{Im} z\ge 1 \}$, and write $g(z) = \arg(z-i)+\arg(z+i)$. We choose the branch cut of the function $g$ to be $\cL_i$, i.e., $g$ is harmonic on $\bbH\backslash\cL_i$, taking value $\pi$ (resp.\ $-\pi$) on the right (resp.\ left) side of $\cL_i$, and 0 on the real line. We will need the following deterministic computation.
	\begin{lemma}\label{lem:GFF-integration}
	    Let $U\subset \bbH$ be a bounded simply connected domain such that there exists $a,b>0$ with $(-a,b)\subset\partial U$. Also assume $U$ is disjoint from $\cL_i$ and $\partial U\cap\bbH$ is smooth.
	    Let $t >0$ be any  time such that $\eta([0,t]) \subset \ol U$.
	    Suppose $\tilde g\in H(\bbH)$ is a smooth function whose support is bounded and disjoint from $\cL_i$ such that $\tilde g(z) = g(z)$ in $U$. Then for $w\in U$,
	    \begin{align}
	        &\int_\bbH \Delta\tilde g(z) G_t(z,w) dz = -2\pi(\arg(f_t(w)-f_t(i))+\arg(f_t(w)+\ol{f_t(i)}));
	         \label{eq:lem-GFF-integration-A} \\& \int_\bbH \Delta\tilde g(z) \mathrm{Im}(\frac{2}{f_t(z)}) dz = -2\pi\mathrm{Re}(\frac{2}{f_t(i)}).\label{eq:lem-GFF-integration-B} 
	    \end{align}
	\end{lemma}
	
	\begin{proof}
	 We only prove~\eqref{eq:lem-GFF-integration-A}; ~\eqref{eq:lem-GFF-integration-B} follows similarly. Note that $\Delta \tilde g = 0$ within $\partial U$. For $\e,R>0$, we write $C_{\e,R} = \{z\in \bbH:|z|<R, \ \dist(z,\cL_i)>\e\}$, and we choose $\e,R$ such that $\supp(\tilde g)\subset C_{\e,R}\cup\bbR$. Since $G_t(z,w)$ is harmonic in $\bbH\backslash U$ and is zero on $\bbR$, it follows from Green's theorem that
	 \begin{equation}\label{eq:lem-GFF-integration-1}
	     \int_\bbH \Delta\tilde g(z) G_t(z,w) dz = \int_{\bbH\backslash U} \Delta\tilde g(z) G_t(z,w) dz = \int_{\partial U\cap\bbH} \big(G_t(z,w)\frac{\partial   g}{\partial \mathbf{n}}(z) -   g(z)\frac{\partial  G_t}{\partial \mathbf{n}}(z,w)\big)dz
	 \end{equation}
	    where in the last equation we used $\tilde g = g$ in $U$ and $\Delta G_t(z,w) = 0$ for $z\in\bbH\backslash U$. Then using the fact that $g$ is harmonic on $\bbH\backslash\cL_i$ and $g(z) = G_t(z,w) = 0$ for $z\in\bbR$, it follows that
	\begin{equation}\label{eq:lem-GFF-integration-2}
	    \int_{\partial U\cap\bbH} \big(G_t(z,w)\frac{\partial   g}{\partial \mathbf{n}}(z) -   g(z)\frac{\partial  G_t}{\partial \mathbf{n}}(z,w)\big)dz =  \int_{\partial C_{\e,R}\cap\bbH} \big(G_t(z,w)\frac{\partial   g}{\partial \mathbf{n}}(z) -   g(z)\frac{\partial  G_t}{\partial \mathbf{n}}(z,w)\big)dz 
	\end{equation}
	    where the normal vector $\mathbf{n}$ on the right hand side of~\eqref{eq:lem-GFF-integration-2} points inwards $C_{\e,R}$. Using the dominated convergence theorem, it is not hard to prove that 
	\begin{equation}\label{eq:lem-GFF-integration-3}
	    \lim_{\e\to0}\int_{\partial C_{\e,R}\cap\bbH} \big(G_t(z,w)\frac{\partial   g}{\partial \mathbf{n}}(z) -   g(z)\frac{\partial  G_t}{\partial \mathbf{n}}(z,w)\big)dz = \int_{\partial B_R(0)\cap\bbH} G_t(z,w)\frac{\partial   g}{\partial \mathbf{n}}(z)dz - \int_{\partial C_{0,R}\cap\bbH}   g(z)\frac{\partial  G_t}{\partial \mathbf{n}}(z,w)dz.
	\end{equation}
	Since $f_t(z)-W_t\sim z+\frac{C}{z}+o(\frac{1}{|z|^2})$ as $|z|\to\infty$, it is straightforward to check that $$|\nabla G_t(z,w)| = |\frac{2f_t'(z)\mathrm{Im}f_t(w)}{(f_t(z)-f_t(w))(f_t(z)-f_t(w))}|\le C|z|^{-2}, \ \ \text{as $|z|\to\infty$} $$  Furthermore, for sufficiently large $R$, $|\frac{\partial   g}{\partial \mathbf{n}}(z)|<2R^{-2}$ on $\{|z|=R\}$. Thus further sending $R\to\infty$, combining~\eqref{eq:lem-GFF-integration-1},\eqref{eq:lem-GFF-integration-2} and~\eqref{eq:lem-GFF-integration-3}, we see
	\begin{equation}\label{eq:lem-GFF-integration-4}
	    \int_\bbH \Delta\tilde g(z) G_t(z,w) dz = -2\pi\int_1^\infty\frac{\partial G_t}{\partial x}(iy,w)dy
	\end{equation}
	where we write $z = x+iy$. Now $G_t(z,w)$ is the real part of the analytic function $\tilde{G_t}(z,w):=\log(f_t(z)-f_t(w)) - \log(f_t(z)-\ol{f_t(w)})$ where we may choose a branch cut of the logarithm function disjoint from $\cL_i$. Thus it follows from the Cauchy-Riemann equation that $\frac{\partial G_t}{\partial x} = \frac{\partial \mathrm{Im} G_t}{\partial y}$, and therefore~\eqref{eq:lem-GFF-integration-A} follows from~\eqref{eq:lem-GFF-integration-4}.
	\end{proof}
	
	\begin{proof}[Proof of Theorem~\ref{thm:radial-ig}]
	   We only prove the $\kappa\in(0,4)$ case; the $\kappa'>4$ case follows analogously. We work on the chordal setting and first construct the coupling of the GFF with chordal $\SLE_\kappa(\rho_1,...,\rho_n;\kappa-6-\sum_{j=1}^n\rho_j)$ in $\bbH$ with a single interior force point of weight $\kappa-6-\sum_{j=1}^n\rho_j$ located at $i$. Consider the coupling between the GFF $h$ and a chordal $\SLE_\kappa(\rho_1,...,\rho_n)$ curve $\eta$, along with the harmonic function $\mathfrak{h}_t$ as described in Theorem~\ref{thm-ig1}. Let $g(z) = \arg(z-i)+\arg(z+i)$. Choose the branch cut $\cL_i$ and let the domain $U$ and $\tilde g\in H(\bbH)$ be as in Lemma~\ref{lem:GFF-integration}. Fix a stopping time $\tau$ such that $\eta([0,\tau])\subset \ol U$. Then it follows from Theorem~\ref{thm-ig1} that the conditional law of $h$ given $\cF_\tau$ is equal to  the law of an independent zero boundary GFF $h^0_\tau$ on $\bbH\backslash\eta([0,\tau])$ plus the harmonic function $\mathfrak{h}_\tau$. Moreover, since $\eta([0,\tau])$ is local, $\eta|_{[0,\tau]}$ is measurable with respect to $(h+\alpha\tilde g)|_{U} = (h+\alpha g)|_U$.
	
	   Now we weight the law of $(h,\eta)$ by $\exp((h,\alpha\tilde g)_\nabla-\frac{\alpha^2}{2}(\tilde g,\tilde g)_\nabla)$. Then by Lemma~\ref{lem:Girsanov-Dirichlet}, the law of $h$ under the weighted measure is equal to the law of $h+\alpha\tilde g$ under the unweighted measure. Furthermore, if we set $\mathcal{E}_t(\varphi) = \iint_{\bbH^2}G_t(z,w)\varphi(z)\varphi(w)$, then as in the proof of~\cite[Lemma 3.11]{MS16a}, $(h^0_t,\varphi)$ is a Gaussian with mean zero and variance $\mathcal{E}_t(\varphi)$ for $\varphi\in C_0^\infty(\bbH)$. Using $(h,\tilde g)_\nabla = -\frac{1}{2\pi}(h,\Delta\tilde g)$, it follows that
	   \begin{equation}\label{eq:pf-radial-ig-1}
	      M_t:= \bbE\bigg[\exp((h,\alpha\tilde g)_\nabla-\frac{\alpha^2}{2}(\tilde g,\tilde g)_\nabla)\bigg|\cF_t\bigg] = \exp\bigg(-\frac{\alpha}{2\pi}(\mathfrak h_t,\Delta \tilde g)+\frac{\alpha^2}{8\pi^2}\mathcal{E}_t(\Delta\tilde g)-\frac{\alpha^2}{2}(\tilde g,\tilde g)_\nabla\bigg),
	   \end{equation}
	   while the law of $\eta|_{[0,\tau]}$ is now weighted by $M_\tau$. By~\cite[Eq.\ (3.19)]{MS16a}, $M_t$ is a local martingale. Moreover, from~\cite[Eq. (3.15),(3.16)]{MS16a},
	   \begin{equation}\label{eq:pf-radial-ig-2}
	       d\mathfrak h_t(z) = \mathrm{Im}\big(\frac{2}{f_t(z)}\big)dB_t, \ \ dG_t(z,w) =  -\mathrm{Im}\big(\frac{2}{f_t(z)}\big) \mathrm{Im}\big(\frac{2}{f_t(w)}\big)dt.
	   \end{equation}
	   Thus we may compute the cross variation between $\log M_t$ and $B_t$:
	   \begin{equation}\label{eq:pf-radial-ig-3}
	       d\langle \log M, B\rangle_t = -\frac{\alpha}{2\pi}\int_\bbH\, \mathrm{Im}\big(\frac{2}{f_t(z)}\big)\Delta\tilde g(z)  \, dz\,dt = \mathrm{Re}\big(\frac{2\alpha}{f_t(i)}\big)\,dt,
	   \end{equation}
	  where in the last equation we applied~\eqref{eq:lem-GFF-integration-B}. Since $2\alpha\sqrt{\kappa} = 6-\kappa+\sum_{j=1}^n\rho_j$, by the Girsanov theorem and comparing with~\eqref{eqn-def-sle-rho}, the evolution of $\eta|_{[0,\tau]}$ under the weighted measure is described by chordal $\SLE_\kappa(\rho_1,...,\rho_n;\kappa-6-\sum_{j=1}^n\rho_j)$, while $\eta|_{[0,\tau]}$ is measurable with respect to $h|_{U}$. 
	
	  Next we compute the conditional law of $h|_U$ given $\cF_\tau$ under the weighted measure. Pick a function $\varphi\in C_0^\infty(U)$. Then we have
	  \begin{equation}\label{eq:pf-radial-ig-4}
	      \bbE\bigg[\exp((h,\varphi)+(h,\alpha\tilde g)_\nabla-\frac{\alpha^2}{2}(\tilde g,\tilde g)_\nabla)\bigg|\cF_\tau\bigg] = \exp\bigg((\mathfrak h_\tau,\varphi-\frac{\alpha}{2\pi}\Delta \tilde g)+\frac{1}{ 2}\mathcal{E}_\tau(\varphi-\frac{\alpha}{2\pi}\Delta\tilde g)-\frac{\alpha^2}{2}(\tilde g,\tilde g)_\nabla\bigg).
	  \end{equation}
	  On the other hand, 
	  $$\frac{1}{ 2}\mathcal{E}_\tau(\varphi-\frac{\alpha}{2\pi}\Delta\tilde g) = \frac{1}{2}\mathcal{E}_\tau(\varphi)+\frac{\alpha^2}{8\pi^2}\mathcal{E}_\tau(\Delta \tilde g) - \frac{\alpha}{2\pi}\iint_{\bbH^2}G_t(z,w)\Delta\tilde g(z)\varphi(w)\,dzdw$$
	  Therefore by~\eqref{eq:lem-GFF-integration-A}, the conditional law of $(h,\varphi)$ given $\cF_\tau$ under the weighted measure is the same as $\big(h+\alpha(\arg(f_t(\cdot)+f_t(i))+\arg(f_t(\cdot)-\ol{f_t(i)}), \varphi\big)$. Thus by varying the domain $U$, we have generated the coupling between $h+\alpha g$ with a chordal   $\SLE_\kappa(\rho_1,...,\rho_n;\kappa-6-\sum_{j=1}^n\rho_j)$ curve $\eta$ in $\bbH$ up until the time $\tau_0$ when $\eta$ hits the continuation threshold or the branch cut $\cL_i$. Under this coupling, for any stopping time $\tau\le\tau_0$ the conditional law of $h+\alpha g$ given $\cF_\tau$ is equal to the law of  of an independent zero boundary GFF $h^0_\tau$ on $\bbH\backslash\eta([0,\tau])$ plus the harmonic function $\mathfrak{h}_\tau+\alpha(\arg(f_t(\cdot)+f_t(i))+\arg(f_t(\cdot)-\ol{f_t(i)})$, and $\eta|_{[0,\tau]}$ is determined by $h+\alpha g$.
	  
	  If $\eta$ hits $\cL_i$ at the stopping time $\tau_0$, then we apply the conformal map $f_{\tau_0}$, and construct the flow line of $$h^0_\tau\circ f_{\tau_0}^{-1}+\mathfrak{h}_\tau\circ f_{\tau_0}^{-1}+\alpha(\arg(\cdot+f_{\tau_0}(i))+\arg(\cdot-\ol{f_{\tau_0}(i)}))-\chi\arg (f_{\tau_0}^{-1})' $$
	  as in the previous step. Here to generate the coupling we shift the branch cut such that it is the straight line $\cL_{f_{\tau_0}(i)}$ from $f_{\tau_0}(i)$ to $\infty$. This gives the flow line up until time $\tau_1$ when $\eta$ hits $f_{\tau_0}^{-1}\circ \cL_{f_{\tau_0}(i)}$. Then we shift the branch cut back to $\cL_i$, which creates a $2\pi\alpha$ difference between $\cL_i$ and $f_{\tau_0}^{-1}\circ \cL_{f_{\tau_0}(i)}$ and gives rise to the jump under the flow line boundary condition. One iterates this process, and as explained in the proof of~\cite[Proposition 3.1]{ig4}, there is some absolute constant $c\in(0,1)$ such that for each step we have $\mathrm{Im} f_{\tau_{j+1}}(i)\le c\mathrm{Im} f_{\tau_{j}}(i)$. This implies that we can generate the coupling before the continuation threshold is hit or $i$ is separated from $\infty$.
	  Then we take a conformal map $\varphi:\bbH\to\bbD$ to  map to the unit disk sending $(0,i)$ to $(-i,0)$ through~\eqref{eq:ig-change-coord} and apply~\cite[Theorem 3]{SW05}, where the image of $\eta$ is now radial $\SLE_\kappa(\rho_1,...,\rho_n)$. The flow line is grown until it separates $\varphi(\infty)$ and $0$, and we can pick a new reference point on the boundary of the connected component containing $0$ and continue this process. As explained at the end of the proof of~\cite[Proposition 3.1]{ig4}, the uniqueness of the coupling follows from the same argument as~\cite[Theorem 2.4]{MS16a}. This completes the proof when $\beta=0$.
	
	  If $\beta\neq0$, the proof follows from the same argument as in the proof of~\cite[Proposition 3.18]{ig4}. Namely, let $(\widehat{h}, \widehat{\eta})$ be the coupling for the $\beta=0$ case. We pick $\delta>0$ and set $\xi_\delta = \log\max(|z|,\delta)$, and weight the law of $(\widehat{h}, \widehat{\eta})$ by $\exp(\beta(\widehat{h},\xi_\delta)_\nabla)$. Then we apply Lemma~\ref{lem:Girsanov-Dirichlet} and adapt the previous argument for the $\beta=0$ case, and finally send $\delta\to0$. We omit the details.
	\end{proof}
	
	Finally we recall the whole plane $\SLE_\kappa^{\mu}(\rho)$ processes, which describes a random curve $\eta:\bbR\to\bbC$ as follows. Let $D_t$ be the unbounded connected component of $\bbC\backslash\eta((-\infty,t])$, and $g_t:D_t\to\bbC\backslash\bbD$ be the unique conformal map with $g_t(\infty)=\infty$ and $g_t'(\infty)>0$. As shown in~\cite[Proposition 2.1]{ig4}, for $\mu\in\bbR,\kappa>0,\rho>-2$ and a standard two-sided Brownian motion $(B_t)_{t\in\bbR}$, the following SDE
	\begin{equation}
	    \begin{split}
	        &dU_t = (i\kappa\mu-\frac{\kappa}{2})U_tdt+i\sqrt{\kappa}U_tdB_t+\frac{\rho}{2}\hat{\Phi}(O_t,U_t)dt\\
	        &dO_t = \Phi(U_t, O_t)dt
	        \end{split}
	\end{equation}
	has  a unique stationary solution. Then the whole plane $\SLE_\kappa^{\mu}(\rho)$ process is characterized by $(g_t)_{t\in\bbR}$ such that $g_t$ solves~\eqref{eq-rad-sle} for $z\in D_t$.
	The whole plane $\SLE_\kappa^{\mu}(\rho)$ processes can also be coupled with $h-\alpha\arg(\cdot)-\beta\log|\cdot|$ as flow lines. See~\cite[Section 2.1.3, Section 3.3]{ig4} for more details.
	
	\subsection{Interacting flow lines}\label{subsec:radial-ig-flow-line}
	In this section, we consider the interaction between flow lines of $h_{\alpha\beta}=h+\alpha\arg(\cdot)+\beta\log|\cdot|$. Suppose $\eta_1$ and $\eta_2$ are two flow lines of $h_{\alpha\beta}$ with angle $\theta_1$ and $\theta_2$. 
	Consider the event $E$ where $\eta_1$ hits $\eta_2$ on its right hand side; call the hitting time $\tau_1$.
	Then following~\cite[Theorem 1.7 and Figure 1.10]{ig4}, one can define a height difference $\cD_{12}$ between $\eta_1$ and $\eta_2$. Roughly speaking, using the $-\alpha$-flow line boundary conditions, $\cD_{12}$ is obtained by subtracting the boundary value on the right hand side of $\eta_2$ from  the boundary value on the right hand side of $\eta_1$ near $\eta_1(\tau_1)$. See Figure~\ref{fig:flowlineradialheight} for an illustration.
	
	\begin{figure}
	    \centering
	    \includegraphics[scale=0.75]{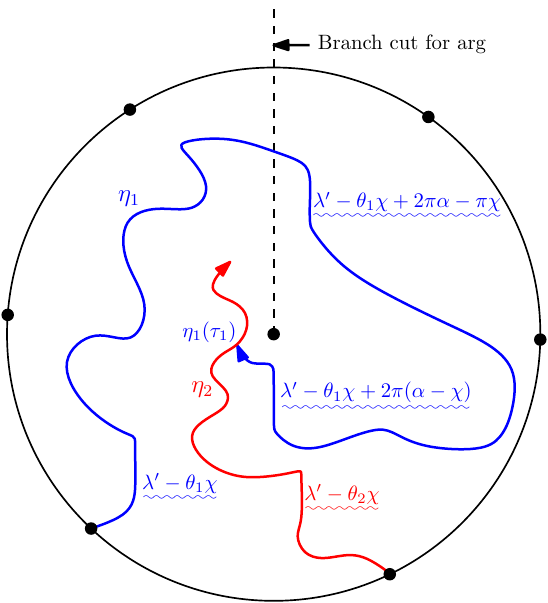}
	    \caption{An illustration of the height difference between the angle $\theta_1$ flow line $\eta_1$ and the angle $\theta_2$ flow line $\theta_2$. In this picture $\cD_{12} = 2\pi(\alpha-\chi)+(\theta_2-\theta_1)\chi$.}
	    \label{fig:flowlineradialheight}
	\end{figure}
	
	\begin{proposition}\label{prop:height-difference-radial}
	    Let $h_{\alpha\beta}$ be as in Theorem~\ref{thm:radial-ig}, and let $\eta_1$ and $\eta_2$ be two flow lines of $h_{\alpha\beta}$ with angle $\theta_1$ and $\theta_2$ started at $x_1,x_2\in\partial\bbD$. On the event $E$ where $\eta_1$ hits $\eta_2$ on its right side, 
	    we have $\cD_{12}\in(-\pi\chi,2\lambda-\pi\chi)$. Furthermore,
	    \begin{enumerate}[(i)]
	        \item If $\cD_{12}\in(-\pi\chi,0)$, then $\eta_1$ crosses $\eta_2$ upon intersecting and does not subsequently
	cross back;
	\item If $\cD_{12} = 0$, then $\eta_1$ merges with and does not subsequently separate from $\eta_2$ upon intersecting;
	\item If $\cD_{12}\in(0,2\lambda-\pi\chi)$, then $\eta_1$ bounces off but does not cross $\eta_2$.
	    \end{enumerate}
	    Finally, if $\eta_1$ hits $\eta_2$ on its left, then the same result holds with $\cD_{21} := -\cD_{12}$ replaced with $\cD_{12}$.
	\end{proposition}
	\begin{proof}
	The proof is identical to that of~\cite[Proposition 3.5]{ig4}. Alternatively, similarly to the proof of Theorem~\ref{thm:radial-ig}, we can use absolute continuity to deduce this from the analogous result in the chordal regime~\cite[Theorem 1.5]{MS16a}. We omit the details.
	\end{proof}

	The remainder of this section aims to prove the following.
	
	\begin{figure}[ht]
		\centering
	 \begin{tabular}{cc}
	     \includegraphics[scale=0.59]{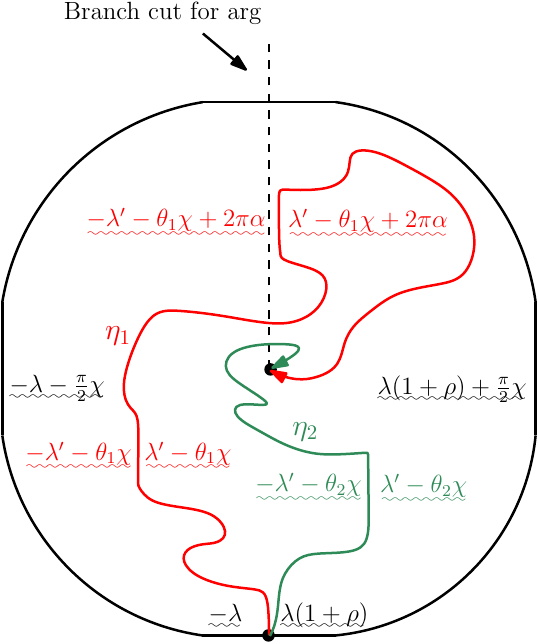}\ \ & \includegraphics[scale=0.59]{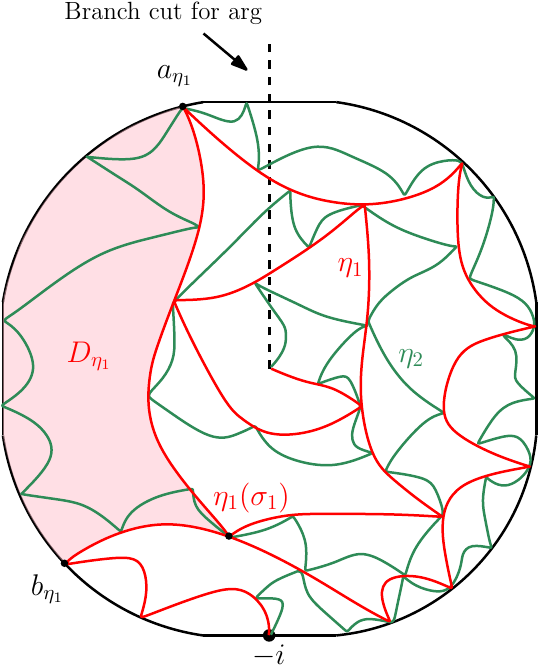}
	 \end{tabular}
		
		\caption{\textbf{Left:} Given the flow lines $\eta_1,\eta_2$ of angle $\theta_1$,$\theta_2$, $h_{\alpha\beta}$ has the illustrated  {flow line boundary conditions}, and one can read off the conditional law of $\eta_1$ given $\eta_2$ and the conditional law of $\eta_2$ given $\eta_1$ as in Proposition~\ref{prop:radial-flowline}. \textbf{Right:} The domain $D_{\eta_1}$ in the case where $\eta_1$ is self intersecting.}\label{fig:radial-gff}
	\end{figure}
	
	\begin{proposition}\label{prop:radial-flowline}
	Fix {$\kappa\in(0,4)$,}  $\alpha>\chi,\beta\in\bbR$, and suppose that $h_{\alpha\beta} = h +\alpha\arg(\cdot)+\beta\log|\cdot|$ has boundary conditions on $\partial\mathbb{D}$ as depicted in Figure~\ref{fig:radial-gff} where $h$ is a Dirichlet GFF on $\bbD$. Let 
	\begin{equation}\label{eq:flowline-angle}
	   2\pi(1-\frac{\alpha}{\chi})<\theta_2<\theta_1<\frac{2\lambda}{\chi}, \ \ \theta_1-\theta_2<2\pi(\frac{\alpha}{\chi}-1); \ \ \rho = \kappa-6+\frac{2\pi\alpha}{\lambda} . 
	\end{equation}
	Let $(\eta_1,\eta_2)$ be the flow lines of $h_{\alpha\beta}$ with angles $(\theta_1,\theta_2)$ from $-i$ to 0. Then:
	\begin{enumerate}[(i)]
	    \item For $i=1,2$, $\eta_i$ is a radial $\SLE_\kappa^\beta(-\frac{\theta_i\chi}{\lambda};\rho+\frac{\theta_i\chi}{\lambda})$ process;
	    \item $\eta_1$ and $\eta_2$ do not cross each other;
	    \item $\eta_1,\eta_2$ are continuous near 0;
	    \item  \label{item-resample}
	   {Conditioned on $\eta_1$, the curve $\eta_2$ is the union of independent curves in certain connected components of $\mathbb{D}\backslash\eta_1$, as follows. 
	   \begin{itemize}
	       \item If $\eta_1$ first touches itself at time $\sigma_1<\infty$, let $D_{\eta_1}$ be the component having $\eta_1(\sigma_1)$ on its boundary whose boundary contains a segment of $\partial \bbD$. Otherwise let $D_{\eta_1}$ be the component having $0$ on its boundary. 
	       Let $a_{\eta_1}$ and $b_{\eta_1}$ be the first and last points on $\partial D_{\eta_1}\cap\partial \bbD$ when one views $\partial \bbD$ as a counterclockwise curve from $(-i)^+$ to $(-i)^-$.  Then in $D_{\eta_1}$, $\eta_2$ is a chordal $\SLE_\kappa(\frac{(\theta_1-\theta_2)\chi}{\lambda}-2;\rho+\frac{\theta_2\chi}{\lambda},-\frac{\theta_1\chi}{\lambda})$ process from $a_{\eta_1}$ to $\eta(\sigma_1)$ with force points $a_{\eta_1}^-;a_{\eta_1}^+,b_{\eta_1}$.
	
	       \item In each remaining component whose boundary contains a segment of the counterclockwise boundary arc of $\partial\bbD$ from $-i$ to $a_{\eta_1}$, $\eta_2$ is a chordal $\SLE_\kappa(\frac{(\theta_1-\theta_2)\chi}{\lambda}-2;\rho+\frac{\theta_2\chi}{\lambda})$ process from the first to the last boundary point hit by $\eta_1$.
	
	       \item In each remaining component whose boundary does not contain a segment of $\partial \bbD$, $\eta_2$ is a chordal $\SLE_\kappa(\frac{(\theta_1-\theta_2)\chi}{\lambda}-2;\rho+\frac{(\theta_2-\theta_1)\chi}{\lambda})$ process from the first to the last boundary point hit by $\eta_1$.
	   \end{itemize}
	   Similarly, given $\eta_2$, defining $D_{\eta_2}, a_{\eta_2}, b_{\eta_2}$ as above with $\eta_1$ replaced by $\eta_2$ and counterclockwise replaced by clockwise, the curve $\eta_1$ is $\SLE_\kappa(-\frac{\theta_1\chi}{\lambda},\rho+\frac{\theta_2\chi}{\lambda};\frac{(\theta_1-\theta_2)\chi}{\lambda}-2)$ in  $D_{\eta_2}$, $\SLE_\kappa(-\frac{\theta_1\chi}{\lambda};\frac{(\theta_1-\theta_2)\chi}{\lambda}-2)$ in each other connected component of $\bbD\backslash\eta_2$ whose boundary contains a segment of the  clockwise boundary arc of $\partial\bbD$ from $(-i)^-$ to $a_{\eta_2}$, and $\SLE_\kappa(\rho+\frac{(\theta_2-\theta_1)\chi}{\lambda};\frac{(\theta_1-\theta_2)\chi}{\lambda}-2)$ in each remaining component whose boundary does not contain an interval of $\partial \bbD$. }
	\end{enumerate}
	\end{proposition}
	
	\begin{figure}[ht]
		\centering
		\includegraphics[scale=0.59]{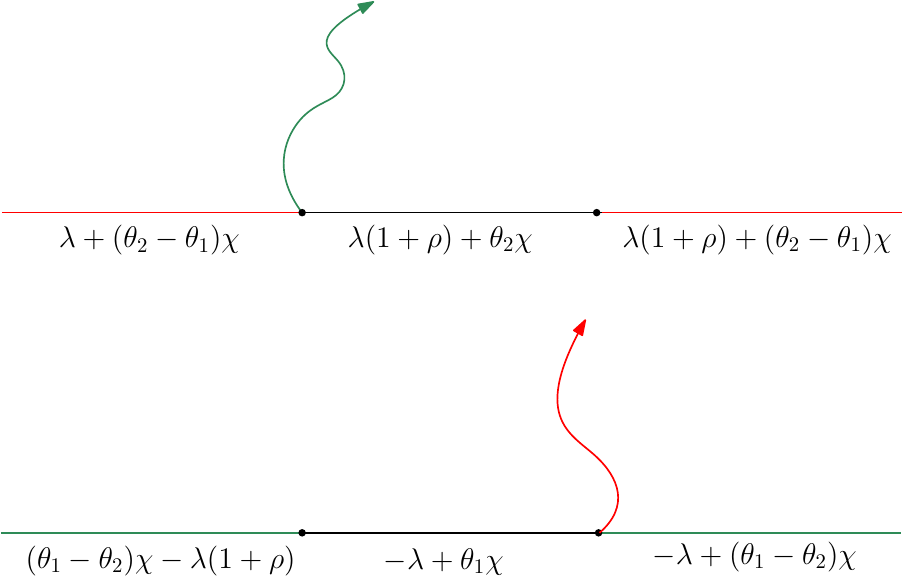}
		\caption{We can further read the conditional laws of $\eta_i$'s based on Figure~\ref{fig:radial-gff}. In the case $\eta_1$ and $\eta_2$ are non-boundary hitting simple curves, given $\eta_1$, $\eta_2$ is the flow line of the GFF with boundary conditions on the upper panel and has law $\SLE_\kappa(\frac{(\theta_1-\theta_2)\lambda}{\chi}-2;\rho+\frac{\theta_2\chi}{\lambda},-\frac{\theta_1\chi}{\lambda})$, and given $\eta_2$, $\eta_1$ is the flow line of the GFF with boundary conditions on the lower panel and has law $\SLE_\kappa(-\frac{\theta_1\chi}{\lambda},\rho+\frac{\theta_2\chi}{\lambda};\frac{(\theta_1-\theta_2)\lambda}{\chi}-2)$. If either curves is boundary or self hitting, then the conditional laws can be inferred similarly.}\label{fig:radial-gff-a}
	\end{figure}
	
	\begin{proof}
	(i) follows from Theorem~\ref{thm:radial-ig}, and (ii) follows from Proposition~\ref{prop:height-difference-radial}. To prove (iii) and (iv), we first assume $\theta_1=0$. Then it follows from~\cite[Proposition 3.30]{ig4} that $\eta_1$ is a continuous curve, and from the same argument as in the proof of ~\cite[Proposition 3.28]{ig4} we know that the conditional law of $\eta_2$ given $\eta_1$ is exactly that described in (iv). Indeed, one take conformal maps from the connected components of $\bbD\backslash\eta_1$ to $\bbH$, from which we can read off the boundary data of the GFF from the flow line boundary conditions, and further determine the conditional law via Theorem~\ref{thm-ig1}. The reason that the branch cut does not affect the conditional law is, as $\eta_2$ hits the branch cut, we would shift the branch cut as explained in the proof of Theorem~\ref{thm:radial-ig} and continue the flow line $\eta_2$. This shift of the branch cut would make the flow line boundary condition on $\eta_1$ to be continuous and compatible with the boundary data before $\eta_2$ hits the unshifted branch cut, up until $\eta_2$ hits the shifted branch cut. Then since chordal $\SLE_\kappa(\underline\rho)$ has end point continuity~\cite[Theorem 1.3]{MS16a}, it follows that $\lim_{t\to\infty}\eta_2(t) = 0$. The conditional law of $\eta_1$ given $\eta_2$ can be determined similarly. This proves the case when $\theta_1=0$, and the case when $\theta_2=0$ follows similarly. In particular, we know that $\lim_{t\to\infty}\eta_1(t) = 0$ for $\theta_1\in(2\pi(1-\frac{\alpha}{\chi}),\min\{2\pi(\frac{\alpha}{\chi}-1),\frac{2\lambda}{\chi}\})$.
	
	Now we pick $\theta_2\in(2\pi(1-\frac{\alpha}{\chi}),2\pi(\frac{\alpha}{\chi}-1))$ and $\theta_1$ such that~\eqref{eq:flowline-angle} holds. Then following the same argument in the previous paragraph, we know that both (iii) and (iv) holds in this range. In particular, it follows that $\lim_{t\to\infty}\eta_1(t) = 0$ for $\theta_1\in(2\pi(1-\frac{\alpha}{\chi}),\min\{4\pi(\frac{\alpha}{\chi}-1),\frac{2\lambda}{\chi}\})$. Thus the proposition follows by further bootstrapping over $\theta_1,\theta_2$.
	\end{proof}
	
	\subsection{Resampling uniqueness of flow lines}\label{subsec:radial-resampling}
	In~\cite[Theorem 4.1]{MS16b}, it is shown that given two boundary-to-boundary flow lines $\eta_1$ and $\eta_2$ of the GFF on the upper half plane, the conditional laws of $\eta_1$ given $\eta_2$ and $\eta_2$ given $\eta_1$ uniquely characterize the joint law of $(\eta_1,\eta_2)$. In~\cite[Theorem 5.1]{ig4}, it is shown that if $\eta_1$, $\eta_2$ are flow lines of $h+\alpha\arg(\cdot)+\beta\log|\cdot|$ where $h$ is a whole plane GFF, then the two conditional laws determine the joint law  of $(\eta_1,\eta_2)$ up to the parameter $\beta$. In this section we prove the analogous result for the flow lines in Proposition~\ref{prop:radial-flowline}.

	\begin{proposition}\label{prop:radial-resampling}
	Fix $\alpha>-\chi$, $\beta\in\bbR$, and $\theta_1,\theta_2$ in the range \eqref{eq:flowline-angle}. Let $\mu_{\alpha\beta}$ be the joint law of the curves $(\eta_1,\eta_2)$ in Proposition~\ref{prop:radial-flowline}. Also consider a Markov chain  $(X_n)_{n\ge0}$ on pair of non-crossing continuous curves $(\tilde\eta_1,\tilde\eta_2)$ from $-i$ to 0 in $\bbD$ such that in each step one uniformly randomly picks $j\in\{1,2\}$, then resamples $\tilde\eta_j$ according to the conditional law of $\eta_j$ given $\eta_{3-j}$ as described in Proposition~\ref{prop:radial-flowline}\eqref{item-resample}. Then any invariant $\sigma$-finite measure $\mu$ of $(X_n)_{n\ge0}$ can be written as $\mu = \int_\bbR \mu_{\alpha\beta}\nu(d\beta)$ where $\nu$ is a $\sigma$-finite measure on $\bbR$. 
	\end{proposition}
	
	To deal with non-probability measures,  we also need the following lemma, {which can be interpreted as extending the notion of conditional law to the $\sigma$-finite setting.}
	\begin{lemma}\label{lem:XYsample}
	    Let $\Omega'$, $\Omega''$ be  Polish spaces and $\mathcal{F}$ be the Borel $\sigma$-algebra on $\Omega'\times\Omega''$. On a measurable space $(\Omega,\mathcal{F})$, suppose $(X,Y)\to\Omega'\times\Omega''$ are random variables with distribution $m_x(dy)\mu_X(dx)$, where $\mu_X$ is a $\sigma$-finite measure and for each $x\in\Omega'$, $m_x$ is a probability measure on $(\Omega'',\mathcal{F})$. Further assume that $m_x$ is a kernel, in the sense that for any measurable set $A\subset \Omega''$, the function $x\mapsto m_x(A)$ is measurable on $\Omega'$. Then there exist $\sigma$-finite measures $(\tilde m_y)_{y\in\Omega''}, \nu$ (which are not necessarily finite), such that $(X,Y)$ can be sampled from the measure $\tilde m_y(dx)\nu(dy)$. 
	\end{lemma}
	\begin{proof}
	Let  $(A_n)_{n\ge0}\subset \mathcal{F}'$ be an increasing  sequence with $\cup_{n=1}^\infty A_n=\Omega'$ such that $\mu_X(A_n)<\infty$ for each $n$.  Then the event $\{X\in A_n\}$ has finite measure, and on the event $\{X\in A_n\}$, we can define the marginal law $\ol\nu^n(dy)$ of $Y$, and the conditional law $\ol m_y^n(dx)$ of $X$ given $Y$. These satisfy $|\ol \nu^n| = \mu_X(\{ X \in A_n\}) < \infty$ and $|\ol m_y^n(dx)|=1$. By definition, for any non-negative measurable function $f$, we have  \begin{equation}\label{eq:pf-XY-sample}
	         \int \mathds{1}_{x\in A_n}f(x,y)m_x(dy)\mu_X(dx) = \int f(x,y)\ol m_y^n(dx)\ol\nu^n(dy).
	     \end{equation}
	We define a sequence of measures $\nu^n(dy)$ and $\tilde m_y^n(dx)$ as follows. For $n=1$, we set $\nu^1=\ol\nu^1$, and $\tilde m_y^1 = \ol m_y^1$. Suppose  $\nu^n$ and $\tilde m_y^n$ have been defined, and~\eqref{eq:pf-XY-sample} holds when replacing $(\ol m_y^n, \ol \nu^n)$ with $(\tilde m_y^n,\nu^n)$. For $n+1$, we decompose $\ol\nu^{n+1} = \nu^{n+1,1}+\nu^{n+1,2}$, where $\nu^{n+1,1}$ (resp.\ $\nu^{n+1,2}$)  is absolutely continuous (resp.\ singular) w.r.t.\ $\nu^n$. Then we define $\nu^{n+1} = \nu^n + \nu^{n+1,2}$.   Let $B_{n+1}$ be the support of $\nu^{n+1,2}$.  
    For $y\in B_{n+1}$, we let $\tilde m_y^{n+1} = \ol m_y^{n+1}$, while for $y\notin B_{n+1}$, we set $\tilde m_y^{n+1} = \frac{d\nu^{n+1,1}}{d\nu^{n}}\ol m_y^{n+1}$. 
	Now for any non-negative measurable function $f$, 
	{
	\begin{equation}\label{eq:pf-XY-sample-1}
	\begin{split}
	    &\int \mathds{1}_{x\in A_{n+1}} f(x,y)\ol m_y^{n+1}(dx) \ol\nu^{n+1}(dy) \\
	    =& \int \mathds{1}_{x\in A_{n+1}} f(x,y)\ol m_y^{n+1}(dx)\nu^{n+1,1}(dy) + \int \mathds{1}_{x\in A_{n+1}}f(x,y) \ol m_y^{n+1}(dx)\nu^{n+1,2}(dy)\\
	    =& \int \mathds{1}_{x\in A_{n+1}} f(x,y)\tilde m_y^{n+1}(dx)\nu^{n}(dy) + \int \mathds{1}_{x\in A_{n+1}}f(x,y) \tilde m_y^{n+1}(dx)\nu^{n+1,2}(dy)\\
	    =& \int \mathds{1}_{x\in A_{n+1}} f(x,y)\tilde m_y^{n+1}(dx) \nu^{n+1}(dy). 
	\end{split}
	\end{equation}
	}
	  This defines a family of measures $(\nu^n,\tilde m_y^n)_n$.  
	  The measure $\nu := \lim_{n\to\infty}\nu^n$ is well-defined since it is a countable sum of mutually singular finite measures. Furthermore, by construction the event $\{X\in A_n,Y\in B_{n+1}\}$ has measure 0, {so $\mathds{1}_{x\in A_{n}} \tilde m_y^{n+1}(dx) \nu^{n}(dy) = \mathds{1}_{x\in A_{n}} \tilde m_y^{n+1}(dx) \nu^{n+1}(dy)$} a.e. Setting $f(x,y) = \mathds{1}_{x\in A_n}g(x,y)$ in~\eqref{eq:pf-XY-sample} and~\eqref{eq:pf-XY-sample-1} we   deduce that for any non-negative measurable function $g$, 
	      \begin{equation}
          \begin{split}
	          \int \mathds{1}_{x\in A_{n}} &g(x,y)\tilde m_y^{n+1}(dx)\nu^n(dy) =  \int \mathds{1}_{x\in A_{n}} g(x,y)\tilde m_y^{n+1}(dx) \nu^{n+1}(dy)\\&= \int \mathds{1}_{x\in A_{n}} g(x,y)\ol m_y^{n+1}(dx) \ol\nu^{n+1}(dy) = \int \mathds{1}_{x\in A_{n}}g(x,y) m_x(dy)\mu_X(dx) \\&= \int  \mathds{1}_{x\in A_{n}} g(x,y)\ol m_y^{n}(dx)\ol\nu^n(dy)  = \int \mathds{1}_{x\in A_{n}}g(x,y)\tilde m_y^{n}(dx)\nu^n(dy)
              \end{split}
	      \end{equation}
	      and therefore $\tilde m_y^{n+1}|_{A_n} = \tilde m_y^n$ for a.e.\ $y$ in the support of $\nu^n$. Therefore the measure $\tilde m_y = \lim_{n\to\infty} \tilde m_y^n$ is also well defined a.e.\ on the support of $\nu$, and for any $n$ and nonnegative measurable function $f$, one has
	      \begin{equation}
	           \int \mathds{1}_{x\in A_n}f(x,y)m_x(dy)\mu_X(dx) = \int \mathds{1}_{x\in A_n} f(x,y)\tilde m_y(dx)\nu(dy).
	      \end{equation}
	      This completes the proof.
	      \end{proof}

	\begin{figure}
	    \centering
	    \includegraphics[scale=0.75]{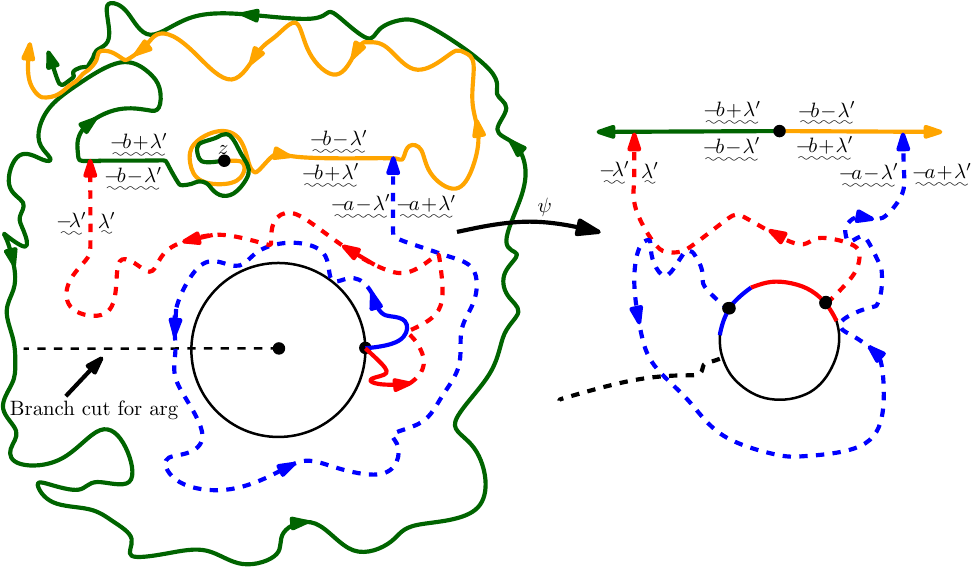}
	    \caption{The setup in Proposition~\ref{prop:radial-resampling}, which is similar to~\cite[Figure 5.7]{ig4}. The orange and the green curves are the flow lines $\gamma_1$ and $\gamma_2$, while the blue and red curves are $\eta_1$ and $\eta_2$. The annular domain $A$ is the region bounded by the unit disk, orange and green curves along with the initial segment of the red and the blue curves.}
	    \label{fig:radial-resampling}
	\end{figure}
	
	\begin{proof}[Proof of Proposition~\ref{prop:radial-resampling}]
	Since the proof is very close to that of \cite[Theorem 5.1]{ig4} discussed in \cite[Section 5.2]{ig4}, we briefly list the main ideas. Similar statements in the chordal setting are also proved in~\cite{MS16b,Yu22,Zhan23}.  We use the transform $z\mapsto\frac{1}{z}$ and work on $\mathbb{C}\backslash\mathbb{D}$.   Conditioned on the pair $(\eta_1,\eta_2)$, let $h$ be an instance of the GFF on $\mathbb{C}\backslash(\mathbb{D}\cup\eta_1\cup\eta_2)$ with boundary conditions as if $\eta_1,\eta_2$ \emph{were} flow lines of a GFF minus an $\alpha\arg$ and a $\beta\log$ singularity  as described in Proposition~\ref{prop:radial-flowline}. Then we pick some point $z$ far away from 0, and construct the flow lines $(\gamma_1,\gamma_2)$ of $h$ started from $z$ with angles uniformly chosen in $[0,2\pi]$. The existence of the flow lines $(\gamma_1,\gamma_2)$ is guaranteed by  the discussion after~\cite[Theorem 1.4]{ig4}.
	Let $\tau_i^\e$ for $i=1,2$ be the first time when $\eta_i$ hits $\partial B(0,1+\e)$, $E$ be the event that the connected component $U$ of $\bbC\backslash(\gamma_1\cup\gamma_2)$ containing 0 also contains the annulus $\{z:1<|z|<R\}$,  {which happens almost surely as $|z|\to\infty$ by~\cite[Lemma 5.6]{ig4}}. On $E$, let $A$ be the connected component of $U\backslash(\cup_{i=1}^2\eta_i([0,\tau_i^\e])\cup\partial\bbD)$ containing the point $\frac{R}{2}$. Since the flow lines are a.s.\ determined by the GFF \cite[Theorem 1.2]{ig4}, the resampling rules of $(\eta_1,\eta_2)$ imply that for $i\neq j$ the conditional law of $\eta_i$ given both $\eta_j$ and $(\gamma_1,\gamma_2)$ can be described in terms of a GFF flow line.  {As shown in~\cite[Lemma 5.7]{ig4}, given the domain $A$, the homotopy class of $(\eta_1,\eta_2)|_A$ is uniquely specified by the boundary data  $h|_{\partial A}$.} 
	
	{Let $X = (\eta_1, \eta_2)$ and $Y = (A, h_{\partial A})$. By Lemma~\ref{lem:XYsample} there exist $\sigma$-finite measures $(\tilde m_y, \mu)$ such that the law of $(X,Y)$ is $\tilde m_y(dx) \mu (dy)$. In Steps 1--3 below, we will show that for a.e.\ $Y$, the measure $\tilde m_Y$ equals some multiple of the probability measure $P_Y$ corresponding to sampling a GFF in $A$ with boundary data $h_{\partial A}$ and drawing its flow lines $(\eta_1, \eta_2)$. We then use this to conclude.
	
	\medskip 
	\noindent \textbf{Step 1: Constructing a Markov chain.}
	Given $Y = (A, h_{\partial A})$, we construct the following Markov chain $(\vartheta_1^n,\vartheta_2^n)_{n\ge0}$ on the space $\mathcal{S}$ of non-crossing continuous curves $(\vartheta_1,\vartheta_2)$ connecting $(\eta_1(\tau_1^\e),\eta_2(\tau_2^\e))$ to $\partial U$ with the homotopy class specified by $h|_{\partial A}$. We pick $i\in\{1,2\}$ uniformly and take $j\neq i$. Then we pick a GFF on $A\backslash\vartheta_j$ with the boundary condition such that the law of its flow line from $\eta_i(\tau_i^\e)$ agrees with the one specified by the conditional law of $\eta_i$ given $(A,\eta_j)$, and let $\vartheta_i$ be a sample of that flow line. 
	
	By the hypothesis of the lemma, $\tilde m_Y$ is an invariant measure with respect to this Markov chain.
	
	\medskip 
	\noindent \textbf{Step 2: Irreducibility of Markov chain.} We
	 follow the arguments in the proof of~\cite[Lemma 2.2]{Yu22} as well as~\cite[Section 4]{Zhan23}.
	
	For any pair of  simple curves $(\vartheta_1, \vartheta_2)$ in $A$ from $(\eta_1(\tau_1^\e),\eta_2(\tau_2^\e))$ to $\partial U$ with homotopy class specified by $h|_{\partial A}$, consider the measure $\nu(\cdot) = \sum_{n=1}^\infty 2^{-n}\bbP\big((\vartheta_1^n,\vartheta_2^n)\in\cdot|(\vartheta_1^0,\vartheta_2^0) = (\vartheta_1, \vartheta_2)\big)$. Let $(\Omega_1,\Omega_2)$ be some fixed pair of disjoint subdomains of $A$ which contains some pair of path  $(\vartheta_1', \vartheta_2')$ with the homotopy class given by $h|_{\partial A}$. Let $\nu_0$ be the flow lines  $(\tilde\vartheta_1, \tilde\vartheta_2)$  from $(\eta_1(\tau_1^\e),\eta_2(\tau_2^\e))$ of the GFF $\tilde{h}$ on $A$  with corresponding boundary conditions and restricted to the event $\tilde\vartheta_j\subset \ol \Omega_j$ for $j=1,2$. This event is has nonzero probability since GFF flow lines can stay arbitrarily close to given curves with positive probability~\cite[Lemma 3.9]{ig4}. Then to justify our claim, it suffices to show that $\nu_0$  is absolutely continuous w.r.t.\ $\nu$, in other words, for any set $E$ with $\nu_0(E)>0$, one has $\nu(E)>0$. Now by the untangling procedure shown in~\cite[Lemma 5.2]{ig4}, there exists a deterministic finite sequence of pairs of non-crossing simple curves $(\tilde\vartheta_1^n, \tilde\vartheta_2^n)_{n=0}^N$ such that (i) $(\tilde\vartheta_1^0, \tilde\vartheta_2^0) = (\vartheta_1,\vartheta_2)$, (ii) for every $n$, $\tilde\vartheta_1^n$ (resp.\ $\tilde\vartheta_2^n$) has the same starting and ending points (iii) for every $1\le n\le N$, there exists $j\in\{1,2\}$ such that $\tilde\vartheta_j^n = \tilde\vartheta_j^{n-1}$ and (iv) for $j=1,2$, $\tilde\vartheta_j^N\subset \ol \Omega_j$. Then combined with~\cite[Lemma 3.9]{ig4}, with positive probability  $\vartheta_j^N\subset \ol \Omega_j$ for $j=1,2$. Further applying the absolute continuity between the GFF ~\cite[Proposition 3.4]{MS16a} when restricted to the domains $\Omega_1$ and $\Omega_2$ and the fact that GFF determines flow lines~\cite[Theorem 1.2]{ig4}, one can show that if $\nu_0(E)>0$ then $\bbP((\vartheta_1^{N+2}, \vartheta_2^{N+2})\in E)>0$.  Thus the Markov chain is irreducible.
	
	\medskip 
	\noindent \textbf{Step 3: Identification of $\tilde m_Y$.} By the resampling properties of GFF flow lines, $P_Y$ is an invariant probability meassure for the Markov chain, so the Markov chain is recurrent. 
	Since it is both recurrent and irreducible, its invariant measure is unique up to a multiplicative constant~\cite[Propositions 4.2.1 and 10.1.1, Theorem 10.0.1]{Meyn-Tweedie}. Thus,  $Y$-a.s.\ the measure $\tilde m_Y$ is a multiple of $P_Y$.
    
	\medskip 
	\noindent \textbf{Conclusion.} Following the proof of Theorem 5.1 in~\cite{ig4}, consider a conformal map $\hat\psi$ sending $A$ to an annulus $\hat A$, and set $\hat h = \tilde h\circ \hat\psi^{-1}-\chi\arg(
	\hat\psi^{-1})'$. Then $\hat h$ can be written as $\hat h_0-\alpha\arg(\cdot)+\hat f_0$, where $\hat h_0$ is a zero  boundary GFF on $\hat A$, $\alpha$ is determined by the resampling property, and $\hat f_0$ is some harmonic function on $\hat A$ which is approximated by an affine transformation of the log function away from $\partial\hat A$. As we send $\e\to 0$ and $R\to\infty$, $\hat f_0$ converges to a multiple of the log function. This implies that $(\eta_1,\eta_2)$ are the angle $(\theta_1,\theta_2)$ flow lines of $h_0-\alpha\arg(\cdot)-\beta\log|\cdot|$ where $h_0$ is a zero boundary GFF on $\bbC\backslash\bbD$ and $\beta$ is  possibly random. This completes the proof.
	}
	\end{proof}

	\subsection{Counterflowline and SLE duality}
    In~\cite{MS16a,ig4}, it was shown that the left and right boundaries of counterflowlines of some GFF $h$ are flow lines of $h$. In this section, we state the results for the radial setting when the counterflowlines have law radial $\SLE_{\kappa'}(\rho_1;\rho_2)$. The proofs are mostly identical to those in~\cite[Section 4]{ig4} and will be omitted.

    \begin{theorem}\label{thm:SLE-duality}
        Let $\kappa'\in(4,8)$, $\kappa=16/\kappa'$,$\rho_1,\rho_2>\frac{\kappa'}{2}-4$, and $\rho_1+\rho_2>\frac{\kappa'}{2}-4$. Let $\alpha = \frac{\kappa'-6-\rho_1'-\rho_2'}{2\sqrt{\kappa'}}$ and $\beta\in\bbR$. Let $h$ be a GFF on $\bbD$ with the boundary conditions described in Theorem~\ref{thm:radial-ig}, such that the counterflowline $\eta'$ of $h+\alpha\arg(\cdot)+\beta\log|\cdot|$ has the law radial $\SLE_{\kappa'}^\beta(\rho_1;\rho_2)$ from $-i$ to 0 with force points $(-i)^-;(-i)^+$. Then the left and right boundaries $\eta^L$ and $\eta^R$ of $\eta'$ are the flow lines of  $h+\alpha\arg(\cdot)+\beta\log|\cdot|$ from 0 with angles $\frac{\pi}{2}$ and $-\frac{\pi}{2}$, respectively. In particular, conditioned on $\eta^L$, the curve $\eta^R$ is the union of independent curves in connected components of $\bbD\backslash\eta^L$, as follows. 
        \begin{itemize}
            \item If $\eta^L$ touches itself, let $\sigma^L$ be the last time when $\eta^L$ touches itself before hitting $\partial\bbD$, and $D_{\eta^L}$ be the connected component of $\bbD\backslash\eta^L$ whose boundary contains $\eta^L(\sigma^L)$ and a segment of $\partial\bbD$. Otherwise let $\sigma^L=0$ and $D_{\eta^L}$ be the connected component of $\bbD\backslash\eta^L$ with 0 on its boundary. Let $a_{\eta^L}$ and $b_{\eta^L}$ be the first and last points on $\partial D_{\eta^L}\cap\partial \bbD$ when one views $\partial \bbD$ as a counterclockwise curve from $(-i)^+$ to $(-i)^-$. Then in $D_{\eta^L}$, $\eta^R$ is an $\SLE_\kappa(\frac{\kappa}{4}(\rho_1+\rho_2)+\kappa-4,-\frac{\kappa}{4}\rho_1;-\frac{\kappa}{2})$ with force points $(\eta^L(\sigma^L)^-,b_{\eta^L};\eta^L(\sigma^L)^+)$ from $\eta^L(\sigma^L)$ to $a_{\eta^L}$.
            \item In each remaining component whose boundary contains a segment of the counterclockwise boundary arc of $\partial\bbD$ from $-i$ to $a_{\eta^L}$, $\eta^R$ is a chordal $\SLE_\kappa(\frac{\kappa}{4}\rho_2+\kappa-4;-\frac{\kappa}{2})$ process from the first to the last boundary point hit by $\eta^L$.
            \item In each remaining component whose boundary does not contains a segment of $\partial\bbD$, $\eta^R$ is a chordal $\SLE_\kappa(\frac{\kappa}{4}(\rho_1+\rho_2)+\kappa-4;-\frac{\kappa}{2})$  process from the first to the last boundary point hit by $\eta^L$.
        \end{itemize}
       Similarly, given $\eta^R$, defining $\sigma^R$, $D_{\eta^R}$, $a_{\eta^R}$, $b_{\eta^R}$ as above with $\eta^L$ replaced by $\eta^R$ and counterclockwise replaced by clockwise, the curve $\eta^L$ is $\SLE_\kappa(-\frac{\kappa}{2};\frac{\kappa}{4}(\rho_1+\rho_2)+\kappa-4,-\frac{\kappa}{4}\rho_1)$ in $D_{\eta^L}$, chordal $\SLE_\kappa(-\frac{\kappa}{2};\frac{\kappa}{4}\rho_1+\kappa-4)$ in each other connected component of $\bbD\backslash\eta^R$  whose boundary contains a segment of the  clockwise boundary arc of $\partial\bbD$ from $(-i)^-$ to $a_{\eta^R}$, and chordal $\SLE_\kappa(-\frac{\kappa}{2};\frac{\kappa}{4}(\rho_1+\rho_2)+\kappa-4)$  process in each other connected component of $\bbD\backslash\eta^R$  whose boundary does not contain any segment of $\partial\bbD$. Finally, given $\eta^L$ and $\eta^R$, in each connected component of $\bbD\backslash(\eta^L\cup\eta^R)$ between $\eta^L$ and $\eta^R$, the law of $\eta'$ is chordal $\SLE_{\kappa'}(\frac{\kappa'}{2}-4;\frac{\kappa'}{2}-4)$.
    \end{theorem}

    \begin{proof}
        The proof is identical to that of~\cite[Theorem 4.6]{ig4}.
    \end{proof}
	
    \begin{proposition}\label{prop:resample-nonsimple}
         Let $\kappa',\kappa,\rho_1,\rho_2,\alpha,\beta$ be as in Theorem~\ref{thm:SLE-duality}. Let $\mu_{\alpha,\beta}$ be the joint law of the curves $(\eta^L,\eta^R)$ in Theorem~\ref{thm:SLE-duality}. Also consider a Markov chain  $(X_n)_{n\ge0}$ on pair of non-crossing continuous curves $(\tilde\eta_1,\tilde\eta_2)$ from $0$ to $-i$ in $\bbD$ such that in each step one uniformly randomly picks $j\in\{1,2\}$, then resamples $\tilde\eta^j$ according to the conditional law of $\eta^j$ given $\eta^{3-j}$ as described in Proposition~\ref{prop:radial-flowline} where we identify $(1,2)=(L,R)$. Then any invariant $\sigma$-finite measure $\mu$ of $(X_n)_{n\ge0}$ can be written as $\mu = \int_\bbR \mu_{\alpha\beta}\nu(d\beta)$ where $\nu$ is a $\sigma$-finite measure on $\bbR$. 
    \end{proposition}	

    \begin{proof}
     Suppose $(\eta_1,\eta_2)$ follows the stationary distribution $\mu$.   As explained in the first step of the  proof of~\cite[Theorem 4.1]{MS16b} and~\cite[Lemma 2.2]{Yu22}, given $\eta_2$, we run some counterflowline $\eta_1'$ started from $-i$ such that $\eta_1$ a.s.\ merges into $\eta_1'$. We further run $\eta_1,\eta_2,\eta_1'$ for a small amount of time, which separates the starting and the ending points of $\eta_1$ and $\eta_2$. The rest of the proof is then identical to Proposition~\ref{prop:radial-resampling}.
    \end{proof}
	
	\section{Preliminaries on Liouville quantum gravity}\label{sec:pre-LQG}

	In Section~\ref{subsec:pre-lf}, we recall the definition of the \emph{Liouville field} which underlies Liouville conformal field theory \cite{DKRV16, HRV-disk,GRV19}. In Section~\ref{sec-prelim-resampling} we state some resampling properties of the Liouville field. In Section~\ref{subsec:pre-qs}, we go over the definitions  of quantum surfaces, and in particular the quantum disks and quantum triangles. In Section~\ref{subsec-third-point} we discuss the quantum surface arising from adding a third boundary point to a quantum disk. 
    Finally in Section~\ref{subsec:weld-qt}, we explain the quantum triangle welding result in~\cite{ASY22} and derive Theorem~\ref{thm:main} assuming Theorem~\ref{thm:w2w}.

	Throughout the paper, we let $\gamma\in(0,2)$, $\kappa=\gamma^2$, $Q=\frac{\gamma}{2}+\frac{2}{\gamma}$, and fix the notation $|z|_+ := \max\{|z|, 1\}$.
	
	\subsection{Liouville fields}\label{subsec:pre-lf}
	The Liouville field on the disk was first introduced in \cite{HRV-disk} in the context of constructing correlation functions of Liouville CFT on the disk, but we follow the presentation of \cite{AHS21} and work in the upper half-plane $\bbH$. Recall from Section~\ref{subsec:GFF} that $P_\bbH$ denotes the law of the free boundary GFF on $\bbH$ normalized to have average zero on the unit semicircle $\{ e^{i \theta} : \theta \in (0,\pi)\}$, and $|z|_+ = \max (|z|, 1)$.

	\begin{definition}
	Let $(h,\mathbf{c})$ be sampled from $P_{\bbH}\times [e^{-Qc}dc]$ and let $\phi(z) = h(z)-2Q\log|z|_++\mathbf{c}$. Let $\LF_{\bbH}$ be the law of the random field $\phi$, and we call a sample from $\LF_{\bbH}$ a Liouville field on $\bbH$.
	\end{definition}
	
	\begin{definition}\label{def:lf-insertion}
	Let $(\alpha_i,z_i)\in\bbR\times\bbH$ and $(\beta_j,s_j)\in\bbR\times\partial\bbH$ for $i=1,...,m$ and $j = 1,...,n$ and the $z_i$'s and $s_j$'s are pairwise distinct. Set 
	\begin{equation*}
	\begin{split}
	&C_{\bbH}^{(\alpha_i,z_i)_i,(\beta_j,s_j)_j} =\\ &\prod_{i=1}^m\prod_{j=1}^n(2\mathrm{Im}z_i)^{-\frac{\alpha_i^2}{2}}|z_i|_+^{-2\alpha_i(Q-\alpha_i)}|s_j|_+^{-\beta_i(Q-\frac{\beta_i}{2})}e^{\frac{\alpha_i\beta_j}{2}G_\bbH(z_i,s_j)+\sum_{i'=i+1}^{m}\alpha_i\alpha_{i'}G_\bbH(z_i,z_{i'})+\sum_{j'=j+1}^{n}\frac{\beta_j\beta_{j'}}{4}G_\bbH(s_j,s_{j'}) }.   
	\end{split}
	\end{equation*}
	Let $(h,\mathbf{c})$ be sampled from $C_{\bbH}^{(\alpha_i,z_i)_i,(\beta_j,s_j)_j}P_\bbH\times [e^{(\sum_i\alpha_i+\frac{1}{2}\sum_j\beta_j-Q)c}dc]$, and 
	\begin{equation*}
	    \phi(z) = h(z)-2Q\log|z|_++\sum_{i=1}^m\alpha_iG_\bbH(z,z_j)+\sum_{j=1}^n\frac{\beta_j}{2}G_\bbH(z,s_j)+\mathbf{c}.
	\end{equation*}
	We write $\LF_{\bbH}^{(\alpha_i,z_i)_i,(\beta_j,s_j)_j}$ for the law of $\phi$ and call a sample from $\LF_{\bbH}^{(\alpha_i,z_i)_i,(\beta_j,s_j)_j}$ the Liouville field on $\bbH$ with insertions $(\alpha_i,z_i)_{1\le i\le m}, (\beta_j,s_j)_{1\le j\le n}$.
	\end{definition}
	This definition can also be extended to the case where one of the insertions is $\infty$.
	\begin{definition}\label{def:lf-insertion-infty}
	Let $\beta\in\bbR$, $(\alpha_i,z_i)\in\bbR\times\bbH$ and $(\beta_j,s_j)\in\bbR\times\partial\bbH$ for $i=1,...,m$ and $j = 1,...,n$ and the $z_i$'s and $s_j$'s are pairwise distinct. Set 
	\begin{equation*}
	\begin{split}
	&C_{\bbH}^{(\alpha_i,z_i)_i,(\beta_j,s_j)_j,(\beta,\infty)} =\\ &\prod_{i=1}^m\prod_{j=1}^n(2\mathrm{Im}z_i)^{-\frac{\alpha_i^2}{2}}|z_i|_+^{-2\alpha_i(Q-\alpha_i-\frac{\beta}{2})}|s_j|_+^{-\beta_i(Q-\frac{\beta_i+\beta}{2})}e^{\frac{\alpha_i\beta_j}{2}G_\bbH(z_i,s_j)+\sum_{i'=i+1}^{m}\alpha_i\alpha_{i'}G_\bbH(z_i,z_{i'})+\sum_{j'=j+1}^{n}\frac{\beta_j\beta_{j'}}{4}G_\bbH(s_j,s_{j'}) }.   
	\end{split}
	\end{equation*}
	Let $(h,\mathbf{c})$ be sampled from $C_{\bbH}^{(\alpha_i,z_i)_i,(\beta_j,s_j)_j,(\beta,\infty)}P_\bbH\times [e^{(\sum_i\alpha_i+\frac{\beta}{2}+\frac{1}{2}\sum_j\beta_j-Q)c}dc]$, and 
	\begin{equation*}
	    \phi(z) = h(z)+(\beta-2Q)\log|z|_++\sum_{i=1}^m\alpha_iG_\bbH(z,z_j)+\sum_{j=1}^n\frac{\beta_j}{2}G_\bbH(z,s_j)+\mathbf{c}.
	\end{equation*}
	We write $\LF_{\bbH}^{(\alpha_i,z_i)_i,(\beta_j,s_j)_j,(\beta,\infty)}$ for the law of $\phi$ and call a sample from $\LF_{\bbH}^{(\alpha_i,z_i)_i,(\beta_j,s_j)_j,(\beta,\infty)}$ the Liouville field on $\bbH$ with insertions $(\alpha_i,z_i)_{1\le i\le m}, (\beta_j,s_j)_{1\le j\le n}, (\beta,\infty)$.
	\end{definition}

	As we state in Lemma~\ref{lem:lf-covariance}, the Liouville field satisfies \emph{conformal covariance}. For a conformal map $f:D\to\tilde{D}$ and a measure $M$ on $H^{-1}(D)$, let $f_*M$ be the pushforward of $M$ under the map $\phi\mapsto\phi\circ f^{-1}+Q\log|(f^{-1})'|$.
	\begin{lemma}\label{lem:lf-covariance}
	For $\alpha\in\bbR$, set $\Delta_\alpha:=\frac{\alpha}{2}(Q-\frac{\alpha}{2})$. Let $(\alpha_i,z_i)\in\bbR\times\bbH$ and $(\beta_j,s_j)\in\bbR\times\partial\bbH$ for $i=1,...,m$ and $j = 1,...,n$. Suppose $f:\bbH\to\bbH$ is a conformal map, such that $f(s_j)\neq\infty$ for all $j$. Then $$\LF_{\bbH}^{(\alpha_i,f(z_i))_i,(\beta_j,f(s_j))_j} = \prod_{i=1}^m\prod_{j=1}^n|f'(z_i)|^{-2\Delta_{\alpha_i}}|f'(s_j)|^{-\Delta_{\beta_j}}f_*\LF_{\bbH}^{(\alpha_i,z_i)_i,(\beta_j,s_j)_j}. $$
	\end{lemma}
	\begin{proof}
	As explained in \cite[Proposition 2.7]{AHS21}, this is a restatement of \cite[Theorem 3.5]{HRV-disk}. 
	\end{proof}
	
	The conformal covariance property also extends to the setting where there is an insertion at $\infty$.
	\begin{lemma}[Lemma 2.11 of~\cite{ASY22}]\label{lem:lcft-H-conf-infty}
		Suppose $\beta_1,\beta_2,\beta_3\in\mathbb{R}$ and $f:\mathbb{H}\to\mathbb{H}$ is the conformal map with $f(0) = 1$, $f(1)=\infty$ and $f(\infty) = 0$. Then 
		\begin{equation}
		\textup{LF}_{\mathbb{H}}^{(\beta_1, 0), (\beta_2, 1), (\beta_3, \infty)} = f_*\textup{LF}_{\mathbb{H}}^{(\beta_1, \infty), (\beta_2, 0), (\beta_3, 1)}.
		\end{equation} 
	\end{lemma}
	
	We define the Liouville field on the strip $\mathcal{S}$ in Definition~\ref{def:lf-strip}, and explain how to switch between $\cS$ and $\bbH$ in Lemma~\ref{lem:lcft-H-strip}.
	\begin{definition}\label{def:lf-strip}
		Let $(h, \mathbf{c})$ be sampled from $C_{\mathcal{S}}^{(\beta_1, +\infty), (\beta_2, -\infty), (\beta_3, s_3)}P_{\mathcal{S}}\times[e^{(\frac{\beta_1+\beta_2+\beta_3}{2}-Q)c}dc]$ with $\beta_1, \beta_2, \beta_3\in\mathbb{R}$,  $s_3\in\partial\mathcal{S}$ and
		{$$C_{\mathcal{S}}^{(\beta_1, +\infty), (\beta_2, -\infty), (\beta_3, s_3)} = e^{(-\Delta_{\beta_3}+\frac{(\beta_1+\beta_2)\beta_3}{4})|\textup{Re}s_3|+\frac{(\beta_1-\beta_2)\beta_3}{4}\textup{Re}s_3}.$$}
		Let $\phi(z) = h(z)+\frac{\beta_1+\beta_2-2Q}{2}|\textup{Re}z|+\frac{\beta_1-\beta_2}{2}\textup{Re}z+\frac{\beta_3}{2}G_{\mathcal{S}}(z, s_3)+\mathbf{c}$. We write $\textup{LF}_{\mathcal{S}}^{(\beta_1, +\infty), (\beta_2, -\infty), (\beta_3, s_3)}$ for the law of $\phi$.
	\end{definition}
	
	\begin{lemma}[Lemma 2.8 of~\cite{ASY22}]\label{lem:lcft-H-strip}
		For $\beta_1, \beta_2, \beta_3\in\mathbb{R}$ and $s_3\in\partial\mathcal{S}$, we have 
		\begin{equation}\label{eqn-lcft-H-strip}
		\textup{LF}_{\mathbb{H}}^{(\beta_1, \infty), (\beta_2, 0), (\beta_3, e^{s_3})}= e^{-\Delta_{\beta_3}\textup{Re}s_3}\exp_*\textup{LF}_{\mathcal{S}}^{(\beta_1, +\infty), (\beta_2, -\infty), (\beta_3, s_3)}.
		\end{equation} 
	\end{lemma}
	
	Finally, we introduce the Liouville field on $\bbC$. Let $P_\bbC$ denote the law of the GFF on $\bbC$ normalized to have average zero on the unit circle $\{|z| =1 \}$. Its Green function is given by 
	\[G_\bbC(z,w) = -\log|z-w| +\log |z|_+ + \log |w|_+, \]
	and we use the convention $G_\bbC(z,\infty) = \lim_{w \to \infty} G_\bbC (z, w) =  \log |z|_+$. 
	\begin{definition}
	Let $m \geq 0$, let $(\alpha_i, z_i) \in \bbR \times \bbC$ for $i = 1, \dots, m$ where the $z_i$'s are pairwise distinct, and let $\alpha \in \bbR$. Set 
	\[C_\bbC^{(\alpha_i, z_i)_i, (\alpha, \infty)} = \prod_{i=1}^m |z_i|_+^{-\alpha_i(2Q-\alpha_i - \alpha)} e^{\sum_{j=i+1}^{m+1} \alpha_i\alpha_j G_\bbC (z_i, z_j)}.\]
	Sample $(h, \mathbf c)$ from $C_\bbC^{(\alpha_i, z_i)_i} P_\bbC \times [e^{(\alpha + \sum_{i=1}^m \alpha_i - 2Q)c}dc]$, and set $\phi(z) = h(z) +(\alpha- 2Q) \log |z|_+ + \sum_{i=1}^m \alpha_i G_\bbC(z, z_i) + \mathbf c$. We write $\LF_\bbC^{(\alpha_i,z_i)_i, (\alpha , \infty)}$ for the law of $\phi$. 
	\end{definition}

	\subsection{Resampling properties of the Liouville field}\label{sec-prelim-resampling}
	
	If $A \subset \bbH$ and  $\phi$ is a Liouville field, one can describe the conditional law of $\phi|_A$ given $\phi|_{\bbH \backslash A}$. In order to do this, since the law of $\phi$ is an infinite measure, to make sense we need to specify the notion of conditioning:
	\begin{definition}\label{def-markov-kernel}
	    Suppose $(\Omega, \cF)$ and $(\Omega', \cF')$ are measurable spaces. A function $\Lambda: \Omega \times \cF' \to [0,1]$ is a \emph{Markov kernel} if $\Lambda(\omega, \cdot)$ is a probability measure on $(\Omega', \cF')$ for each $\omega \in \Omega$, and $\Lambda(\cdot, A)$ is $\cF$-measurable for each $A \in \cF'$. If $\mu$ is some measure on $(\Omega, \cF)$ and $(X,Y)$ is a sample from $\Lambda(x, dy) \mu(dx)$, we say the \emph{conditional law} of $Y$ given $X$ is $\Lambda(X, \cdot)$.
	\end{definition}
	
	We now state a number of resampling properties of Liouville fields with insertions. These essentially follow immediately from the Markov property of the Gaussian free field; see e.g.\ \cite[Lemma 5.5]{ASY22}.
	\begin{lemma}\label{lem-markov-LF-C}
	Let $\alpha_1, \alpha_2, \alpha_3 \in \bbR$.
	    Suppose $\psi\sim \LF_\bbC^{(\alpha_1, 0), (\alpha_2, 1), (\alpha_3, \infty)}$, and $U$ is a deterministic neighborhood of $0$ bounded away from $1$ and $\infty$. Conditioned on $\psi|_{\bbC \backslash U}$,  we have $\psi|_U \stackrel d= h + \mathfrak h + \alpha_1 G_U(\cdot, 0)$ where $h$ is a GFF on $U$ with zero boundary conditions, $\mathfrak h$ is the harmonic extension of $\psi|_{\bbC \backslash U}$ to $U$, and $G_U$ is the Green function of $h$. 
	\end{lemma}
	
	\begin{lemma}\label{lem-lf-resample-i-infty}
	    Let $\alpha, \beta \in \bbR$. Suppose $\psi \sim \LF_\bbH^{(\alpha, i), (\beta, \infty)}$. If $U$ is a deterministic bounded neighborhood of $i$ disjoint from $\bbR$, then conditioned on $\psi|_{\bbH \backslash U}$, we have $\psi|_U \stackrel d= h + \mathfrak h + \alpha G_U(\cdot , i)$ where $h$ is a GFF on $U$ with zero boundary conditions, $\mathfrak h$ is the harmonic extension of $\psi|_{\bbH \backslash U}$ to $U$, and $G_U$ is the Green function of $h$. If instead $U$ is a deterministic neighborhood of $\infty$ bounded away from $i$, then conditioned on $\psi|_{\bbH \backslash U}$, we have $\psi|_U \stackrel d= h + \mathfrak h + (\frac\beta2 - Q) G_U(\cdot, \infty)$ where $h$ is a GFF on $U$ with zero (resp.\ free) boundary conditions on $\partial U \cap \bbH$ (resp.\ $\partial U \cap \bbR$), $\mathfrak h$ is the harmonic extension of $\psi|_{\bbH \backslash U}$ to $U$ with zero normal derivative on $\partial U \cap \bbR$, and $G_U$ is the Green function for $h$. 
	\end{lemma}
	
	\begin{lemma}\label{lem-lf-resample-3}  
	Let $\beta_1, \beta_2, \beta_3 \in \bbR$.
	     Suppose $\psi\sim \LF_\bbH^{(\beta_1, 0), (\beta_2, 1), (\beta_3, \infty)}$. Let $\mathfrak h$ be the harmonic function on $\bbH$ whose values on $(-\infty, 1)$ agree with those of $\psi$ and which has normal derivative zero on $(1,\infty)$. Conditioned on $\mathfrak h$, we have $\psi \stackrel d= h + \mathfrak h$ where $h$ is a GFF on $\bbH$ with zero boundary conditions on $(-\infty, 1)$ and free boundary conditions on $(1, \infty)$.
	\end{lemma}
	
	\begin{lemma}\label{lem-lf-resample-4}
	Let $\beta_1, \beta_2, \beta_3, \beta_4 \in \bbR$ and $x \in (1, \infty)$.
	     Suppose $\psi\sim \LF_\bbH^{(\beta_1, 0), (\beta_2, 1), (\beta_3, x), (\beta_4, \infty)}$. Let $\mathfrak h$ be the harmonic function on $\bbH$ whose values on $(0,x)$ agree with those of $\psi$ and which has normal derivative zero on $(-\infty, 0) \cup (x, \infty)$. Conditioned on $\mathfrak h$, we have $\psi \stackrel d= h + \mathfrak h + (\frac{\beta_2}2 -Q) G(\cdot, \infty)$ where $h$ is a GFF on $\bbH$ with zero boundary conditions on $(0,x)$ and free boundary conditions on $(-\infty, 0) \cup (x, \infty)$, and $G$ is the Green function of $h$. 
	\end{lemma}

	\subsection{Quantum surfaces, quantum disks and quantum triangles}\label{subsec:pre-qs}

	In this section we review the  quantum disks and quantum triangles that will be used in our paper. Let $\mathcal{D}\mathcal{H} = \{(D,h,z_1, ..., z_m):D\subset\mathbb{C}\ \text{open},\ h\text{ a distribution on }D,\ z_1,...,z_m\in\bar{D}\}$.  For two tuples $(D,h,z_1, ..., z_m)$ and $(\tilde{D}, \tilde{h}, \tilde{z}_1, ..., \tilde{z}_m)\in\mathcal{D}\mathcal{H} $, we say  
	\begin{equation}\label{eqn-qs-relation}
	(D, h, z_1, ..., z_m)\sim_\gamma (\tilde{D}, \tilde{h}, \tilde{z}_1, ..., \tilde{z}_m)
	\end{equation}
	if one can find a conformal mapping $f:\tilde{D}\to D$ such that $f(\tilde{z}_j) = z_j$ for each $j$ and $\tilde{h} = f\bullet_\gamma h$, where 
	\eqb\label{eq-bullet-gamma}
	f\bullet_\gamma h
	:= h\circ f^{-1}+Q\log|(f^{-1}) '|.
	\eqe
	We call each tuple $(D, h, z_1, ..., z_m)$ modulo the equivalence relation $\sim_\gamma$ a decorated $\gamma$-\textit{quantum surface}. An \emph{embedding} of a decorated quantum surface is a choice of representative $(D, h, z_1, ..., z_m)$. We can similarly define curve-decorated quantum surfaces by considering tuples $(D, h, z_1, \dots, z_m, \eta)$ where $\eta$ is a curve on $\ol D$ by further requiring that $\eta = f \circ \tilde \eta$.
	
	Let $\phi=h_D+g$  , where $h_D$ is the GFF on a domain $D$ and $g$ is a (possibly random) continuous function.  One can then define the quantum area measure $\cA_\phi(d^2z):=\lim_{\e\to0}\e^{\gamma^2/2}e^{\gamma \phi_\e(z)}d^2z$, where $\phi_\e(z)$ denotes the average of $\phi$ over $\partial B_\e(z)\cap D$. When $D=\bbH$, the quantum length measure is defined by  $\cL_\phi(dx):=\lim_{\e\to0}\e^{\gamma^2/4}e^{\gamma \phi_\e(x)/2}dx$   where $\phi_\e(x)$ is the average of $\phi$ over  $\partial B_\e(x)\cap\bbH$.  It has been shown in \cite{DS11} that the limits exist in probability. The quantum area and  length measures depend only on the quantum surface: in~\eqref{eqn-qs-relation}, if $D = \tilde D = \bbH$ and $\psi: \bbH \to \bbH$ is a conformal automorphism, then $\psi_* \cA_{\tilde \phi} = \cA_\phi$ and $\psi_* \cL_{\tilde \phi} = \cL_\phi$ \cite{DS11, SW16}. Thus, the notion of quantum length can be extended to free boundary GFFs on arbitrary domains with boundary via conformal maps and the $\bullet_\gamma$ relation.

	Next we present the definition of the \textit{weight $W$ (thick) quantum disk},  introduced in \cite[Section 4.5]{DMS14}.  Recall that for the free boundary GFF $h_{\mathcal{S}}$ on the strip, we have the decomposition $h_{\mathcal{S}} = h_{\mathcal{S}}^1+h_{\mathcal{S}}^2$, where $h_{\mathcal{S}}^1$ is constant on each of $u+[0,i\pi]$ for $u\in\bbR$, and  $h_{\mathcal{S}}^2$ has mean zero on all such lines \cite[Section 4.1.6]{DMS14}.
	
	\begin{definition}\label{def:thickdisk}
		Fix $W\ge\frac{\gamma^2}{2}$ and let $\beta = \gamma+ \frac{2-W}{\gamma}\leq Q$. Let $\hat\psi=\psi_1+\psi_2$ where $\psi_1$ and $\psi_2$ are independent distributions on $\mathcal{S}$ such that: 
		\begin{enumerate}
			\item $\psi_1(z) = X_{\mathrm{Re}(z)}$, with
			\begin{equation}
			X_t:=\left\{ \begin{array}{rcl} 
				B_{2t}-(Q-\beta) t & \mbox{for} & t\ge 0\\
				\tilde{B}_{-2t} +(Q-\beta) t & \mbox{for} & t<0
			\end{array} 
			\right.
			\end{equation}
	        where $(B_t)_{t\ge 0}$ and $(\tilde{B}_t)_{t\ge 0}$ are independent standard Brownian motions conditioned on  $B_{2t}-(Q-\beta)t<0$ and  $ \tilde{B}_{2t} - (Q-\beta)t<0$ for all $t>0$\footnote{This conditioning can be made sense of via Bessel processes; see e.g.~\cite[Section 4.2]{DMS14}.};  
			\item $\psi_2$ has the same law as $h_{\mathcal{S}}^2$.
		\end{enumerate}
		 Let   $\mathbf{c}$ be independently sampled from $\frac{\gamma}{2}e^{(\beta-Q)c}dc$ and set $\psi = \hat\psi+\mathbf{c}$. Let  $\mathcal{M}^{\textup{disk}}_{0,2}(W)$ be the infinite measure describing the law of $(\mathcal{S}, \psi, -\infty, +\infty)/{\sim_\gamma}$. We call a sample from $\mathcal{M}^{\textup{disk}}_{0,2}(W)$ a (two-pointed) \emph{quantum disk of weight $W$}.
	\end{definition}

	When $0<W<\frac{\gamma^2}{2}$, we can also define the weight $W$ quantum wedge and weight $W$ quantum disk as  concatenations of weight $\gamma^2-W$ (two-pointed) thick quantum disks as in \cite[Section 2]{AHS20}.

	\begin{definition}\label{def-thin-wedge}
		For $W\in(0, \frac{\gamma^2}{2})$, the probability measure $\mathcal{M}^{\textup{wed}}(W)$ is defined as follows. Sample a Poisson point process $\{(u, \mathcal{D}_u)\}$ from the intensity measure $\mathds{1}_{t>0}dt\times \mathcal{M}_{0,2}^{\textup{disk}}(\gamma^2-W)$ and concatenate the quantum disks $\{\mathcal{D}_u\}$ according to the ordering induced by $u$. The resulting beaded quantum surface is a \emph{weight $W$ quantum wedge}, and we denote its law by $\cM^\mathrm{wed}(W)$. For $x>0$, the \emph{quantum cut point measure} assigns  mass
	$x$ to the collection of cut points between the quantum disks $\{\mathcal{D}_u:u\leq x\}$.
	\end{definition}

	\begin{definition}\label{def-thin-disk}
		Let $W\in(0, \frac{\gamma^2}{2})$. 
	    Sample $(\cW, T) \sim \cM^\mathrm{wed}(W) \times (1-\frac{2}{\gamma^2}W)^{-2}\textup{Leb}_{\mathbb{R}_+}$, and write $\cW = \{(u, \cD_u)\}$. The \emph{weight $W$ quantum disk} is the concatenation of the quantum disks $\{\cD_u \: : \: u \leq T\}$ according to the ordering induced by $u$. Its left (resp.\ right) boundary length is the sum of the left (resp.\ right) boundary lengths of its constituent $\cD_u$.
	    We define the infinite measure $\Md_{0,2}(W)$ to be the law of the weight $W$ quantum disk.
	    For $x\in [0,T]$, the \emph{quantum cut point measure} assigns  mass
	    $x$ to the collection of cut points between the quantum disks $\{\mathcal{D}_u:u\leq x\}$.
	\end{definition}
	As explained in \cite[Corollary 2.14]{AHS20}, the quantum cut point measure of a quantum wedge $\cW$ is measurable with respect to $\cW$. 
    Indeed, map the left boundary of $\cW$ to $[0,\infty)$ via the quantum length measure. Then the pushforward of the quantum cut point measure agrees (up to deterministic multiplicative constant) with the $(1-\frac{2W}{\gamma^2})$-Minkowski content measure of the image of the cut points.  
    This measurability is not immediate from Definition~\ref{def-thin-wedge} since, a priori, the label $u$ for each $\cD_u$ is not part of the data of $\cW$. Similarly, the quantum cut point measure of a thin quantum disk is measurable with respect to the thin quantum disk.

	We now recall the notion of quantum triangle in Section~\ref{subsec:intro-qt}. It is a quantum surface parameterized by weights $W_1,W_2,W_3>0$, and is defined based on Liouville fields with three insertions and the thick-thin duality. See Figure~\ref{fig-qt}.
	
	\begin{definition}[Thick quantum triangles]\label{def-qt-thick}
		Fix {$W_1, W_2, W_3>\frac{\gamma^2}{2}$}. Set $\beta_i = \gamma+\frac{2-W_i}{\gamma}<Q$ for $i=1,2,3$, and {let $\phi$ be sampled from $\frac{1}{(Q-\beta_1)(Q-\beta_2)(Q-\beta_3)}\textup{LF}_{\mathbb{H}}^{(\beta_1, \infty), (\beta_2, 0), (\beta_3, 1)}$}. Then we define the infinite measure $\textup{QT}(W_1, W_2, W_3)$ to be the law of $(\bbH, \phi, \infty, 0,1)/{\sim_\gamma}$.
	\end{definition}
	
	\begin{definition}[Thin quantum triangles]\label{def-qt-thin}
		Fix {$W_1, W_2, W_3\in (0,\frac{\gamma^2}{2})\cup(\frac{\gamma^2}{2}, \infty)$}. Let $I:=\{i\in\{1,2,3\}:W_i<\frac{\gamma^2}{2}\}$ and define  $\textup{QT}(W_1, W_2, W_3)$ to be the law of the quantum surface $S$ constructed as follows. 
	    {Sample $(S_0, (S_i)_I)$ from $\textup{QT}(\tilde{W}_1, \tilde{W}_2, \tilde{W}_3) \times \prod_{i\in I} (1-\frac{2W_i}{\gamma^2})\mathcal{M}_2^{\textup{disk}}(W_i)$ where $\tilde{W}_i = \max (W_i, \gamma^2 - W_i)$. Let $(D, \phi, a_1, a_2, a_3)$ be an embedding of $S_0$. For each $i \in I$, identify the first marked point of $S_i$ with the vertex $a_i$ of $S_0$, and forget this marked point. Let $S$ be the resulting three-pointed quantum surface. }
	\end{definition}

	These definitions can be extended to the setting where one or more weights $W_i$ is equal to $\frac{\gamma^2}2$, see \cite[Section 2.5]{ASY22} for details. 
	
	We introduce the \emph{disintegration} of the measure  $\Md_{0,2}(W)$  over the boundary arc lengths as in \cite[Section 2.6]{AHS20}. Namely, we have 
	\eqb\label{eq-disintegrate-disk}
	\Md_{0,2}(W) = \iint_{\bbR_+^2}\Md_{0,2}(W;\ell_1,\ell_2)\ d\ell_1\ d\ell_2
	\eqe
	where $\Md_{0,2}(W;\ell_1,\ell_2)$ is the measure supported on doubly marked quantum surfaces with left and right boundary arcs having quantum lengths $\ell_1$ and $\ell_2$. Similarly, this notion can be extended to quantum triangles by setting
	\begin{equation}\label{eq-disintegrate-triangle}
	    \QT(W_1,W_2,W_3) = \iiint_{\bbR_+^3}\QT(W_1,W_2,W_3;\ell_1,\ell_2,\ell_3)\ d\ell_1\ d\ell_2\ d\ell_3
	\end{equation}
	where $\QT(W_1,W_2,W_3;\ell_1,\ell_2,\ell_3)$ is the measure supported on quantum surfaces $(D,\phi,a_1,a_2,a_3)/{\sim_\gamma}$ such that the boundary arcs between $a_1a_2$, $a_1a_3$ and $a_2a_3$ has quantum lengths $\ell_1,\ell_2,\ell_3$. We can also disintegrate over one or two boundary arc lengths of quantum triangles. For instance, we can define 
	$$\QT(W_1,W_2,W_3;\ell_1,\ell_2) = \int_0^\infty \QT(W_1,W_2,W_3;\ell_1,\ell_2,\ell_3)\ d\ell_3$$
	and $$\QT(W_1,W_2,W_3;\ell_1) = \iint_{\bbR^2_+}\QT(W_1,W_2,W_3;\ell_1,\ell_2,\ell_3)\ d\ell_2\ d\ell_3.$$

	Finally, we define quantum disks with one bulk and one boundary insertion using the Liouville field. 
	\begin{definition}\label{def:mdisk11}
	    Fix $W>0,W'>\frac{\gamma^2}{2}$ and let $\alpha = Q-\frac{W}{2\gamma}$, $\beta = \frac{\gamma}{2}+\frac{2-W'}{\gamma}$. Let $\phi$ be a sample from $\frac{1}{Q-\beta}\LF_{\bbH}^{(\alpha,i),(\beta,0)}$. Then we define the infinite measure $\Md_{1,1}(W,W')$ to be the law of $(\bbH,\phi,i,0)/{\sim_\gamma}$.
	\end{definition}
	
	\subsection{Adding a third marked point to a quantum disk} \label{subsec-third-point}
	
	For a sample $\cD$ from $\Md_{2}(W)$, let $\cL_\cD^\mathrm{left}$ denote the quantum length measure on the left boundary arc of $\cD$. For $(p, \cD)$ sampled from $\cL_\cD^\mathrm{left} (dp)\Md_{0,2}(W)(d\cD)$, let $\Md_{0,2, \bullet}(W)$ denote the law of the three-pointed quantum surface obtained from adding to $\cD$ the third marked point $p$.
	
	\begin{lemma}\label{lem-embed-marked-disk}
		Let $W > \frac{\gamma^2}2$ and $\beta = \frac\gamma2 + \frac{2-W}\gamma$. Embed a sample from $\Md_{0,2, 
				\bullet}(W)$
		in $(\bbH, \infty, 0, 1)$, then the law of the resulting field is $\frac\gamma{2(Q-\beta)^2}\LF_\bbH^{(\beta, 0), (\gamma, 1), (\beta, \infty)}$.
	\end{lemma}
	\begin{proof}
	    \cite[Proposition 2.18]{AHS21} obtains the field when embedding in the strip $\cS$, and Lemma~\ref{lem:lcft-H-strip} allows us to pass to $\bbH$. 
	\end{proof}

	Lemma~\ref{lem-disintegrate-3-pt} gives a concrete description of disintegrations of $\Md_{0, 2, \bullet}(W)$ via $\Md_{0,2}(W)$; it is essentially immediate from definitions, but we include the proof for completeness. We will use it frequently to pass between $\Md_{0,2}(W)$ and $\Md_{0,2, \bullet}(W)$.
	
	\begin{lemma}\label{lem-disintegrate-3-pt}
		Let $\Md_{0, 2, \bullet}(W; \ell', \ell'', r)$ denote the law of  a sample from $\Md_{0,2}(W; \ell'+\ell'', r)$ with a third marked point added to the left boundary at quantum length $\ell'$ from the first marked point. Then 
		\[\Md_{0,2, \bullet}(W) = \iiint_{\bbR_+^3} \Md_{0,2, \bullet}(W; \ell',\ell'', r) \, d\ell' \, d\ell'' \, dr. \]
		Similarly, 
		let $\Md_{0,2, \bullet}(W; \ell',\cdot, r)$ denote the law of  a sample from $\Md_{0,2}(W; \cdot, r)$ restricted to the event that the left boundary length is greater than $\ell'$, with a third marked point added to the left boundary at quantum length $\ell'$ from the first marked point. Then
		\[\Md_{0,2, \bullet}(W) = \iint_{\bbR_+^2} \Md_{0,2, \bullet}(W; \ell', \cdot, r) \, d\ell' \, dr. \]
	\end{lemma}
	\begin{proof}
		A sample from $\Md_{0,2, \bullet}(W)$ with the third marked point forgotten has law \\$\iint_0^\infty \Md_{0,2}(W; \ell, r) \, \ell d\ell\,dr$. The third marked point was sampled from the probability measure proportional to the boundary length measure on the left boundary arc; this is equivalent to sampling a uniform random variable on $[0,\ell]$ and marking the point at that quantum length:
		\[ \Md_{0,2, \bullet}(W) = \int_0^\infty \int_0^\infty \int_0^{\ell} \Md_{0,2, \bullet} (W; \ell', \ell - \ell', r)\, d\ell' \, d\ell \, dr.\]
		The change of variables $\ell'' := \ell - \ell'$ yields the first claim. The second claim follows from the first by integrating over  $\ell''>0$.
	\end{proof}
	
	We now identify quantum disks with a third marked point with certain quantum triangles.
	\begin{lemma}\label{lem-disk=qt}
	For $W > \frac{\gamma^2}2$, there is a constant $c = c(W)$ such that 
	\[\Md_{0,2, \bullet}(W) =  c(W) \QT(W, W, 2).\]	
	\end{lemma}
	\begin{proof}
		This follows immediately from Lemma~\ref{lem-embed-marked-disk} and Definition~\ref{def-qt-thick}. 
	\end{proof}
	We will state and prove a $W \in (0,\frac{\gamma^2}2)$ version of this statement later in Lemma~\ref{lem-disk=qt-thin}.

	\subsection{Welding of a quantum disk with a quantum triangle}\label{subsec:weld-qt}
	In this section, we recall the main result in~\cite{ASY22} and prove Theorem~\ref{thm:main} given Theorem~\ref{thm:w2w}.
	
	Given a quantum triangle of weights $W+W_1,W+W_2,W_3$ with $W_2+W_3 = W_1+2$ embedded as $(D,\phi,a_1,a_2,a_3)$, we start by making sense of the $\SLE_\kappa(W-2;W_1-2,W_2-W_1)$ curve $\eta$ from $a_2$ to $a_1$. If the domain $D$ is simply connected (which corresponds to the case where $W+W_1,W+W_2\ge\frac{\gamma^2}{2}$), $\eta$ is just the ordinary $\SLE_\kappa(W-2;W_1-2,W_2-W_1)$  with force points at $a_2^-,a_2^+$ and $a_3$. Otherwise, let $(\tilde{D}, \phi,\tilde{a}_1,\tilde{a}_2, \tilde{a}_3)$ be the thick quantum triangle component, and sample an $\SLE_\kappa(W-2;W_1-2,W_2-W_1)$ curve $\tilde{\eta}$ in $\tilde{D}$ from $\tilde{a}_2$ to $\tilde{a}_1$. Then our curve $\eta$ is the concatenation of $\tilde{\eta}$ with $\SLE_\kappa(W-2;W_1-2)$ curves in each bead of the weight $W+W_1$ (thin) quantum disk  and $\SLE_\kappa(W-2;W_2-2)$ curves in each bead of the weight $W+W_2$ (thin) quantum disk.
	
	With this notation, we state the welding of quantum disks with quantum triangles below.
	\begin{theorem}[Theorem 1.2 of \cite{ASY22}]\label{thm:disk+QT}
	Fix $W,W_1,W_2,W_3>0$ such that $W_2+W_3=W_1+2$ or $W_1+W_3=W_2+\gamma^2-2$. There exists some constant $c:=c_{W,W_1,W_2,W_3}\in (0,
	\infty)$ such that
	\begin{equation}\label{eqn:disk+QT}
	    \QT(W+W_1,W+W_2,W_3)\otimes \SLE_\kappa(W-2;W_2-2,W_1-W_2) = c\int_0^\infty \Md_{0,2}(W;\ell)\times\QT(W_1,W_2,W_3;\ell)d\ell.
	\end{equation}
	That is, if we draw an independent $\SLE_\kappa(W-2;W_1-2,W_2-W_1)$ curve $\eta$ on a sample from $\QT(W+W_1,W+W_2,W_3)$ embedded as $(D,\phi,a_1,a_2,a_3)$, then the quantum surfaces to the left and right of $\eta$ are corresponding independent quantum disks and triangles conditioned on having the same interface quantum length.
	\end{theorem}
	Using this result, Theorem~\ref{thm:main} follows directly from Theorem~\ref{thm:w2w}.
	
	\begin{figure}[ht]
		\centering
		\includegraphics[scale=0.59]{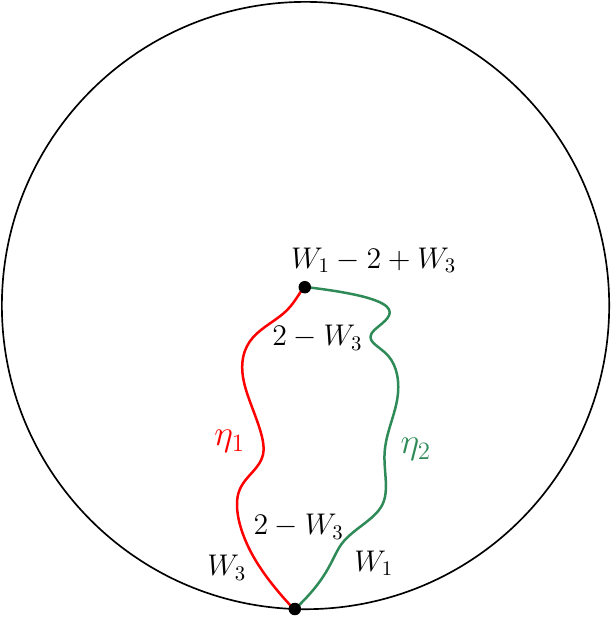}
		\caption{An illustration of the proof of Theorem~\ref{thm:main} given Theorem~\ref{thm:w2w} in the regime $2-W_1<W_3<2$.}\label{fig:radial-general}
	\end{figure}
	
	\begin{proof}[Proof of Theorem~\ref{thm:main} given Theorem~\ref{thm:w2w}]
	The $W_3=2$ case is done in Theorem~\ref{thm:w2w}. Now first suppose $2-W_1<W_3<2$.  We start with a sample from $\Md_{1,1}(W_1,W_1+2)$ embedded as $(\bbH,\phi,i,0)$ and run an independent radial $\SLE_\kappa(W-2)$ curve $\eta_2$ from 0 to $i$ with force point at $0^+$. Conditioned on $\eta_2$, we sample a chordal $\SLE_\kappa(W_3-2,W_1-W_3;-W_3)$ curve $\eta_1$ from $0$ to $i$ in $\bbH\backslash\eta_1$ with force points $0^-,0^+;a$ where $a$ is the point  other than $0^+$ which lies on $\eta_2$ and stays infinitesimally close to 0. When $\bbH\backslash\eta_2$ is not simply connected, the curve $\eta_1$ is understood as explained above Theorem~\ref{thm:disk+QT} and also in Proposition~\ref{prop:radial-flowline}.  Then by Theorem~\ref{thm:w2w} and Theorem~\ref{thm:disk+QT}, this splits $(\bbH,\phi,i,0)$ into an independent weight $2-W_3$ quantum disk and a weight $(W_1-2+W_3,W_3,W_1)$ quantum triangle conditioned on having the same interface length. That is, the curve-decorated surface  $(\bbH,\phi,i,0,\eta_1,\eta_2)/{\sim_\gamma}$ can be viewed as
	\begin{equation}
	    c\iint_{\bbR_+^2}\Md_{0,2}(2-W_3;\ell_1,\ell_2)\times\QT(W_1-2+W_3,W_1,W_3;\ell_2,\ell_1)\ d\ell_1\ d\ell_2
	\end{equation}
	where $\ell_1,\ell_2$ corresponds to the quantum lengths of $\eta_1$ and $\eta_2$. Now as we integrate over $\ell_2$ and forget about the interface $\eta_2$, by Theorem~\ref{thm:disk+QT} the surface $(\bbH\backslash\eta_1,\phi,i,0^+,0^-)$ can be seen as
	\begin{equation}\label{eqn:radial-qt}
	    c\int_0^\infty \QT(W_1,W_1-W_3+2,W_3;\ell_1,\ell_1)d\ell_1.
	\end{equation}
	On the other hand, the law of $(\eta_1,\eta_2)$ agrees with the one in Proposition~\ref{prop:radial-flowline} with  $\alpha = \frac{4-\gamma^2+W_1}{2\gamma}$, $\theta_2 = 0$, $\theta_1 = \frac{(2-W_3)\chi}{\lambda}$. Therefore by Proposition~\ref{prop:radial-flowline}, the marginal law of the interface $\eta_1$ is radial $\SLE_\kappa(W_3-2;W_1-W_3)$, which finishes this case since $W_3-2 = W_1-W_2$. 
    
    For $2-2W_1<W_3\le 2-W_1$, we can pick $\e>0$ small such that $W_3+W_1-\e>2-W_1$. Then by what we have proved, as we sample an independent radial $\SLE_\kappa(W_1+W_3-\e-2;\e-W_3)$ curve $\eta_1$ from 0 to $i$, the surface $(\bbH\backslash\eta_1,\phi,i,0^+,0^-)$ has law 
	\begin{equation*}
	    c\int_0^\infty \QT(W_1,2+\e-W_3,W_1+W_3-\e;\ell_1,\ell_1)d\ell_1.
	\end{equation*}
	Then we sample an chordal $\SLE_\kappa(W_3-2,\e-W_3;W_1-\e-2)$ curve $\eta_2$ on $\bbH\backslash\eta_1$ so that by Theorem~\ref{thm:disk+QT} the curve-decorated surface  $(\bbH,\phi,i,0,\eta_1,\eta_2)/\sim_\gamma$ can be viewed as
	\begin{equation}
	    c\iint_{\bbR_+^2}\Md_{0,2}(W_1-\e;\ell_1,\ell_2)\times\QT(\e,2+\e-W_3,W_3;\ell_2,\ell_1)\ d\ell_1\ d\ell_2.
	\end{equation}
	Again by integrating over $\ell_2$, the surface    $(\bbH\backslash\eta_1,\phi,i,0^+,0^-)$ has the law as~\eqref{eqn:radial-qt}, and the interface $\eta_1$ has marginal law radial $\SLE_\kappa(W_3-2;W_2-2)$. This proves Theorem~\ref{thm:main} for $2-2W_1<W_3\le 2-W_1$, and by an induction this extends to all $W_3<2$. The $W_3>2$ case is similarly proved; we omit the details.
	\end{proof}

	\section{Theorem~\ref{thm:w2w} for $W > \frac{\gamma^2}2$}\label{sec:pf-thick}
	In this section, we prove Theorem~\ref{thm:w2w} in the thick regime where $W>\frac{\gamma^2}{2}$. In Section~\ref{subsec:pre-zipper}, we recall the   reverse SLE processes  and the LCFT zipper. In Section~\ref{subsec:thick-field} we prove that the welding~\eqref{eqn:thm-main} holds for some random measure $\mathfrak{m}(W)$ on curves from 0 to $i$, while in Section~\ref{subsec:thick-curve} we show that the law $\mathfrak{m}(W)$ is radial $\SLE_\kappa(W-2)$.
	
	\subsection{The quantum zipper for Liouville CFT}\label{subsec:pre-zipper}
	In this section we briefly recall the reverse $\SLE_\kappa$ processes and the LCFT zipper in~\cite{Ang23}. 
	
	Recall the (forward) chordal $\SLE_{\kappa}(\underline\rho)$ processes defined in Section~\ref{subsec:pre-sle-ig}. The reverse $\SLE_\kappa$ processes is defined analogously. We say a compact set $K\subset\ol\bbH$ is a hull if $\bbH\backslash K$ is simply connected. Let $\tilde{\rho}_1,...,\tilde{\rho}_n\in\bbR$, $\tilde{x}_1,...,\tilde{x}_n\in\bar{\bbH}$. The SDE for the driving process $\tilde W_t$ of a reverse $\SLE_\kappa(\underline{\tilde{\rho}})$ process on compact hulls $(K_t)_{t\geq0}$ is given by
	\begin{equation}\label{eqn:def-reverse-sle-rho}
	\begin{split}
	&\tilde{W}_t = \sqrt{\kappa}B_t+\sum_{i=1}^n \int_0^t \mathrm{Re}\big(\frac{\tilde{\rho}_i}{\tilde{W}_s-\tilde{g}_s(\tilde{x}_i)}\big)ds; \\
	& \tilde{g}_t(z) =z-\int_0^t \frac{2}{\tilde{g}_t(z)-\tilde{W}_s}ds, \ z\in\ol\bbH.
	\end{split}
	\end{equation}
	Here $\tilde g_t$ maps $\bbH\backslash K_t$ to $\bbH$ such that $\lim_{|z|\to\infty}|\tilde g_t(z)-z|=0$. For $\kappa\in(0,4]$ and all $t>0$, $K_t$ is generated by a simple curve, i.e., there exists a curve $\eta_t$ such that $\bbH\backslash K_t$ equals the unbounded connected component of $\bbH\backslash\eta_t$. The existence and uniqueness of \eqref{eqn:def-reverse-sle-rho} has been established in \cite[Section 3.3.1]{DMS14}.

	Now, we state a special case of the quantum zipper for Liouville CFT. Let $\gamma \in (0,2)$. Let $\beta < Q$ and sample $\phi \sim \LF_\bbH^{(\beta, 0), (\beta, 1)}$ restricted to the event that $\{ \cL_{\phi}((-\infty,0)) > \cL_{\phi}((0,1))\}$. Let $s = \cL_{\phi}((0,1))$, and for each $u \in (0, s]$ let $p_u \in (-\infty, 0)$ and $q_u \in (0,1]$ satisfy $\cL_{\phi}((p_u,0)) = \cL_{\phi}((0, q_u)) = u$. Consider pairs $(\hat  \eta,  \hat  g)$ where $ \hat \eta:[0,s] \to \ol \bbH$ is a simple curve and $ \hat g: \bbH \to \bbH \backslash  \hat 
	 \eta$ is a conformal map. We call $( \hat \eta,  \hat  g)$ a \emph{conformal welding} if for each $u\in (0,s]$ we have $ \hat g(p_u) = \hat  g(q_u) = \hat \eta(u)$. Almost surely, $(\hat  \eta, \hat g)$ is unique up to conformal automorphisms of $\bbH$ \cite{She16a}, so we can fix a choice of $(\hat \eta, \hat g)$ by specifying that $\lim_{z \to \infty} \hat g(z)-z = 0$. Let $\psi := \hat g \bullet_\gamma \phi$ and  let $ \eta$ be the monotone reparametrization of $\hat \eta$ by half-plane capacity.
	
	\begin{proposition}\label{prop-quantum-zipper}
	In the above setting, the law of $(\psi, \eta)$ is 
	\[ 
	(2 \Im  \tilde g_\tau(0))^{\alpha^2/2} |\tilde g_\tau(0) - \tilde W_\tau|^{\alpha \beta'}
	\LF_\bbH^{(\alpha, \tilde g_\tau(0)),(\beta', \tilde W_\tau)}(d \psi) 1_{\tau < \infty} \mathrm{rSLE}_{\kappa,\tilde\rho+2,\tilde\rho}^\tau (d \eta),\]
	\[ \hfill \kappa = \gamma^2, \ \alpha = \frac{\beta}{2}+\frac{1}{\gamma}, \ \beta' =\beta-\frac{2}{\gamma},  \ \tilde\rho =  \gamma \beta,\]
	where $\mathrm{rSLE}_{\kappa,\tilde\rho+2,\tilde\rho}^\tau$ denotes the law of reverse $\SLE_\kappa$ with a weight $\tilde\rho+2$ force point located infinitesimally above 0 and a weight $\tilde\rho$ force point at $1$ run until the time $\tau$ that the force point hits the driving function, i.e., $\tilde g_\tau(1) = \tilde W_\tau$.
	\end{proposition}
	\begin{proof}
	This is a special case of \cite[Theorem 1.5]{Ang23}.
	\end{proof}

	\subsection{The zipping up and embedding}\label{subsec:thick-field}
	The goal of this section is to prove the following. 
	\begin{proposition}\label{prop:radialzipup}
	Fix $W>\frac{\gamma^2}{2}$. For a sample from $\Md_{1,1}(W,W+2)$ embedded as $(\bbH,\phi,i,0)$, there exists a constant $c:=c_W\in(0,\infty)$ and  $\sigma$-finite measure $\mathfrak{m}(W)$ on the space of continuous simple curves from $0$ to $i$ such that
	\begin{equation}\label{eqn:radial-zip-up}
	    \Md_{1,1}(W,W+2)\otimes\mathfrak{m}(W) = c\int_0^\infty \mathrm{Weld}(\QT(W,W,2;\ell,\ell))\, d\ell.
	\end{equation}
	\end{proposition}

In Lemma~\ref{lem-marginal-disk}, we show that when we restrict to a certain event for the Liouville field $\LF_\bbH^{(\beta, 0), (\beta, 1)}$ where $\beta = Q + \frac\gamma2 - \frac W\gamma$, then $(\bbH, \phi, 0, 1)/{\sim_\gamma}$ has the law of a weight $W$ quantum disk. On the other hand, using the quantum zipper, Lemma~\ref{lem-quantum-zipper} identifies the law of the conformally welded quantum surface as the left hand side of~\eqref{eqn:radial-zip-up}. Combining these gives Proposition~\ref{prop:radialzipup}.

	\begin{lemma}\label{lem-marginal-disk}
    Let $W > \frac{\gamma^2}2$, let $\beta = Q + \frac\gamma2 - \frac W\gamma < Q$ and let $I \subset (-\infty, 0)$ be a compact interval of positive length. 
	    Sample $\phi \sim \LF_\bbH^{(\beta, 0), (\beta, 1)}$ and restrict to the event that there is a point $p \in I$ satisfying $\cL_\phi((p, 0)) = \cL_{\phi}((0,1))$. 
        There is a constant $c = c(I) \in (0, \infty)$ such that the law of $(\bbH, \phi, 0, 1)/{\sim_\gamma}$ is $c \Md_{0,2}(W)$ restricted to the event that the left boundary length is greater than the right boundary length.  
	\end{lemma}
	
	To obtain Lemma~\ref{lem-marginal-disk}, we state and prove a variant for the strip $\cS$.
	
	\begin{lemma}\label{lem-marginal-disk-strip}
	Sample $\phi \sim \LF_{\cS}^{(\beta,\pm\infty)}$ and restrict to the event that $\cL_{\phi}(\bbR) > \cL_{\phi}(\bbR + i\pi)$. Let $p' \in \bbR$ be the point such that $\cL_\phi((p', +\infty)) = \cL_\phi(\bbR + i\pi)$. Suppose $I' \subset \bbR$ is a compact  interval of positive length. Then there is a constant $c' = c'(I') \in (0,\infty)$ such that, if we further restrict to the event that $p' \in I'$, the law of $(\cS, \phi, +\infty, -\infty)/{\sim_\gamma}$ is $c' \cM_{0,2}^\mathrm{disk}(W)$ restricted to the event that the left boundary length is greater than the right boundary length. 
	\end{lemma}
	\begin{proof}
	    The key input is the  \emph{uniform embedding} result from \cite[Theorem 2.13]{AHS21}: Let $M_W$ denote the law of the field $\psi$ constructed in Definition~\ref{def:thickdisk}, so for $\psi \sim M_W$ the law of $(\cS, \psi, +\infty, -\infty)/{\sim_\gamma}$ is $\cM_{0,2}^\mathrm{disk}(W)$. Suppose we sample $(\psi, T)\sim \frac{2(Q-\beta)^2}\gamma M_W \times \mathrm{Leb}_\bbR$ where $\mathrm{Leb}_\bbR$ denotes Lebesgue measure on $\bbR$, and set $\phi := \psi(\cdot + T)$. Then the law of $\phi$ is $\LF_\cS^{(\beta, \pm\infty)}$.
	
	In this setup, restrict to the event that $\cL_\phi(\bbR) > \cL_{\phi}(\bbR + i\pi)$, or equivalently $\cL_\psi(\bbR) > \cL_{\psi}(\bbR + i\pi)$. Let $q = q(\psi) \in \bbR$ be the point such that $\cL_\psi((q, +\infty)) = \cL_\psi(\bbR + i\pi)$. Then $\{p' \in I'\} = \{ q - T \in I'\}$.  For each $\psi$ the $\mathrm{Leb}_\bbR$-measure of $T$ satisfying $\{ q - T \in I'\}$ is $|I'|$. Thus, further restricting to $\{p' \in I'\}$, the marginal law of $\psi$ is $\frac{2(Q-\beta)^2}\gamma |I'| M_W$ restricted to $\{\cL_\psi(\bbR) > \cL_{\psi}(\bbR + i\pi) \}$, so the lemma holds with $c' = \frac{2(Q-\beta)^2}\gamma |I'|$.
	\end{proof}
	
	\begin{proof}[Proof of Lemma~\ref{lem-marginal-disk}]
	Lemma~\ref{lem-marginal-disk-strip} gives the analogous statement for $(\cS, +\infty, -\infty)$, and Lemmas~\ref{lem:lcft-H-conf-infty} and~\ref{lem:lcft-H-strip}  allow us to pass from $(\cS, +\infty, -\infty)$ to $(\bbH, 0, 1)$. 
	\end{proof}
	
	\begin{lemma} \label{lem-quantum-zipper}
	    Consider the setting of Lemma~\ref{lem-marginal-disk}. 
	    Conformal welding of the boundary arcs $(p,0), (0,1)$ of $(\bbH, \phi, 0, 1, p)/{\sim_\gamma}$ according to quantum length gives a curve-decorated quantum surface with a marked bulk and boundary point; embed it as $(\bbH, \psi', 
	    \eta', i, 0)$. The law of $(\psi', \eta')$ is 
	    \[\LF_\bbH^{(\alpha,i),(\beta', 0)}(d\psi')  \mathfrak m(d\eta') \qquad \text{ with } \ \ \alpha = \frac{\beta}{2}+\frac{1}{\gamma}, \ \  \beta' = \beta-\frac{2}{\gamma},\]
	    where $\mathfrak m$ is a $\sigma$-finite measure on the space of curves in $\ol\bbH$ from $0$ to $i$. 
	\end{lemma}
	\begin{proof}
        We first note that the law of $(\psi', \eta')$ is $\sigma$-finite; indeed, by Lemma~\ref{lem-marginal-disk} the law of $(\bbH \backslash \eta', \psi, i, 0^-)/{\sim_\gamma}$ is $c \Md_{0,2}(W)$, and $\Md_{0,2}(W)$ is $\sigma$-finite.
        
	    Now, define $(\psi, \eta)$ from $\phi$ via conformal welding as in Proposition~\ref{prop-quantum-zipper}. By Proposition~\ref{prop-quantum-zipper} the law of $(\psi, \eta)$ is 
	    \[ (2 \Im  \tilde g_\tau(0))^{\alpha^2/2} |\tilde g_\tau(0) - \tilde W_\tau|^{\alpha \beta'
	} \LF_\bbH^{(\alpha, \tilde g_\tau(0)),(\beta', \tilde W_\tau)}(d \psi) 1_{\tilde g_\tau^{-1}(0^-) \in I} 1_{\tau < \infty} \mathrm{rSLE}_{\kappa,\tilde\rho+2,\tilde\rho}^\tau (d \eta).\] 
	    Since $(\bbH, \psi', \eta', i, 0)/{\sim_\gamma} = (\bbH, \psi, \eta, \tilde g_\tau(0), \tilde g_\tau(1))/{\sim_\gamma}$, 
	    the claim follows from the conformal covariance of the Liouville field (Lemma~\ref{lem:lf-covariance}).
	\end{proof}
	
	\begin{proof}[Proof of Proposition~\ref{prop:radialzipup}]
	    We work in the setting of Lemma~\ref{lem-quantum-zipper}. That is, we fix a compact interval $I \subset (-\infty, 0)$ of positive length, sample $\phi\sim \LF_\bbH^{(\beta, 0), (\beta,1)}$ and restrict to the event that for some point $p \in I$ we have $\cL_\phi((p,0)) = \cL_\phi((0,1))$. By Lemma~\ref{lem-marginal-disk}, there is some constant $c$ such that the law of $(\bbH, \phi, 0, 1)/{\sim_\gamma}$ is $c \cM_{0,2}^\mathrm{disk}(W)$ restricted to the event that the left boundary length is greater than the right. Thus, by the definition of $\Md_{0,2,\bullet}(W; \ell', \ell'', r)$ in Lemma~\ref{lem-disintegrate-3-pt} the law of $(\bbH, \phi, 0, 1, p)/{\sim_\gamma}$ is $c\iint_{\bbR_+^2} \cM_{0,2, \bullet}^\mathrm{disk}(W; r, \ell'', r)\, dr \, d\ell'' = c\int_0^\infty \cM_{0,2, \bullet}^\mathrm{disk}(W; \ell, \cdot, \ell)\, d\ell$, and by Lemmas~\ref{lem-disintegrate-3-pt} and~\ref{lem-disk=qt}
	    this law agrees with $c'\int_0^\infty\QT(W,W,2;\ell,\ell)\,d\ell$ for some $c'$. On the other hand, conformally welding to obtain $(\psi', \eta')$ as in Lemma~\ref{lem-quantum-zipper}, the law of the welded surface is $\Md_{1,1}(W,W+2)\otimes\mathfrak{m}(W)$. Absorbing the constant $c'$ into $\mathfrak m (W)$ gives the claim. 
	\end{proof}

	\subsection{The interface law}\label{subsec:thick-curve}
	In this section, we conclude the proof of Theorem~\ref{thm:w2w} by proving that the probability measure $\mathfrak{m}(W)$ in~\eqref{eqn:radial-zip-up} is the same as radial $\SLE_\kappa(W-2)$. We first prove that the law of $\eta$ is $\mathrm{raSLE}_\kappa^\beta(W-2)$ (i.e.\ radial $\SLE^\beta_\kappa(W-2)$) for some (possibly random) $\beta$ via the resampling properties, and then show $\beta=0$ by comparing with the welding of a quantum disk into a quantum sphere.
	
	\begin{lemma}\label{lem:beta-law-1}
	There exists some {$\sigma$-finite} measure $\nu$ on $\bbR$, such that as equation of measures, $$\mathfrak{m}(W) = \int_\bbR \mathrm{raSLE}_\kappa^\beta(W-2)\nu(d\beta).$$ 
	\end{lemma}
	\begin{proof}
	For a quantum surface from $\Md_{1,1}(W,W+2)$ embedded as $(\bbH, \phi, i, 0)$, we sample an independent curve $\eta_1$ from $\mathfrak{m}$ and, conditioned on $(\phi,\eta_1)$, we sample an independent chordal $\SLE_\kappa(\frac{W}{2}-2,\frac{W}{2}-2)$ curve $\eta_2$ from $0^+$ to $i$ on $\bbH\backslash\eta_1$.
    By Proposition~\ref{prop:radialzipup}, the law of the curve-decorated quantum surface $(\bbH, \phi, i, 0, \eta_1, \eta_2)/{\sim_\gamma}$ is $c\int_0^\infty \QT(W,W,2;\ell_1,\ell_1)\otimes \SLE_\kappa(\frac{W}{2}-2,\frac{W}{2}-2)\, d\ell_1$ for some $c$. On the other hand, Theorem~\ref{thm:disk+QT} (with all weights equal to $\frac W2$ except $W_3 = 2$) gives
    \[\QT(W,W,2)\otimes \SLE_\kappa(\frac W2 -2, \frac W2-2) = c'\int_0^\infty \Md_{0,2}(\frac W2;\ell)\times\QT(\frac W2,\frac W2,2;\ell)d\ell;\]
    disintegrating on the two boundary lengths adjacent to the first vertex of $\QT(W,W,2)$ 
    yields for all $a,b>0$ 
    \[\QT(W,W,2; a,b)\otimes \SLE_\kappa(\frac W2 -2, \frac W2-2) = c'\int_0^\infty \Md_{0,2}(\frac W2;a,\ell)\times\QT(\frac W2,\frac W2,2;\ell, b)d\ell.\]
    Thus, setting $a = b = \ell_1$, the law of $(\bbH, \phi, i, 0, \eta_1, \eta_2)/{\sim_\gamma}$ is 
    $$cc'\iint_{\bbR_+^2}\Md_{0,2}(\frac{W}{2};\ell_1,\ell)\times\QT(\frac{W}{2},\frac{W}{2},2;\ell,\ell_1)\ d\ell_1\ d\ell.$$
    Hence $\eta_1$ is the interface arising from welding a weight $\frac{W}{2}$ quantum disk to the right side of a weight $(\frac{W}{2},\frac{W}{2},2)$ quantum triangle. By Theorem~\ref{thm:disk+QT}, given $(\phi, \eta_2)$ the curve $\eta_1$ is the chordal $\SLE_\kappa(\frac{W}{2}-2,0;\frac{W}{2}-2)$ on $\bbH\backslash\eta_2$.
    Therefore, by restricting to any event that depends only on $\phi$ and which has finite and nonzero $\Md_{1,1}(W, W+2)$-measure (such as $\{(\phi, f) \in I \}$, where $f$ is any smooth compactly supported function and $I$ any interval of finite and positive length), we see that the measure $\mathfrak{m}(W)$ is invariant under the Markov chain defined in Proposition~\ref{prop:radial-resampling} with\ $\alpha = \frac{4-\gamma^2+W}{2\gamma}$, $\theta_1 = 0$, $\theta_2 = -\frac{W\chi}{2\lambda}$. 
	Therefore the claim immediately follows from Proposition~\ref{prop:radial-resampling}.
	\end{proof}
	
	It remains to identify the value $\beta$. We are going to  prove that $\beta=0$ by \emph{zooming in} near the tip of the interface $\eta$.
	\begin{lemma}\label{lem:beta-law-2}
	The measure $\nu$ in Lemma~\ref{lem:beta-law-1} is a constant multiple of the Dirac measure $\delta_0$, i.e., the measure $\mathfrak{m}$ is the same as radial $\SLE_\kappa(W-2)$ {up to a multiplicative constant}.
	\end{lemma}
	
	\begin{proof}
	Let $f_0:\bbH\to\bbD$ be the conformal map with $f_0(i)=0$ and $f_0(0)=1$. Start with a sample from $(\bbH, \phi, i, 0)$ and let $\phi^0=\phi\circ f_0^{-1}+Q\log |(f_0^{-1})'|$. We work on the event that the additive constant $\mathbf{c}$ in $\phi$ lies in $[-1,1]$. For a constant $C>0$, consider the \emph{circle average embedding} of $\phi_0+\frac{C}{\gamma}$ as described in~\cite[Section 4.2]{DMS14}. That is, let $\tau_C = \sup\{r>0:\phi^0_r(0)+Q\log r+\frac{C}{\gamma}=0\}$ and $\phi_C = \phi^0(\tau_C z)+Q\log\tau_C$. Also let $\eta_C=\tau_C^{-1}f_0(\eta)$. Since the field $\phi^0$ is $h-\alpha\log|\cdot|$ plus a (bounded) random continuous function where $h$ is the GFF on $\bbD$, it follows from~\cite[Proposition 4.3.5]{DMS14} 
	that as we scale by letting $C\to\infty$, the quantum surfaces $(\bbC, \phi_C, 0, \infty)$ (viewed as a distribution) converge weakly to a weight $W$ \emph{quantum cone} $(\bbC, \phi_\infty, 0, \infty)$. Moreover, since $\tau_C\to 0$ almost surely, as in the construction in~\cite[Proposition 3.18]{ig4}, the law of the curve $\eta_C$ converges to that of a whole plane $\SLE_\kappa^\beta(W-2)$ process $\eta_\infty$ independent of $\phi_\infty$ running from $\infty$ to 0.
	
	On the other hand, let $f_\infty:\bbH\to\bbC\backslash\eta_\infty$ be a conformal map fixing 0 and $\infty$. Then as in the proof of \cite[Proposition 7.2.1]{DMS14},
	the law of $\tilde{\phi}_\infty:=\phi_\infty\circ f_\infty+Q\log|f_\infty|$ in a sufficiently (random) small neighborhood of 0 is the same as $h-\alpha\log|\cdot|$ and invariant under the operation of multiplying the quantum area by a constant. By~\cite[Proposition 4.2.4]{DMS14}
	$\tilde{\phi}_\infty$ has the law of a weight $W$ quantum wedge, and by welding its boundary arcs together we get a weight $W$ quantum cone decorated by an independent curve $\eta_\infty$. Therefore from~\cite[Theorem 1.2.4]{DMS14}, the law of $\eta_\infty$ is the whole plane $\SLE_\kappa(W-2)$ process from $0$ to $\infty$, and  by~\cite[Theorem 1.20]{ig4} its time reversal is the whole plane $\SLE_\kappa(W-2)$ process from $\infty$ to 0. Then as in~\cite[Section 5.3]{ig4}, the claim that $\nu(\{0\})=1$ immediately follows from~\cite[Proposition 5.8]{ig4}.    
	\end{proof}
	  
	\begin{proof}[Proof of Theorem~\ref{thm:w2w} for $W >\frac{\gamma^2}2$]
	The theorem follows immediately by combining Proposition~\ref{prop:radialzipup} with Lemma~\ref{lem:beta-law-2}.
	\end{proof}
	
	\section{Theorem~\ref{thm:w2w} for $W < \frac{\gamma^2}2$}\label{sec:pf-thin}
	
	In this section we will prove Theorem~\ref{thm:w2w} for $W < \frac{\gamma^2}2$.

	\begin{definition}\label{def-MW}
	Let $W \in (0, \frac{\gamma^2}2)$. 
	Sample a quantum disk $ \cD \sim \Md_{0,2}(W)$ and restrict to the event that the left boundary length of $ \cD$ is greater than its right boundary length, and conformally weld the entire right boundary arc to the initial segment of the left boundary arc starting from the first marked point. Embed the resulting quantum surface as $(\bbH, \phi, \eta, i, \infty)$ where $i$ and $\infty$ correspond to the first and second marked point of $\tilde \cD$ respectively, and the curve $\eta$ goes from $\infty$ to $i$. Let $M_W$ be the law of $(\phi, \eta)$. 
	\end{definition}
    We note that the conformal welding in Definition~\ref{def-MW} a.s.\ exists and is unique; indeed, \cite[Theorem 1.2.4]{DMS14} gives the a.s.\ existence and uniqueness for the conformal welding of a thin quantum wedge to itself, which implies the corresponding result for thin quantum disks since a thin quantum disk can be viewed as an initial subset of a thin quantum wedge (see  Definitions~\ref{def-thin-wedge} and~\ref{def-thin-disk}).
	
	The statement of Theorem~\ref{thm:w2w} for $W < \frac{\gamma^2}2$ then reduces to the claim that there exists a constant $C = C(W)$ such that for $(\phi, \eta)$ sampled from $M_W$,  the law of $(\bbD, \phi, 0, 1)/{\sim_\gamma}$ is $C \Md_{1,1}(W, W+2)$, and conditioned on $\phi$, the conditional law of $\eta$ is radial $\SLE_\kappa(W-2)$. 
	Using the resampling properties of the Liouville field discussed in Section~\ref{sec-prelim-resampling}, we will establish resampling properties of $M_W$. Note that since $M_W$ is an infinite measure, the conditional laws described below are understood in the sense of Definition~\ref{def-markov-kernel}.

	\begin{proposition}\label{prop-resample-inside}
	For $(\phi, \eta) \sim M_W$, the conditional law of $\eta$ given $\phi$ is radial $\SLE_{\kappa}(W-2)$ in $\bbH$ from $\infty$ to $i$ with force point infinitesimally counterclockwise of $\infty$.  Furthermore, consider a bounded neighborhood $A\subset \bbH$ of $i$ such that $\ol A \cap \partial \bbH = \emptyset$, and let $\alpha = Q - \frac W{2\gamma}$. Conditioned on $\phi |_{\bbH \backslash A}$, we have $\phi|_A \stackrel d= h + \mathfrak h + \alpha G_A(\cdot, i)$, where $\mathfrak h$ is the harmonic extension of $\phi|_{\bbH \backslash A}$ to $A$, $h$ is a zero boundary GFF on $A$ and $G_A$ is the Green function of $h$. 
	\end{proposition}

	\begin{proposition}\label{prop-resample-1}
	    Let $A \subset \bbH$ be a neighborhood of $\infty$ not containing $i$ such that $\bbH \backslash A$ is simply connected and  $\partial \bbH \backslash \ol A$ contains an interval,  and let $\beta = \gamma - \frac W\gamma$. For $(\phi, \eta) \sim M_W$, conditioned on $\phi|_{\bbH \backslash A}$, we have $\phi|_A \stackrel d= h + \mathfrak h + (\frac\beta2 - Q)G_A(\cdot, \infty)$, where $\mathfrak h$ is the harmonic extension of $\phi|_{\bbH \backslash A}$ to $A$ which has normal derivative zero on $\partial A \cap \partial \bbH$,
	     $h$ is a mixed boundary GFF on $A$ with zero boundary conditions on $\partial A \cap \bbH$ and free boundary conditions on $\partial A \cap \partial \bbH$, and $G_A$ is the Green function of $h$. 
	\end{proposition}

	We recall some decompositions of thin quantum disks in Section~\ref{sec-decomp}. In Section~\ref{subsec-quantum-cones} we introduce the \emph{quantum cone} and its uniform embedding, and use it to deduce Proposition~\ref{prop-resample-inside}. In Section~\ref{sec-resample-rest}, as a warm-up we prove Lemma~\ref{lem-resample-bdy}, a variant of Proposition~\ref{prop-resample-1} where we resample $\phi$ away from  both $\infty$ and $i$. The proof uses the Poissonian description of the thin quantum disk and its relationship with the Liouville field. In Section~\ref{sec-resample-infty} we use similar ideas to prove Proposition~\ref{prop-resample-1}. Finally, we combine Propositions~\ref{prop-resample-inside} and~\ref{prop-resample-1} to prove Theorem~\ref{thm:main} for $W < \frac{\gamma^2}2$ in Section~\ref{sec-proof-main-thin}.

	\subsection{Decompositions of thin quantum disks and wedges} \label{sec-decomp}
	
	The Poissonian structures of thin quantum disks and wedges give them decompositions into independent pieces. 
	
	\begin{lemma}\label{lem-cut-disk}
	    Let $W <\frac{\gamma^2}2$. 
	    For $\cD$ sampled from $\Md_{0,2}(W)$, let $\mathcal{QCP}_\cD$ denote the quantum cut point measure on $\cD$. For $(p, \cD)$ sampled from $\mathcal{QCP}_\cD \Md_{0,2}(W)$, let $\cD_1$ (resp.\ $\cD_2$) be the ordered subset of $\cD$ of quantum surfaces preceding (resp.\ succeeding) $p$, i.e., $p$ cuts $\cD$ into $\cD_1$ and $\cD_2$. Then the law of $(\cD_1, \cD_2)$ is $(1-\frac2{\gamma^2} W)^{2}\Md_{0,2}(W) \times \Md_{0,2}(W)$. 
	\end{lemma}
	\begin{proof}
	    If $(T_1, T) \sim \mathrm{Leb}_{[0,T]}(dT_1)\, \mathrm{Leb}_{\bbR_+}(dT)$ and $T_2 := T - T_1$, then the law of $(T_1, T_2)$ is $\mathrm{Leb}_{\bbR_+}(dT_1)\, \mathrm{Leb}_{\bbR_+}(dT_2)$; indeed, the Jacobian of the map $(T_1, T) \mapsto (T_1, T_2)$ has unit determinant. 
	    The lemma follows from the above and Definition~\ref{def-thin-disk}.
	\end{proof}
	
	Recall $\cM_{0,2, \bullet}(W)$ as defined in Section~\ref{subsec-third-point}.
	The following decompositions were shown in \cite{AHS20}; their proofs use the idea of  Lemma~\ref{lem-cut-disk} together with Palm's theorem for Poisson point processes. 
	\begin{lemma}[{\cite[Proposition 4.4]{AHS20}}]\label{lem-disk-decomp}
	    Let $W < \frac{\gamma^2}2$. 
	    For a sample $\cD$ from $\Md_{0,2, \bullet}(W)$, let $\cD_2$ be the connected component of $\cD$ having three marked points, and let $\cD_1$ (resp.\ $\cD_3$) be the ordered subset of $\cD$ preceding (resp.\ succeeding) $\cD_2$. Then the law of $(\cD_1, \cD_2, \cD_3)$ is $(1-\frac2{\gamma^2}W)^2\Md_{0,2}(W) \times \Md_{0,2,\bullet}(\gamma^2-W) \times \Md_{0,2}(W)$. 
	\end{lemma}
	
	\begin{lemma}[{\cite[Proposition 4.2]{AHS20}}] \label{lem-wedge-decomp}
	    Let $W < \frac{\gamma^2}2$. For $(p, \cW)$ sampled from $\cL_\cW^\mathrm{right} \cM^\mathrm{wedge}(W)$, let $\cW_\bullet$ be $\cW$ with $p$ added as the third marked point. 
	    Let $\cD_2$ be the connected component of $\cW_\bullet$ having three marked points, and let $\cD_1$ (resp.\ $\cW'$) be the ordered subset of $\cW_\bullet$ preceding (resp.\ succeeding) $\cD_2$. Then the law of $(\cD_1, \cD_2, \cW')$ is $(1-\frac2{\gamma^2}W)^2\Md_{0,2}(W) \times \Md_{0,2,\bullet}(\gamma^2-W) \times \cM^\mathrm{wedge}(W)$. 
	\end{lemma}
	
	Now we state a variant of 
	Lemma~\ref{lem-wedge-decomp} where instead of adding a boundary point we add a bulk point; the proof is identical. For a quantum surface $\cD$, let $\cA_\cD$ denote the quantum area measure on $\cD$. Let $\cM^\mathrm{disk}_{1,2}(W)$ be the law of a three-pointed quantum surface obtained from $(p, \cD) \sim  \cA_\cD(dp) \Md_{0,2}(W)(d\cD)$ by adding to $\cD$ the third marked point $p$. In other words, sample $\cD$ from the weighted measure $|\cA_\cD|\Md_{0,2}(W)(d\cD)$ (so the weighting is given by the total quantum area), and conditioned on $\cD$ sample a point $p$ from the probability measure proportional to $\cA_\cD$; let $\Md_{1,2}(W)$ be the law of the quantum surface obtained by adding $p$ to $\cD$.
	
	\begin{lemma}\label{lem-wedge-decomp-bulk}
	    Let $W < \frac{\gamma^2}2$.
	    For $(p, \cW)$ sampled from $\cA_\cW \cM^\mathrm{wedge}(W)$, let $\cW_\bullet$ be $\cW$ with $p$ added as the third marked point.  Let 
	    $\cD_2$ be the connected component of $\cW_\bullet$ having three marked points, and let $\cD_1$ (resp.\ $\cW'$) be the ordered subset of $\cW_\bullet$ preceding (resp.\ succeeding) $\cD_2$. Then the law of $(\cD_1, \cD_2, \cW')$ is $(1-\frac2{\gamma^2}W)^2\Md_{0,2}(W) \times \Md_{1,2}(\gamma^2-W) \times \cM^\mathrm{wedge}(W)$. 
	\end{lemma}
	
	Finally, we give a thin version of Lemma~\ref{lem-disk=qt}. 
	
		\begin{lemma}\label{lem-disk=qt-thin}
			For $W \in (0, \frac{\gamma^2}2$), there is a constant $c = c(W)$ such that 
			\[\Md_{0,2, \bullet}(W) =  c(W) \QT(W, W, 2).\]	
		\end{lemma}
		\begin{proof}
			Lemma~\ref{lem-disk=qt} identifies $\Md_{0, 2,\bullet}(\gamma^2 - W)$ with $\QT(\gamma^2-W, \gamma^2-W, 2)$, so the claim follows immediately from Definition~\ref{def-qt-thin} and  Lemma~\ref{lem-disk-decomp}.
		\end{proof}

	\subsection{Resampling a neighborhood of $i$}\label{subsec-quantum-cones}
	
	In this section, we prove Proposition~\ref{prop-resample-inside}. The starting point of our proof is an infinite-volume quantum surface called the quantum cone, which we now define. 
	
	Let $\cC = \bbR \times [0,2\pi]/{\sim}$ be the horizontal cylinder, where we identify $(t, 0) \sim (t, 2\pi)$ for all $t \in \bbR$. Let $h_\cC$ be the GFF on $\cC$ normalized to have mean zero on $\{0\} \times [0,2\pi]/{\sim}$. As in the case of the horizontal strip, we can decompose $h_\cC = h^\mathds{1}_\cC + h^2_\cC$ where $h^\mathds{1}_\cC$ is constant on each line segment $\{t \} \times [0,2\pi]$ for $t \in \bbR$, while $h^2_\cC$ has mean zero on all such line segments. 
	\begin{definition}\label{def-quantum-cone}
	Fix $W > 0$ and let $\alpha = Q - \frac{W}{2\gamma}$. Sample independent standard Brownian motions $(B_t)_{t \geq 0}, (\tilde B_t)_{t \geq 0}$ conditioned on $\tilde B_{t} - (Q-\alpha) < 0$ for all $t>0$. Let
	\[Y_t := \left\{ \begin{array}{rcl} 
				B_{t}+(Q-\alpha) t & \mbox{for} & t\ge 0\\
				\tilde{B}_{-t} -(Q-\alpha) |t| & \mbox{for} & t<0
			\end{array} 
			\right. \]
	Let $\hat \psi = \psi_1 + \psi_2$ where $\psi_1$ is the distribution on $\cC$ which is constant on each line segment $\{t\} \times [0,2\pi]$ satisfying $\psi_1(z) = Y_{\mathrm{Re}(z)}$, and $\psi_2$ is an independent copy of $h_\cC^2$. The \emph{weight $W$ quantum cone} is $(\cC, \hat \psi, -\infty, +\infty)/{\sim_\gamma}$, and we define the probability measure $\cM^\mathrm{cone}(W)$ to be its law.
	\end{definition}
	
	The following variant of \cite[Proposition 5.1]{Ang23} states that the uniform embedding of a quantum cone is the Liouville field.
	\begin{proposition}\label{prop-cone-unif-embed}
	Let $W>0$ and $\alpha = Q - \frac W{2\gamma}$. Sample $(\cV, T, \theta)$ from $\cM^\mathrm{cone}(W) \times dT \times \mathds{1}_{\theta \in [0,2\pi)}\frac1{2\pi} d\theta$ and let $(\bbC, \psi_0, 0, \infty)$ be an embedding of $\cV$ chosen in a way not depending on $(T, \theta)$. Let $f_{T, \theta}(z) = e^{T+i\theta}z$. Then the law of $\psi = f_T \bullet_\gamma \psi_0$ is $(Q -\alpha)^{-1} \LF_{\bbC}^{(\alpha, 0), (2Q-\alpha, \infty)}$.
	\end{proposition}
	\begin{proof}
	We will prove the result in the case where $\psi_0 = \exp \bullet_\gamma \hat \psi$, with $\hat \psi$ as in Definition~\ref{def-quantum-cone}. The result for arbitrary $\psi_0$ then follows since the measure $dT \times \mathds{1}_{\theta \in [0,2\pi)} \frac1{2\pi} d\theta$ induces a Haar measure on the space of conformal automorphisms of the complex plane $\bbC$.

	We first claim that the random processes $X_1, X_2$ defined below have the same law:
	\begin{itemize}
	    \item Let $P$ be the law of $Y$ defined in Proposition~\ref{def-quantum-cone}. Sample $(Y,T) \sim P \times dT$ and define $X_1(t) = Y_{t-T}$ for all $t \in \bbR$.
	    \item Let $P'$ be the law of standard two-sided Brownian motion $B$. Sample $(B, \mathbf c)\sim (Q-\alpha)^{-1} P' \times dc$ and let $X_2(t) = B_t + (Q-\alpha)t + \mathbf c$ for all $t \in \bbR$. 
	\end{itemize}
	The proof is identical to that of \cite[Proposition 5.2]{Ang23}. Given this claim, the proposition follows by adding an independent copy of $h_\cC^2$ to obtain an identity of fields on $\cC$, namely, $X_1(\mathrm{Re}\, \cdot) + h^2_\cC \stackrel d= X_2(\mathrm{Re} \,\cdot) + h^2_\cC$, then parametrizing in $\bbC$ via the map $z \mapsto e^z$; the first field agrees in law with $\psi$, and the second field has law $(Q -\alpha)^{-1} \LF_{\bbC}^{(\alpha, 0), (2Q-\alpha, \infty)}$. See the proof of \cite[Proposition 5.1]{Ang23} for details.
	\end{proof}
	
	As we now see, the quantum cone with extra marked point is described by a Liouville field. For $W > 0$, let $P_W$ be the law of a field $\psi_0$ such that $(\bbC, \psi_0, 0, \infty)/{\sim_\gamma}$ is a weight $W$ quantum cone (e.g., let $P_W$ be the law of $\exp \bullet_\gamma \hat \psi$ where $\hat \psi$ is as sampled in Definition~\ref{def-quantum-cone}). The exact choice of $P_W$ does not matter in Proposition~\ref{prop-pointed-cone}, since we will re-embed in a quantum intrinsic way anyway.
	
	\begin{proposition}\label{prop-pointed-cone}
	Let $W >0$. Sample $(z, \psi_0) \sim \cA_{\psi_0}(dz) P_W(d\psi_0)$ and let $\psi$ be the field on $\bbC$ such that $(\bbC, \psi, 0, 1, \infty)/{\sim_\gamma} = (\bbC, \psi_0, 0, z, \infty)/{\sim_\gamma}$. Then the law of $\psi$ is $\frac1{2\pi(Q-\alpha)}\LF_{\bbC}^{(\alpha, 0), (\gamma, 1), (2Q-\alpha, \infty)}$, where $\alpha = Q - \frac W{2\gamma}$.
	\end{proposition}
	\begin{proof}
	The proof is identical to that of \cite[Proposition 2.18]{AHS21}, using Proposition~\ref{prop-cone-unif-embed} in place of \cite[Proposition 2.13]{AHS21}.
	\end{proof}
	
	On the other hand, the quantum cone with extra marked point cut by an SLE curve has the following useful decomposition. 
	
	\begin{lemma}\label{lem-decomp-cone}
	    Let $W \in (0, \frac{\gamma^2}2)$. Let $\SLE_\kappa^\mathrm{wp}(\rho)$ denote the law of whole-plane $\SLE_\kappa(\rho)$ from $0$ to $\infty$. Sample $(z, \psi_0, \eta_0^\mathrm{wp}) \sim \cA_{\psi_0}(dz) P_W(d\psi_0) \SLE_\kappa^\mathrm{wp}(W-2)$. For each connected component $B$ of $\bbC \backslash \eta_0^\mathrm{wp}$, let $p_B$ and $q_B$ be the first and last points of $\partial B$ hit by $\eta_0^\mathrm{wp}$. Let $\tilde \cW$ be the collection of all quantum surfaces $(B, \psi_0, p_B, q_B)/{\sim_\gamma}$, and  endow $\tilde \cW$ with the ordering where $(B_1, \psi_0, p_{B_1}, q_{B_1})/{\sim_\gamma}$ precedes $(B_2, \psi_0, p_{B_2}, q_{B_2})/{\sim_\gamma}$ if $\eta_0^\mathrm{wp}$ hits $q_{B_1}$ for the last time before it hits $q_{B_2}$ for the last time.

	    Let $B_\bullet$ be the connected component of $\bbC \backslash \eta_0^\mathrm{wp}$ containing $z$, and let $\cB_\bullet = (B_\bullet, \psi_0, z, p_{B_\bullet}, q_{B_\bullet})/{\sim_\gamma}$. Let $ \cD \subset \tilde \cW$ (resp.\ $\cW \subset \tilde \cW$) be the ordered collection of elements that come before (resp.\ after) $(B_\bullet, \psi_0, p_{B_\bullet}, q_{B_\bullet})/{\sim_\gamma}$. Then the joint law of $(\cD, \cB_\bullet, \cW)$ is $(1 - \frac{2W}{\gamma^2})^2 \Md_{0,2}(W) \times \Md_{1,2}(\gamma^2 - W)  \times \cM^\mathrm{wed}(W)$.
	\end{lemma}
	\begin{proof}
	    This follows from two facts. First, if we had instead sampled $(\psi_0, \eta_0^\mathrm{wp}) \sim P_W \SLE_\kappa^\mathrm{wp}(W-2)$ (that is, did not sample $z$), then $\tilde \cW$ would have the law of a weight $W$ quantum wedge \cite[Theorem 1.2.4]{DMS14}. Second, if we add to $\tilde \cW$ a marked point from quantum area measure, we can decompose it as in Lemma~\ref{lem-wedge-decomp-bulk}. 
	\end{proof}
	
	Combining the previous two claims gives the following description when a certain Liouville field is cut by independent whole-plane $\SLE$. See Figure~\ref{fig-sample-inside} (left).
	
	\begin{lemma}\label{lem-LF-SLE}
	Let $W \in (0, \frac{\gamma^2}2)$ and $\alpha = Q - \frac W{2\gamma}$. 
	    Sample $(\psi, \eta^\mathrm{wp})$ from $\frac1{2\pi(Q-\alpha)} \LF_{\mathbb C}^{(\alpha, 0), (\gamma, 1), (2Q-\alpha, \infty)} \times \SLE_\kappa^\mathrm{wp}(W-2)$. For each connected component $B$ of $\bbC \backslash \eta^\mathrm{wp}$, let $p_B$ and $q_B$ be the first and last points of $\partial B$ hit by $\eta^\mathrm{wp}$. Let $\tilde \cW$ be the collection of all quantum surfaces $(B, \psi, p_B, q_B)/{\sim_\gamma}$, and  endow $\tilde \cW$ with the ordering where $(B_1, \psi, p_{B_1}, q_{B_1})/{\sim_\gamma}$ precedes $(B_2, \psi, p_{B_2}, q_{B_2})/{\sim_\gamma}$ if $\eta^\mathrm{wp}$ hits $q_{B_1}$ for the last time before it hits $q_{B_2}$ for the last time.

	    Let $B_\bullet$ be the connected component of $\bbC \backslash \eta_0^\mathrm{wp}$ containing $1$, and let $\cB_\bullet = (B_\bullet, \psi, 1, p_{B_\bullet}, q_{B_\bullet})/{\sim_\gamma}$. Let $ \cD \subset \tilde \cW$ (resp.\ $\cW \subset \tilde \cW$) be the ordered collection of elements that come before (resp.\ after) $(B_\bullet, \psi, p_{B_\bullet}, q_{B_\bullet})/{\sim_\gamma}$. Then the joint law of $(\cD, \cB_\bullet, \cW)$ is $(1 - \frac{2W}{\gamma^2})^2 \Md_{0,2}(W) \times \Md_{1,2}(\gamma^2 - W) \times \cM^\mathrm{wed}(W)$.
	\end{lemma}
	
	\begin{figure}[ht]
		\centering
		\includegraphics[scale=0.42]{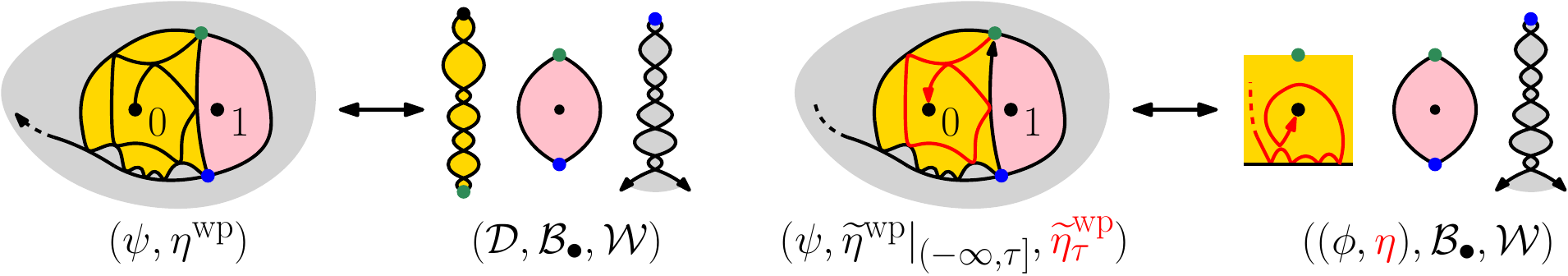}
		\caption{\textbf{Left:} Figure for Lemma~\ref{lem-LF-SLE}. The pair $(\psi, \eta^\mathrm{wp})$ are a Liouville field and independent whole-plane $\SLE_\kappa(W-2)$ curve; cutting along the curve gives three independent quantum surfaces $(\cD, \cB_\bullet, \cW)$ with explicit laws. \textbf{Right:} Proof of Proposition~\ref{prop-resample-inside}. In the left setup, let $\tilde \eta^\mathrm{wp}$ be the time-reversal of $\eta^\mathrm{wp}$. Let $\tau$ be the time it disconnects $1$ from $0$; splitting $\tilde \eta^\mathrm{wp}$ at time $\tau$ gives $\tilde \eta^\mathrm{wp}|_{(-\infty, \tau]}$ and $\tilde \eta^\mathrm{wp}_\tau$ (black and red). Restrict to the event $E$ that for the curve $\tilde \eta^\mathrm{wp}$, at time $\tau$ the force point is infinitesimally counterclockwise of the curve tip. From~\eqref{eq-resample-inside-2}, the triple $(\phi, \eta), \cB_\bullet, \cW$ are independent and the law of $(\phi, \eta)$ is described by $M_W$. Thus, one can use the resampling properties of $(\psi, \tilde \eta^\mathrm{wp})$ to prove the desired resampling properties of $M_W$.}\label{fig-sample-inside}
	\end{figure}
	
	Using Lemma~\ref{lem-LF-SLE}, we can now deduce Proposition~\ref{prop-resample-inside}.

	\begin{proof}[Proof of Proposition~\ref{prop-resample-inside}]
	As in the setting of Lemma~\ref{lem-LF-SLE}, sample $(\psi, \eta^\mathrm{wp})$ from \\$\frac1{2\pi(Q-\alpha)} \LF_{\mathbb C}^{(\alpha, 0), (\gamma, 1), (2Q-\alpha, \infty)} \times \SLE_\kappa^\mathrm{wp}(W-2)$ and define $\cD, \cB_\bullet, \cW$. By Lemma~\ref{lem-LF-SLE} the  law of $(\cD, \cB_\bullet, \cW)$ is $(1 - \frac{2W}{\gamma^2})^2 \Md_{0,2}(W) \times \Md_{1,2}(\gamma^2 - W) \times \cM^\mathrm{wed}(W)$.
	
	See Figure~\ref{fig-sample-inside} (right). 
	Let $\tilde \eta^\mathrm{wp}$ be the time-reversal of $\eta^\mathrm{wp}$, so by reversibility of whole-plane $\SLE_\kappa(W-2)$ \cite[Theorem 1.20]{ig4} the conditional law of $\tilde \eta^\mathrm{wp}$ given $\psi$ is whole-plane $\SLE_\kappa(W-2)$ from $\infty$ to $0$. For concreteness, parametrize\footnote{The exact choice of parametrization is unimportant, so long as at each point in time, the parameter of the curve is a function of the curve up until that point.} $\tilde \eta^\mathrm{wp}$ such that $\mathrm{CR}(\eta((-\infty, t]), 0) = e^{-t}$. 
	
	Let $\tau$ be the time that $\tilde \eta^\mathrm{wp}$ separates $0$ and $1$, i.e., the first time $t$ that $0$ and $1$ lie in different connected components of  $\bbC \backslash\tilde\eta^\mathrm{wp}((-\infty, t])$. 
	Let $D_\tau$ be the connected component of $\bbC \backslash \tilde\eta^\mathrm{wp}((-\infty, \tau])$ containing $0$, and let $\tilde \eta^\mathrm{wp}_\tau$ be the curve $\tilde\eta^\mathrm{wp}(\cdot + \tau)|_{[0,\infty)}$ reparametrized according to quantum length. Let $(\phi, \eta)$ be the field and curve on $\bbH$ such that $(\bbH, \phi, \eta, i, \infty)/{\sim_\gamma} = (D_\tau, \psi, \tilde \eta_\tau^\mathrm{wp}, 0, \tilde \eta^\mathrm{wp}(\tau))/{\sim_\gamma}$.

	Let $E$ be the event that at time $\tau$, the force point for $\eta$ lies infinitesimally counterclockwise of $\tilde\eta^\mathrm{wp}(\tau)$ on $\partial D_\tau$. Since the law of $( \cD, \cB_\bullet, \cW)$ is $(1 - \frac{2W}{\gamma^2})^2 \Md_{0,2}(W) \times \Md_{1,2}(\gamma^2-W) \times \cM^\mathrm{wed}(W)$, and the event $E$ corresponds to the left boundary length of $\cD$ being larger than the right boundary length, by the definition of $M_W$ the law of $((\phi, \eta), \cB_\bullet, \cW)$ restricted to $E$ is 
	\eqb\label{eq-resample-inside-2}
	(1 - \frac{2W}{\gamma^2})^2 M_W \times \Md_{1,2}(\gamma^2-W) \times \cM^\mathrm{wed}(W).
	\eqe
	
	We claim that on $E$, 
	\eqb \label{eq-resample-inside}
	\sigma ((\phi, \cB_\bullet, \cW)) = \sigma((\psi, \tilde \eta^\mathrm{wp}|_{(-\infty, \tau]})),
	\eqe 
	that is, the tuples $(\phi, \cB_\bullet, \cW)$ and $(\psi, \tilde \eta^\mathrm{wp}|_{(-\infty, \tau]})$ each determine the other. Indeed, starting with the former, one can obtain the latter by conformal welding of $(\bbH, \phi, i, \infty)/{\sim_\gamma}, \cB_\bullet, \cW$, then embedding the resulting quantum surface with marked points sent to $(0, 1, \infty)$. Conversely, starting with the latter the former can be obtained by cutting along $\tilde \eta^\mathrm{wp}|_{(-\infty, \tau]}$. 
	
	By the Markov property for whole-plane $\SLE_\kappa(W-2)$, conditioned on $E$ and $\sigma((\psi, \tilde \eta^\mathrm{wp}|_{(-\infty, \tau]}))$, the conditional law of $\tilde \eta^\mathrm{wp}_\tau$ is radial $\SLE_\kappa(W-2)$ in $(D_\tau, 0, \tilde \eta^\mathrm{wp}(\tau))$ with force point infinitesimally counterclockwise of $\tilde \eta^\mathrm{wp}(\tau)$. By~\eqref{eq-resample-inside}, equivalently,  the conditional law of $\eta$ given $E$ and $\sigma(\phi, \cB_\bullet, \cW)$ is radial $\SLE_\kappa(W-2)$ in $(\bbH, i, \infty)$ with force point infinitesimally counterclockwise of $\infty$.  Combining with~\eqref{eq-resample-inside-2} gives the first claim. 
	
	For the second claim, let $f: \bbD \to D_\tau$ be the conformal map satisfying $f(0)=0$ and $f(1) = \tilde \eta^\mathrm{wp}(\tau)$. By the same argument used to obtain~\eqref{eq-resample-inside}, on $E$ we have 
	\eqb\label{eq-resample-inside-3}
	\sigma ((\phi|_{\bbH \backslash A}, \cB_\bullet, \cW)) = \sigma((\psi|_{\bbC \backslash f(A)}, \tilde \eta^\mathrm{wp}|_{(-\infty, \tau]}));
	\eqe
	note that in the argument showing that $(\phi|_{\bbH \backslash A}, \cB_\bullet, \cW)$ determines $(\psi|_{\bbC \backslash f(A)}, \tilde \eta^\mathrm{wp}|_{(-\infty, \tau]})$, the conformal welding can be carried out because $\phi|_{\bbH \backslash A}$ determines the quantum boundary length measure. 
	By Lemma~\ref{lem-markov-LF-C}, conditioned on $\sigma((\psi|_{\bbC \backslash f(A)}, \tilde \eta^\mathrm{wp}|_{(-\infty, \tau]}))$, we have $\psi|_{f(A)} \stackrel d= h_0 + \mathfrak h_0 + \alpha G_{f(A)}(\cdot, 0)$, where $h_0$ is a zero boundary GFF on $f(A)$, $\mathfrak h_0$ is the harmonic extension of $\psi|_{\bbC \backslash f(A)}$ to $f(A)$, and $G_{f(A)}$ is the Green function of $h_0$. By conformal invariance and~\eqref{eq-resample-inside-3}, we conclude that conditioned on $\sigma((\phi|_{\bbH \backslash A}, \cB_\bullet, \cW))$, we have $\phi|_A \stackrel d= h + \mathfrak h + \alpha G_A(\cdot, i)$. 
	The second claim thus follows from~\eqref{eq-resample-inside-2}.
	\end{proof}

	\subsection{Resampling away from $i$ and $\infty$}\label{sec-resample-rest}
	
	In this section we state and prove Lemma~\ref{lem-resample-bdy}, which is a weaker version of Proposition~\ref{prop-resample-1}; it is weaker because it does not allow us to resample a neighborhood of $\infty$. This argument is a warm-up for the proof of Proposition~\ref{prop-resample-1}, which is essentially the same except with several complications. 
	
	\begin{lemma}\label{lem-resample-bdy}
	    Consider a bounded open set $A \subset \bbH$ not containing $i$ such that $\ol A \cap \partial \bbH$ is an interval. For $(\phi, \eta) \sim M_W$, conditioned on $\phi |_{\bbH \backslash A}$, we have $\phi|_A \stackrel d= h + \mathfrak h$, where $\mathfrak h$ is the harmonic extension of $\phi|_{\bbH \backslash A}$ to $A$ which has normal derivative zero on $\partial A \cap \partial \bbH$, and 
	     $h$ is a mixed boundary GFF on $A$ with zero boundary conditions on $\partial A \cap \bbH$ and free boundary conditions on $\partial A \cap \partial \bbH$.
	\end{lemma}
	
	Our arguments use quantum disks with one or more specified boundary lengths, defined in terms of \emph{disintegration}. Recall from~\eqref{eq-disintegrate-disk} that $\Md_{0,2}(W) = \iint_0^\infty \Md_{0,2}(W; \ell, r)\,d\ell\,dr$, and each $\Md_{0,2}(W; \ell, r)$ is supported on the space of quantum surfaces whose left and right boundary lengths are $(\ell, r)$.  
	
	Lemma~\ref{lem-resample-infty-1} below gives a decomposition of the surfaces obtained by cutting $(\phi, \eta)\sim M_W$ along $\eta$; 
	see Figure~\ref{fig-sample-bdy} (left arrow). 
	
	\begin{lemma}\label{lem-resample-infty-1}
	    For $(\phi, \eta)\sim M_W$, let $\tau$ be the time $\eta$ separates $i$ from $\infty$. Let $\tilde \cD$ be the ordered collection of two-pointed quantum surfaces disconnected from $i$ by $\eta$ strictly before time $\tau$; the two marked points on each quantum surface are the last and first points hit by $\eta$, and the ordering is the reverse of the ordering induced by the times that these quantum surfaces are disconnected from $i$. Let $\cD_2$ be the three-pointed quantum surface disconnected from $i$ by $\eta$ at time $\tau$; its marked points are $\eta(\tau)$,  the rightmost point of  $\eta\cap \bbR$, and $\infty$. Let $\cD_1$ be the ordered collection of two-pointed quantum surfaces disconnected from $i$ after time $\tau$; the two marked points on each quantum surface are the two boundary points visited twice by $\eta$. Then the  law of $(\cD_1, \cD_2, \tilde \cD)$ is 
	    \eqb \label{eq-MW-decomp}
	    (1-\frac2{\gamma^2}W)^2\int 1_{r_1+r_2+\tilde r > \ell_2} \Md_{0,2}(W; r_1+r_2+\tilde r - \ell_2, r_1)  \times \Md_{0,2, \bullet}(\gamma^2-W; \ell_2, \cdot,  r_2) \times \Md_{0,2}(W;\cdot, \tilde r) ,
	    \eqe
	    where the integral is taken over $(\ell_2, r_1, r_2, \tilde r)$ with respect to Lebesgue measure on $(0,\infty)^4$.  
	\end{lemma}
	\begin{proof}
		Consider $\cD$ sampled as in Definition~\ref{def-MW}, and add a third marked point on the left boundary arc such that the quantum lengths from the first to the second and third boundary points agree. 
		By the definition of $\Md_{0,2, \bullet}(W; \ell, \cdot, r)$ in Lemma~\ref{lem-disintegrate-3-pt}, this three-pointed quantum surface has law $\int_0^\infty \Md_{0,2, \bullet}(W; r, \cdot, r)\, dr$. The result then follows from the decomposition of $\Md_{0,2, \bullet}(W)$ given in  Lemma~\ref{lem-disk-decomp}. 
	\end{proof}
	
	\begin{figure}[ht]
		\centering
		\includegraphics[scale=0.42]{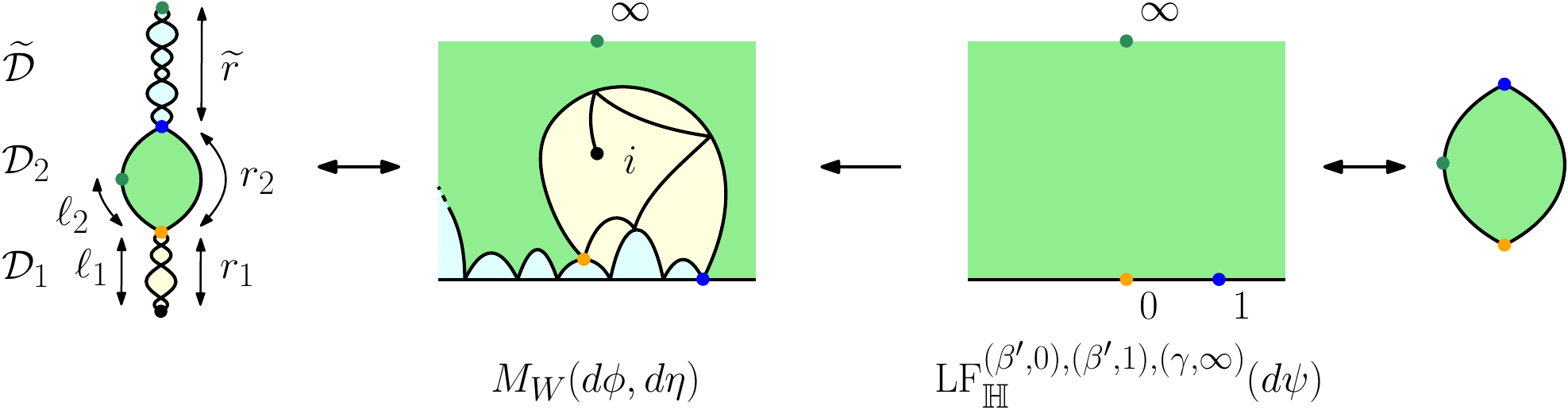}
		\caption{\textbf{Left arrow:} Depiction of~\eqref{eq-MW-decomp} in Lemma~\ref{lem-resample-infty-1}; here, we set $\ell_1 = r_1+r_2+\tilde r - \ell_2$ since the boundary arc of length $\ell_1+\ell_2$ is to be conformally welded to the boundary arc of length $r_1+r_2+\tilde r$.   \textbf{Right arrow:} \cite{AHS21} implies that when a sample from $\Md_{0,2, \bullet}(\gamma^2-W)$ is embedded as $(\bbH, \psi, 0, 1, \infty)$, the law of $\psi$ is a multiple of $\LF_\bbH^{(\beta', 0), (\beta', 1), (\gamma, \infty)}$. \textbf{Middle arrow:} We know resampling properties of  $\psi$ since it is described by a Liouville field. This implies a resampling property for 
	    $\phi$ restricted to the green region (Lemma~\ref{lem-resample-infty-2}). The independence of $\phi$ and $\eta$ then gives  Lemma~\ref{lem-resample-bdy}. }\label{fig-sample-bdy}
	\end{figure}
	
	\begin{lemma}\label{lem-resample-infty-2}
	    For $(\phi, \eta) \sim M_W$, let $\tau$ be the time $\eta$ separates $i$ from $\infty$, and let $U$ be the unbounded connected component of $\bbH \backslash \eta((0,\tau])$. Conditioned on $\eta$ and $\phi|_{\bbH \backslash U}$, we have $\phi|_U \stackrel d= h_0 + \mathfrak h_0$ where $h_0$ is a GFF on $U$ with zero boundary conditions on $\partial U \cap \bbH$ and free boundary conditions elsewhere, and $\mathfrak h_0$ is the harmonic extension of $\phi|_{\bbH \backslash U}$ to $U$ having zero normal derivative on $\partial U \cap \bbH$. 
	\end{lemma}
	\begin{proof} 
	    See Figure~\ref{fig-sample-bdy} (middle arrow). 
	    Let $T = \sup \{ t < \tau\::\: \eta(t) \in \partial \bbH\}$. 
	    Let $\psi$ be the field on $\bbH$ such that $(\bbH, \psi, 0, 1, \infty)/{\sim_\gamma} = (U, \phi, \eta(\tau), \eta(T), \infty)/{\sim_\gamma}$. Let $\hat {\mathfrak h}$ be the harmonic function on $\bbH$ which agrees with $\psi$ on $(-\infty, 1)$ and has zero normal derivative on $(1, \infty)$. We first claim that
	    \eqb\label{eq-sigma-algebras-bdy}
	    \sigma( (\hat{\mathfrak h}, \tilde \cD, \cD_1)) = \sigma((\phi|_{\bbH \backslash U}, \eta)).
	    \eqe
	    Indeed, starting with the left tuple, $\hat{\mathfrak h}$ determines the quantum length measure on the boundary arc $(-\infty, 1)$ of $(\bbH, \hat{\mathfrak h}, 0, 1, \infty )/{\sim_\gamma}$, so it can be conformally welded with $\tilde \cD$ and $\cD_1$ to give $(\phi|_{\bbH \backslash U}, \eta)$ (the field $\phi|_{\bbH \backslash U}$ corresponds to the fields of $\tilde \cD$ and $\cD_1$ after conformal welding). Conversely, starting with $(\phi|_{\bbH \backslash U}, \eta)$ and cutting along $\eta$ gives $(\hat{\mathfrak h}, \tilde \cD, \cD_1)$. 
	    
	    Next, we claim that conditioned on $\sigma(\hat{\mathfrak h}, \tilde \cD, \cD_1)$, we have $\psi \stackrel d= \hat h + \hat {\mathfrak h}$ where $\hat h$ is a GFF on $\bbH$ with zero boundary conditions on $(-\infty, 1)$ and free boundary conditions on $(1, \infty)$. By Lemma~\ref{lem-resample-infty-1}, conditioned on $(\tilde \cD, \cD_1)$ and on the boundary lengths $L_2$ and $R_2$ of $\cD_2$ from the first to the third and second marked points, the conditional law of $\cD_2$ is $\Md_{0,2, \bullet}(\gamma^2-W; L_2, \cdot, R_2)$. Lemma~\ref{lem-embed-marked-disk} states that when a sample from $\Md_{0,2, \bullet}(\gamma^2-W)$ is embedded in $(\bbH, \infty, 0, 1)$, the law of the resulting field is a multiple of $\LF_\bbH^{(\frac1\gamma(2+W), \infty), (\frac1\gamma(2+W), 0), (\gamma, 1)}$, hence by Lemma~\ref{lem:lcft-H-conf-infty},
	    the conditional law of $\psi$ given $(\tilde \cD, \cD_1, L_2, R_2)$ is $\LF_{\bbH}^{(\frac1\gamma(2+W), 0), (\frac1\gamma(2+W), 1), (\gamma, \infty)}$ conditioned on the quantum lengths of $(-\infty, 0)$ and $(0, 1)$ being $R_2$ and $L_2$ respectively. Since these quantum lengths are measurable with respect to $\hat{\mathfrak h}$, by Lemma~\ref{lem-lf-resample-3}  the conditional law of $\psi$ given $(\mathfrak h, \tilde \cD, \cD_1)$ is $\psi \stackrel d= \hat h + \hat {\mathfrak h}$, as claimed. 
	    
	    Finally, by the above claim, conformal invariance and~\eqref{eq-sigma-algebras-bdy}, we see that conditioned on $\sigma(\phi|_{\bbH \backslash U}, \eta)$, we have $\phi|_U \stackrel d= h_0 + \mathfrak h_0$ as desired. 
	\end{proof}

	\begin{proof}[Proof of Lemma~\ref{lem-resample-bdy}]
	For $(\phi, \eta) \sim M_W$, let $U = U(\eta)$ be defined as in Lemma~\ref{lem-resample-infty-2}. Condition on $\{A \subset U\}$ and on $(\phi|_{\bbH \backslash U}, \eta)$. Lemma~\ref{lem-resample-infty-2} states that the conditional law of $\phi|_U$ is a GFF with mixed boundary conditions. By the domain Markov property of the GFF, conditioned on $\{A \subset U\}$ and on $(\phi|_{\bbH \backslash A}, \eta)$, we have $\phi|_A \stackrel d= h + \mathfrak h$. Since $U$ depends only on $\eta$, and $\phi$ is independent of $\eta$ (Proposition~\ref{prop-resample-inside}), the result holds even if we do not condition on $\{A \subset U\}$ and $\eta$.
	\end{proof}
	
	\subsection{Resampling a neighborhood of $\infty$}
	\label{sec-resample-infty}
	
	We now adapt the argument of the previous section to prove Proposition~\ref{prop-resample-1}. 
	
	\begin{figure}[ht]
		\centering
		\includegraphics[scale=0.42]{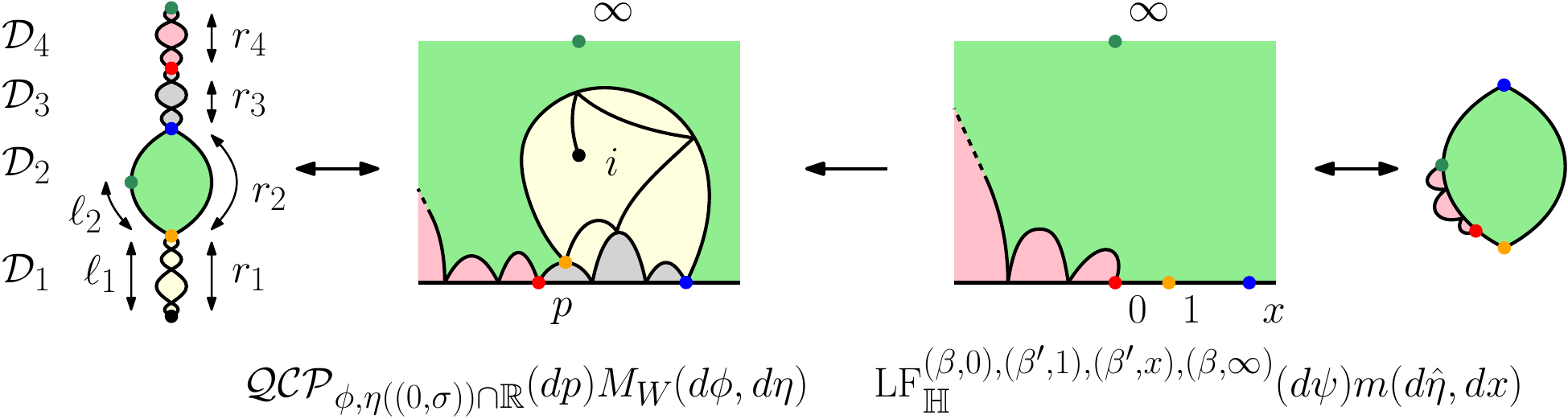}
		\caption{\textbf{Left arrow:} Figure for Lemma~\ref{lem-resample-infty-2.5}; here, we set $\ell_1 = \sum_{i=1}^4 r_i - \ell_2$ since the boundary arc of length $\ell_1+\ell_2$ is to be conformally welded to the boundary arc of length $\sum_{i=1}^4 r_i$.  \textbf{Right arrow:} Figure for Proposition~\ref{prop-4-pt}, which describes the conformal welding of green and pink surfaces via the Liouville field.  \textbf{Middle arrow:} Proof of Lemma~\ref{lem-resample-infty-3}. The field $\phi$ inherits its resampling  property from $\psi$.}\label{fig-sample-infty}
	\end{figure}
	
	The following is a variant of Lemma~\ref{lem-resample-infty-1} when a point is added according to quantum cut point measure. Recall that for a finite measure $\mu$, we write $\mu^\# = \mu/|\mu|$.
	\begin{lemma}\label{lem-resample-infty-2.5}
	    For $(\phi, \eta) \sim M_W$, let $\tau$ be the time $\eta$ separates $i$ from $\infty$, let $\sigma< \tau$ be the first time at which $\eta (\sigma) = \eta(\tau)$, and let $\mathcal{QCP}_{\phi, \eta((0, \sigma))\cap \bbR}$ denote the quantum cut point measure of $\eta ((0, \sigma)) \cap \bbR$. More precisely, this is the quantum cut point measure of the thin quantum disk defining $(\phi, \eta)$ (Definition~\ref{def-thin-wedge}) restricted to the set of cut points $\eta ((0, \sigma)) \cap \bbR$.
	     Weight the law of $(\phi, \eta)$ by $|\mathcal{QCP}_{\phi, \eta((0, \sigma))\cap \bbR}|$ and sample a point $p \sim \mathcal{QCP}_{\phi, \eta((0, \sigma))\cap \bbR}^\#$.

	    Define $(\cD_1,\cD_2, \tilde \cD)$ as in Lemma~\ref{lem-resample-infty-1}. Let $\cD_3$ (resp.\ $\cD_4$) be the ordered subset of $\tilde \cD$ corresponding to two-pointed quantum surfaces disconnected after (resp.\ before) $\eta$ hits $p$. Then the law of $(\cD_1, \cD_2, \cD_3, \cD_4)$ is
        \begin{align*}
           C\int 1_{r_1+r_2+r_3+r_4> \ell_2, r_4 < \ell_2}  \Big( &\Md_{0,2}(W; r_1+r_2+r_3+r_4 - \ell_2, r_1) \times \Md_{0,2, \bullet}(\gamma^2-W; \ell_2, \cdot, r_2) \\
           &\times \Md_{0,2}(W; \cdot, r_3) \times \Md_{0,2}(W; \cdot, r_4) \Big), 
        \end{align*}
	    where the integral is taken over $(\ell_2, r_1,r_2,r_3,r_4)$ with respect to Lebesgue measure on $(0,\infty)^5$, and $C = (1-\frac2{\gamma^2}W)^4$.
	\end{lemma}
	\begin{proof}
	Let $\mathcal{QCP}_{\phi, \eta \cap \bbR}$ be the quantum cut point measure of $\eta \cap \bbR$. If we had instead weighted by $|\mathcal{QCP}_{\phi, \eta \cap \bbR}|$ and sampled $x \sim \mathcal{QCP}_{\phi, \eta \cap \bbR}^\#$, then by Lemmas~\ref{lem-cut-disk} and~\ref{lem-resample-infty-1}, the law of $(\cD_1, \cD_2, \cD_3, \cD_4)$ would be 
    \begin{align*}
        C\int 1_{r_1+r_2+r_3+r_4 > \ell_2}  \Big( &\Md_{0,2}(W; r_1+r_2+r_3+r_4 - \ell_2, r_1) \times \Md_{0,2, \bullet}(\gamma^2-W; \ell_2, \cdot, r_2) \\
        &\times\Md_{0,2}(W; \cdot, r_3)\times \Md_{0,2}(W; \cdot, r_4) \Big).
    \end{align*}
	    The event $\{ x \in \eta((0,\sigma))\}$ exactly corresponds to the event that the right boundary length of $\cD_4$ is shorter than the boundary length of $\cD_2$ between its first and third marked points, i.e., $r_4 < \ell_2$. Restricting to this event gives the claim.
	\end{proof}
	Next, Proposition~\ref{prop-4-pt} below shows that the conformal welding of two components of Lemma~\ref{lem-resample-infty-2.5} is described by the Liouville field. 
	\begin{proposition}\label{prop-4-pt}
	Suppose $(\cD_2, \cD_4) \sim \iint_0^\infty   \Md_{0,2,\bullet} (\gamma^2-W; \cdot, \ell + r_4, \cdot) \times \Md_{0,2}(W;  \cdot, r_4) \, dr_4 \, d\ell$. Conformally weld $\cD_4$ to $\cD_2$, identifying the second marked point of $\cD_4$ with the third marked point of $\cD_2$, and welding the whole right boundary of $\cD_4$ to the left boundary arc of $\cD_2$. 
	This gives a four-pointed quantum surface decorated by a curve; the four marked points are the three marked points of $\cD_2$ (in the same order) and finally the first marked point of $\cD_4$. 
	Embed the resulting curve-decorated quantum surface as $(\bbH, \psi, \hat \eta,  1, x, \infty, 0)$ where $\hat \eta$ is a curve from $\infty$ to $0$ in $\bbH$. Then there exists a measure $m$  such that the joint law of $(\psi, (\hat \eta, x))$ is $\LF_\bbH^{(\beta, 0), (\beta', 1), (\beta', x), (\beta, \infty)} m(d\hat \eta, dx)$, where $\beta = \gamma - \frac W\gamma$ and $\beta' = \frac1\gamma(2+W)$.
	\end{proposition}
	\begin{proof}
	This is a special case of~\cite[Proposition 3.4]{SY23}.
	\end{proof}
	Finally, Lemma~\ref{lem-resample-infty-3} below is the analog of Lemma~\ref{lem-resample-infty-2}.
	\begin{lemma}\label{lem-resample-infty-3}
	    In the setting of Lemma~\ref{lem-resample-infty-2.5},
	    let $\eta_p$ be the trace of the curve after hitting $p$, and let 
	    $U \subset \bbH$ be the unbounded connected component of $\bbH \backslash \eta_p$. Conditioned on $\phi|_{\bbH \backslash U}$, we have $\phi|_U \stackrel d= h_0 + \mathfrak h_0 + \frac\beta2 G_U(\cdot, \infty)$, where $h_0$ is a GFF on $U$ with zero boundary conditions on $\partial U \cap \bbH$ and free boundary conditions elsewhere, $\mathfrak h_0$ is the harmonic extension of $\phi|_{\bbH \backslash U}$ to $U$ having zero normal derivative on $\partial U \cap \partial \bbH$, and $G_U$ is the Green function for $h_0$. 
	\end{lemma}
	\begin{proof}
	The proof is identical to that of Lemma~\ref{lem-resample-infty-2}, except that~\eqref{eq-sigma-algebras-bdy} is replaced by
	\eqb\label{eq-resample-infty-sigma}
	\sigma((\hat{\mathfrak h}, \hat \eta, x),  \cD_1, \cD_3)  = \sigma(\phi|_{\bbH \backslash U}, \eta),
	\eqe
	 \cite[Proposition 2.18]{AHS21} is replaced by Proposition~\ref{prop-4-pt}, and Lemma~\ref{lem-lf-resample-3} is replaced by Lemma~\ref{lem-lf-resample-4}. 
	\end{proof}

	\begin{proof}[{Proof of Proposition~\ref{prop-resample-1}}] Let $U = U(\eta)$ be defined as in Lemma~\ref{lem-resample-infty-3}.  
	The following two procedures are easily seen to be equivalent:
	\begin{itemize}
	    \item 
	    Sample $(\phi, \eta) \sim M_W$, weight by $|\mathcal{QCP}_{\phi, \eta \cap \bbR}|$, sample a point $p \sim \mathcal{QCP}_{\phi, \eta \cap \bbR}^\#$, and restrict to the event that $p \in (\partial A)^c$. 
	    \item Sample $(\phi, \eta) \sim M_W$, weight by $|\mathcal{QCP}_{\phi, \eta \cap \bbR \cap (\partial A)^c}|$, and sample a point $p \sim \mathcal{QCP}_{\phi, \eta \cap \bbR \cap (\partial A)^c}^\#$.
	\end{itemize}
	Thus, by Lemma~\ref{lem-resample-infty-3} and the Markov property of the GFF, for $(\phi, \eta, p)$ sampled from the second procedure and restricted to $\{A \subset U\}$, conditioned on $(\phi|_{\bbH \backslash A}, \eta)$ we have $\phi|_A \stackrel d= h + \mathfrak h + (\frac\beta2 - Q)G_A(\cdot, \infty)$.
	
	Since the weighting factor $|\mathcal{QCP}_{\phi, \eta\cap \bbR \cap (\partial A)^c}|$ depends only on $(\phi|_{\bbH \backslash A}, \eta)$, we deduce the following. For $(\phi, \eta)\sim M_W$, sample a point $p \sim \mathcal{QCP}_{\phi, \eta \cap \bbR \cap (\partial A)^c}^\#$, and restrict to $\{ A \subset U\}$, then conditioned on  $(\phi|_{\bbH \backslash A}, \eta)$ we have $\phi|_A \stackrel d= h + \mathfrak h + (\frac\beta2 - Q)G_A(\cdot, \infty)$. 
	
	Let $E = E(\phi|_{\bbH \backslash A}, \eta)$ be the event that there exists a time $t$ such that $|\mathcal{QCP}_{\phi, \eta([t, \sigma]) \cap \bbR}|>0$ and $\ol A \cap \eta([t,\tau]) = \emptyset$. Using the previous claim and conditioning on the event that $p \in \eta([t,\sigma])$, we get the following. For $(\phi, \eta) \sim M_W$, conditioned on  $E$ and on  $(\phi|_{\bbH \backslash A}, \eta)$ we have $\phi|_A \stackrel d= h + \mathfrak h + (\frac\beta2 - Q)G_A(\cdot, \infty)$. 
	
	To conclude, we note that for $(\phi, \eta) \sim M_W$, $\phi$ and $\eta$ are independent (Proposition~\ref{prop-resample-inside}), and  $\phi$-a.s., conditioned on $\phi$ the conditional probability of $E$ is positive. This gives the desired resampling property. 
	\end{proof}

	\subsection{Proof of Theorem~\ref{thm:w2w} for $W < \frac{\gamma^2}2$}\label{sec-proof-main-thin}
	
	For a measurable space $(\Omega, \cF)$, consider a Markov kernel $\Lambda : \Omega \times \cF \to [0,1]$ as defined in Definition~\ref{def-markov-kernel} (with $(\Omega, \cF) = (\Omega', \cF')$). A $\sigma$-finite measure $\mu$ on $\Omega$ is called \emph{invariant} if for any nonnegative measurable function $f: \Omega \to \bbR$ we have $\iint f(y) \Lambda(x, dy) \mu (dx) = \int f(y) \mu(dy)$. 
	We will need the following criterion for uniqueness of invariant $\sigma$-finite measures. 
	
	\begin{lemma}[{\cite[Lemma 5.11]{ASY22}}]\label{lem-invariant}
	Suppose $\Lambda: \Omega \times \cF \to [0,1]$ is a Markov kernel with two $\sigma$-finite invariant measures $\mu_1, \mu_2$ such that for each $\omega \in \Omega$, the measure $\mu_1$ is absolutely continuous with respect to $\Lambda(\omega, -)$. Further assume that for $i = 1,2$ we have reversibility: $\Lambda(x, dy) \mu_i(dx) = \Lambda(y, dx) \mu_i(dy)$. Then $\mu_1 = c\mu_2$ for some $c \in (0,\infty)$. 
	\end{lemma}

	With this and the resampling results Propositions~\ref{prop-resample-inside} and~\ref{prop-resample-1}, we are now ready to tackle the thin case of Theorem~\ref{thm:w2w}. The argument is very similar to that of \cite[Proposition 5.12]{ASY22}.
	
	\begin{proof}[{Proof of Theorem~\ref{thm:w2w} for $W < \frac{\gamma^2}2$}]
	Recall $M_W$ from Definition~\ref{def-MW}. Let $\alpha = Q - \frac W{2\gamma}$ and $\beta = \gamma - \frac W \gamma$.  By Lemma~\ref{lem-disk=qt-thin} and Definition~\ref{def-MW}, there is a constant $c$ such that for a sample from $\int_0^\infty \Wd(\QT(W_1,W_2,W_3;\ell,\ell))\,d\ell$ embedded in $(\bbD, 0, 1)$, the law of the field and curve is $c M_W$. Thus, by Definition~\ref{def:mdisk11}, we must show that for $(\phi, \eta) \sim M_W$, the marginal law of $\phi$ is $C \LF_\bbH^{(\alpha, i), (\beta, \infty)}$ for some constant $C$, and the conditional law of $\eta$ given $\phi$ is radial $\SLE_\kappa(W-2)$ in $\bbH$ from $\infty$ to $i$ with force point infinitesimally counterclockwise of $\infty$. The claim on the conditional law of $\eta$ given $\phi$ is shown in Proposition~\ref{prop-resample-inside}. 
	
	The remainder of this proof will identify the marginal law of $\phi$; call this law $\mu_2$. 
	Let $\mu_1 = \LF_\bbH^{(\alpha, i), (\beta, \infty)}$. We will construct a Markov kernel $\Lambda$ such that $(\Lambda, \mu_1, \mu_2)$ satisfy the conditions of Lemma~\ref{lem-invariant}, and thus conclude that $\mu_1$ and $\mu_2$ agree up to multiplicative constant as desired. 
	
	Let $L_1$ (resp.\ $L_2$) be the line segment joining $i$ and $1$ (resp.\ $-1$). Let $A_1 = \bbH \backslash \ol{B_{1/10}(L_1)}$, $A_2 = \bbH \backslash \ol{B_{1/10} (L_2)}$ and $A_3 = B_{1/2}(i)$. For $i = 1,2$, let $\Lambda_i(\phi, d\psi)$ be the law of $\psi$ defined via $\psi|_{\bbH \backslash A_i} = \phi|_{\bbH \backslash A_i}$ and $\psi|_{A_i} = h + \mathfrak h + (\frac\beta2 - Q) G_{A_i}(\cdot, \infty)$ where $h$ is a GFF on $A_i$ with zero (resp.\ free) boundary conditions on $\partial A_i \cap \bbH$ (resp.\ $\partial A_i \cap \bbR$), $\mathfrak h$ is the harmonic extension of $\phi|_{\bbH \backslash A_i}$ to $A_i$ having zero normal derivative on $\partial A_i \cap \bbR$, and $G_{A_i}$ is the Green function for $h$. Similarly, let $\Lambda_3(\phi, d\psi)$ be the law of $\psi$ defined via $\psi|_{\bbH \backslash A_3} = \phi|_{\bbH \backslash A_3}$ and $\psi|_{A_3} = h + \mathfrak h + \alpha G_{A_3}(\cdot, i)$ where $h$ is a zero boundary GFF on $A_3$, $\mathfrak h$ is the harmonic extension of $\phi|_{\bbH \backslash A_3}$ to $A_3$, and $G_{A_3}$ is the Green function for $h$. 
	
	By Proposition~\ref{prop-resample-1} 
	$\mu_2$ is invariant with respect to $\Lambda_1$ and $\Lambda_2$, and by Proposition~\ref{prop-resample-inside} $\mu_2$ is invariant with respect to $\Lambda_3$. More strongly, these propositions immediately give reversibility: $\Lambda_i(x,dy) \mu_2(dx) = \Lambda_i(y, dx) \mu_2(dy)$ for $i=1,2,3$. Likewise, the resampling properties of $\mu_1$ stated in Lemma~\ref{lem-lf-resample-i-infty} imply that the same results hold when $\mu_2$ is replaced by $\mu_1$. 
	
	Let $\Lambda$ be the Markov kernel $\Lambda = \Lambda_1 \Lambda_2 \Lambda_3 \Lambda_2 \Lambda_1$ obtained by composition. The symmetric construction of $\Lambda$ and the reversibility for each $\Lambda_i$ with each $\mu_j$ imply that $\Lambda(x, dy) \mu_j(dx) = \Lambda(y, dx) \mu_j(dy)$ for $j=1,2$. 
	
	The remaining criterion of Lemma~\ref{lem-invariant} we need to check is that for each $\phi$, $\mu_1$ is absolutely continuous with respect to $\Lambda(\phi, -)$. Let $B_1 = \bbH \backslash B_{1/9}(L_1)$. It is well known that if $h$ is GFF on $A_1$ with zero (resp.\ free) boundary conditions on $\partial A_1 \cap \bbH$ (resp.\ $\partial A_1 \cap \bbR$), 
	and $g$ is a smooth function on $A_1$ with finite Dirichlet energy in $B_1$, then the laws of $h|_{B_1}$ and $(h+g)|_{B_1}$ are mutually absolutely continuous, see for instance the argument of \cite[Proposition 2.9]{ig4}. Putting $g = \mathfrak h$ from the definition of $\Lambda_1$, we see that the $\mu_1(d\psi)$-law of $\psi|_{B_1}$ is absolutely continuous with respect to the $\Lambda_1(\phi, dx)$-law of $x|_{B_1}$. Let $B_2 = \bbH \backslash B_{1/3}(i)$. By the same argument, the $\mu_1(d\psi)$-law of $\psi|_{B_2}$ is absolutely continuous with respect to the $\int_x \Lambda_1(\phi, dx)\Lambda_2(x, dy)$ law of $y|_{B_2}$. Finally, a third application of the argument shows that $\mu_1$ is absolutely continuous with respect to $\int_x \int_y \Lambda_1(\phi, dx) \Lambda_2(x, dy) \Lambda_3(y, -)$, and hence $\mu_1$ is absolutely continuous with respect to $\Lambda(\phi, -)$. Thus, by Lemma~\ref{lem-invariant}, we have $\mu_2 = c \mu_1$, so we are done. 
	\end{proof}

	\section{Radial conformal welding for $\kappa \in (4,8)$ }\label{sec-nonsimple}
	
	In this section we prove Theorem~\ref{thm:main-nonsimple}. As with Theorem~\ref{thm:main}, most of the work goes into proving the case where $W_1 = W_2$:
	
	\begin{theorem}\label{thm:w2w-ns}
		For $W \in (0,\frac{\gamma^2}2) \cup (\frac{\gamma^2}2, \infty)$, Theorem~\ref{thm:main-nonsimple} holds for $W_1 = W_2 = W$ and $W_3 = \gamma^2-2$. 
	\end{theorem}
	
	In Section~\ref{sec-forested} we recall the notion of forested quantum surfaces, and in Section~\ref{subsec-third-point-ns} we describe the surface obtained by adding a third marked point to a forested quantum disk. In Sections~\ref{sec-thick-nonsimple} and~\ref{sec-thin-nonsimple} we prove Theorem~\ref{thm:w2w-ns} for $W > \frac{\gamma^2}2$ and $W < \frac{\gamma^2}2$ respectively; the arguments are almost identical to those of Sections~\ref{sec:pf-thick} and~\ref{sec:pf-thin} respectively. Finally, in Section~\ref{sec-proof-nonsimple} we prove Theorem~\ref{thm:main-nonsimple} using Theorem~\ref{thm:w2w-ns}.

	\subsection{Forested quantum surfaces}\label{sec-forested}
	
	Forested quantum surfaces were first introduced in \cite{DMS14}.  We will follow the presentation of \cite{AHSY23}. 
	
	Let $\gamma \in (\sqrt2, 2)$. 
	There exists a $\sigma$-finite measure called $\mathrm{GQD}_1$, a sample of which is called a \emph{generalized quantum disk}. A generalized quantum disk is a rooted looptree of $\gamma$-LQG surfaces each having the disk topology. The boundary of the looptree comes equipped with a notion of \emph{generalized boundary length}, which can be defined (up to multiplicative constant) by counting the number of loops on a boundary arc having length at least $\eps$, normalizing by a power of $\eps$, and sending $\eps \to 0$. We will not need the precise definitions of $\mathrm{GQD}_1$ and generalized boundary length so we omit them, but see \cite[Sections 3.1 and 3.2]{AHSY23} for details (see also \cite{msw-nonsimple,HL22CLE} for alternative treatments). 
	
	Given a quantum surface $\cD$, one can obtain a \emph{forested quantum surface} $\cD^f$ by sampling a Poisson point process with intensity measure $c \mathrm{GQD}_1 \times \cL_\cD$ (where $\cL_\cD$ is the quantum boundary length measure of $\cD$) and rooting each generalized quantum disk at the corresponding boundary point of $\cD$.  Here, the constant $c$ is defined as in \cite[Proposition 3.11]{AHSY23}. See Figure~\ref{fig-forested} (left). 
	We say that $\cD^f$ is obtained by foresting the boundary of $\cD$. 
	
	\begin{definition} \label{def-fd}
			Let $\mathcal M_{0,2}^\mathrm{f.d.}(W)$ be the law of the forested quantum surface obtained from a sample from $\Md_{0,2}(W)$ by foresting its boundary. Similarly, define $\mathcal M_{1,1}^\mathrm{f.d.}(W_1, W_2)$  and $\QT^f(W_1, W_2, W_3)$ to be the laws of samples from $\Md_{1,1}(W_1, W_2)$ and $\QT(W_1,W_2,W_3)$ after foresting their boundaries. 
	\end{definition}

	As with~\eqref{eq-disintegrate-disk} and~\eqref{eq-disintegrate-triangle}, we define the disintegrations with respect to generalized boundary lengths of $\mathcal M^\mathrm{f.d.}_{0,2}(W)$ and $\QT^f(W_1,W_2,W_3)$:
	\begin{equation*}
	\begin{split}
		\mathcal M^\mathrm{f.d.}_{0,2}(W) &= \iint_{\bbR_2^+}\mathcal M^\mathrm{f.d.}_{0,2}(W;\ell_1,\ell_2)\, d\ell_1\, d\ell_2, \\
			\QT^f(W_1,W_2,W_3) &= \iiint_{\bbR_+^3}\QT^f(W_1,W_2,W_3;\ell_1,\ell_2,\ell_3)\, d\ell_1\, d\ell_2\, d\ell_3.
	\end{split}
	\end{equation*}
	Here, a sample from $\mathcal M^\mathrm{f.d.}_{0,2}(W;\ell_1,\ell_2)$ a.s.\ has left and right generalized boundary lengths equal to $\ell_1$ and $\ell_2$, and for a sample from  $\QT^f(W_1,W_2,W_3;\ell_1,\ell_2,\ell_3)$, calling its vertices $a_1, a_2, a_3$, the boundary arcs $a_1a_2$, $a_1a_3$ and $a_2a_3$ a.s.\ have generalized boundary lengths $\ell_1,\ell_2,\ell_3$. We similarly define $\QT^f(W_1,W_2,W_3;\ell_1,\ell_2) = \int_0^\infty \QT^f(W_1,W_2,W_3;\ell_1,\ell_2,\ell_3)\, d\ell_3$ and $\QT^f(W_1,W_2,W_3;\ell_1) = \int_0^\infty \QT^f(W_1,W_2,W_3;\ell_1,\ell_2)\, d\ell_2$. 

   Let $\kappa' = 16/\gamma^2 \in (4,8)$. 
	As explained in \cite[Theorem 1.4.8]{DMS14}, there is a way to conformally weld a pair of independent forested quantum wedges with respect to generalized boundary length to obtain a new forested quantum wedge with an independent $\SLE_{\kappa'}(\rho_-; \rho_+)$ curve on it, wherein the welded curve-decorated forested quantum surface is measurable with respect to the two original forested quantum wedges. This is the only conformal welding satisfying certain regularity conditions \cite{mmq-uniqueness}, but the remarkable recent work \cite{lz-nonremovability} proves that for $\kappa'$ close to 8 there do exist other conformal weldings. In all future instances, when we refer to the conformal welding of forested quantum surfaces, we mean the conformal welding one obtains from  \cite{DMS14}. See the discussion immediately above \cite[Theorem 1.4]{AHSY23} for further details. 
	
		\begin{figure}[ht]
		\centering
		\includegraphics[scale=0.39]{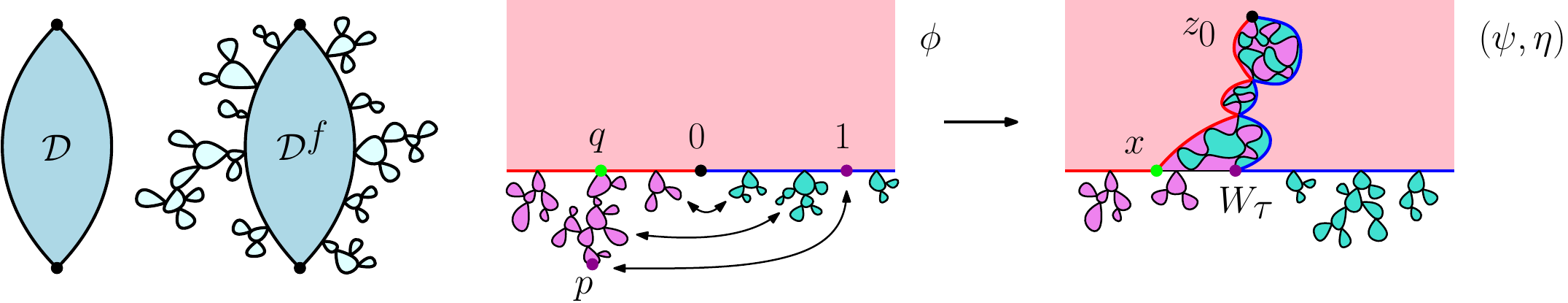}
		\caption{\textbf{Left:} Given a quantum surface $\cD$, attaching a Poissonian collection of generalized quantum disks to its boundary gives a forested quantum surface $\cD^f$. Since $\mathrm{GQD}_1$ is an infinite measure, there are infinitely many looptrees attached to $\cD$. \textbf{Right:} Figure for Proposition~\ref{prop-quantum-zipper-ns} and Lemma~\ref{lem-quantum-zipper-ns}. We forest the boundary of a quantum surface whose field is $\phi$, conformally weld the left and right boundary arcs according to generalized boundary length, and stop when $1$ is welded to $p$. This gives a curve-decorated forested quantum surface; $(\bbH, \psi, \eta)$ is the unique embedding of the connected component containing the curve such that if $g$ is the conformal map between the pink regions such that $g^{-1} \bullet_\gamma \psi = \phi$, then  $\lim_{z \to \infty} g(z) - z = 0$.}\label{fig-forested}
	\end{figure}

	\subsection{Adding a third marked point to a forested quantum disk}\label{subsec-third-point-ns}
	
	In this section we give analogs of results in Section~\ref{subsec-third-point} for forested quantum disks. 
	For a sample $\cD^f$ from $\mathcal{M}^\mathrm{f.d.}_{0,2}(W)$, let $\cL_{\cD^f}^\mathrm{left}$ denote the generalized boundary length measure on the left boundary arc of $\cD^f$. For $(p, \cD^f)$ sampled from $\cL_{\cD^f}^\mathrm{left} \mathcal{M}^\mathrm{f.d.}_{0,2}(W)$, let $\mathcal{M}^\mathrm{f.d.}_{0,2, \bullet}(W)$ denote the law of the three-pointed  surface obtained from adding to $\cD^f$ the third marked point $p$.

	The following is an analog of Lemma~\ref{lem-disintegrate-3-pt}; its proof is identical so we omit it. 
	
	\begin{lemma}\label{lem-disintegrate-3-pt-ns}
		Let $\mathcal{M}^\mathrm{f.d.}_{0,2, \bullet}(W; \ell', \ell'', r)$ denote the law of  a sample from $\mathcal{M}^\mathrm{f.d.}_{0,2}(W; \ell'+\ell'', r)$ with a third marked point added to the left boundary at generalized boundary length $\ell'$ from the first marked point. Then 
		\[\mathcal{M}^\mathrm{f.d.}_{0,2, \bullet}(W) = \iiint_0^\infty \mathcal{M}^\mathrm{f.d.}_{0,2, \bullet}(W; \ell',\ell'', r) \, d\ell' \, d\ell'' \, dr. \]
		Similarly, 
		let $\mathcal{M}^\mathrm{f.d.}_{0,2, \bullet}(W; \ell',\cdot, r)$ denote the law of  a sample from $\mathcal{M}^\mathrm{f.d.}_{0,2}(W; \cdot, r)$ restricted to the event that the left generalized boundary length is greater than $\ell'$, with a third marked point added to the left boundary at generalized boundary length $\ell'$ from the first marked point. Then
		\[\mathcal{M}^\mathrm{f.d.}_{0,2, \bullet}(W) = \iint_0^\infty \mathcal{M}^\mathrm{f.d.}_{0,2, \bullet}(W; \ell', \cdot, r) \, d\ell' \, dr. \]
	\end{lemma}
	
	We now give the forested analog of Lemmas~\ref{lem-disk=qt} and~\ref{lem-disk=qt-thin}.
	
	\begin{lemma}\label{lem-embed-marked-disk-ns}
		Let $W \in (0, \frac{\gamma^2}2) \cup (\frac{\gamma^2}2, \infty)$, then there is a constant $C= C(W)$ such that $\mathcal{M}^\mathrm{f.d.}_{0,2, \bullet}(W) = C \mathrm{QT}^f(W, W, \gamma^2-2)$. 
	\end{lemma}
	\begin{proof}
		By Definition~\ref{def-qt-thin}, there is a constant $c$ such that a sample from $\QT(W, W, \gamma^2 - 2)$ can be obtained from a pair of quantum surfaces $(\cD, \cD_2) \sim c\Md_{0,2, \bullet}(W) \times \Md_{0,2}(\gamma^2-2)$ by attaching $\cD_2$ to the third marked point of $\cD$. Consequently, a sample from $\QT^f(W, W, \gamma^2 - 2)$ can be obtained from the above by foresting the boundaries of $\cD$ and $\cD_2$; \cite[Lemma 3.15]{AHSY23} 
		identifies the law of this forested quantum surface as a constant times $\cM^\mathrm{f.d.}_{0,2,\bullet}(W)$, as claimed. 
	\end{proof}

	\subsection{Theorem~\ref{thm:w2w-ns} for $W > \frac{\gamma^2}2$}\label{sec-thick-nonsimple}
	In this section we prove Theorem~\ref{thm:w2w-ns} for $W > \frac{\gamma^2}2$. The argument is parallel to that of Section~\ref{sec:pf-thick}. 
	
	We begin with a forested variant of Proposition~\ref{prop-quantum-zipper}. See Figure~\ref{fig-forested} (right). Let $\gamma \in (\sqrt2, 2)$. Let $\beta < Q$, sample $\phi \sim \LF_\bbH^{(\beta, 0), (\beta, 1)}$, forest the boundary of the quantum surface $(\bbH, \phi, 0, 1)$, and restrict to the event that the generalized boundary length from $-\infty$ to $0$ is larger than the generalized boundary length from $0$ to $1$. Let $p$ be the point on the  boundary arc from $-\infty$ to $0$ such that the generalized boundary length from $p$ to $0$ agrees with the generalized boundary length from $0$ to $1$. Conformally weld the boundary arc between $0$ and $1$ to the boundary arc between $0$ and $p$ according to generalized boundary length, to obtain a curve-decorated forested quantum surface. Forget the foresting and consider the unique embedding $(\bbH, \psi, \eta)$ such that with $U$ the unbounded connected component of $\bbH \backslash \eta$ and $z_0$ the endpoint of $\eta$ lying in $\bbH$, we have $(\bbH, \phi, 0, \infty)/{\sim_\gamma} = (U, \psi, z_0, \infty)/{\sim_\gamma}$  and $\lim_{z \to \infty} g(z)-z = 0$  where $g: \bbH \to U$ is the conformal map such that $g^{-1} \bullet_\gamma \psi = \phi$.
	
	\begin{proposition}\label{prop-quantum-zipper-ns}
		In the above setting, the law of $(\psi, \eta)$ is 
		\[ 
		(2 \Im  \tilde g_\tau(0))^{\alpha^2/2} |\tilde g_\tau(0) - \tilde W_\tau|^{\alpha \beta'}
		\LF_\bbH^{(\alpha, \tilde g_\tau(0)),(\beta', \tilde W_\tau)}(d \psi) 1_{\tau < \infty} \mathrm{rSLE}_{\kappa',\tilde\rho+2,\tilde\rho}^\tau (d \eta),\]
		\[ \hfill \kappa' = \frac{16}{\gamma^2}, \ \alpha = \frac{\beta}{2}+\frac{\gamma}{4}, \ \beta' =\beta-\frac{\gamma}{2},  \ \tilde\rho =  \frac4\gamma \beta,\]
		where $\mathrm{rSLE}_{\kappa',\tilde\rho+2,\tilde\rho}^\tau$ denotes the law of reverse $\SLE_{\kappa'}$ with a weight $\tilde\rho+2$ force point located infinitesimally above 0 and a weight $\tilde\rho$ force point at $1$ run until the time $\tau$ that the force point hits the driving function, i.e., $\tilde g_\tau(1) = \tilde W_\tau$.
	\end{proposition}
	\begin{proof}
		This is a special case of \cite[Theorem 1.7]{Ang23}.
	\end{proof}

	We will use this to prove the following analog of Proposition~\ref{prop:radialzipup}:
	
	\begin{proposition}\label{prop:radialzipup-ns}
		Fix $W>\frac{\gamma^2}{2}$. There exists a constant $c:=c_W\in(0,\infty)$ and  $\sigma$-finite measure $\mathfrak{m}(W)$ on the space of continuous curves from $0$ to $i$ such that
		\begin{equation}\label{eqn:radial-zip-up-ns}
			\mathcal M^\mathrm{f.d.}_{1,1}(W,W+2)\otimes\mathfrak{m}(W) = c\int_0^\infty \mathrm{Weld}(\QT^f(W,W,2;\ell,\ell))d\ell.
		\end{equation}
	\end{proposition}

    The proof of Proposition~\ref{prop:radialzipup-ns} is essentially the same as that of Proposition~\ref{prop:radialzipup}. It depends on the following analogs of Lemmas~\ref{lem-marginal-disk} and~\ref{lem-quantum-zipper}.

    	\begin{lemma}\label{lem-marginal-disk-ns}
        Let $W > \frac{\gamma^2}2$, let $\beta = Q + \frac\gamma2 - \frac W\gamma < Q$ and let $I \subset (-\infty, 0)$ be a compact interval of positive length.  Sample $\phi \sim \LF_\bbH^{(\beta,0), (\beta, 1)}$, forest the boundary of $(\bbH, \phi, 0, 1)$, forest the boundary to get a forested quantum surface $\cD^f$, restrict to the event that there is a point $p$ on the forested boundary arc from $0$ to $-\infty$ such that the generalized boundary length from $p$ to $0$ agrees with that from $0$ to $1$, let $q \in \partial \bbH$ be the boundary point on the same generalized quantum disk as $p$, and further restrict to the event that $q \in I$. 
        There is a constant $c = c(I) \in (0, \infty)$ such that the law of $\cD^f$ is $c \mathcal M^\mathrm{f.d.}_{0,2}(W)$ restricted to the event that the generalized boundary length of the left boundary arc is greater than that of the right boundary arc.
	\end{lemma}
    \begin{proof}
        The proof is identical to that of Lemma~\ref{lem-marginal-disk}, but with $q$ replacing the marked point $p$ from Lemma~\ref{lem-marginal-disk}.
    \end{proof}

	\begin{lemma}\label{lem-quantum-zipper-ns}
        In the setting of Lemma~\ref{lem-marginal-disk-ns}, conformally weld the boundary arc of $\cD^f$ between $0$ and $1$ to the boundary arc between $0$ and $p$ according to generalized boundary length to get a forested curve-decorated quantum surface with a marked bulk and boundary point; forget the foresting and embed it as $(\bbH, \psi', \eta', i, 0)$. The law of $(\psi', \eta')$ is   \[\LF_\bbH^{(\alpha,i),(\beta', 0)}(d\psi')  \mathfrak m(d\eta') \qquad \text{ with } \ \ \alpha = \frac{\beta}{2}+\frac{1}{\gamma}, \ \  \beta' = \beta-\frac{2}{\gamma},\]
		where $\mathfrak m$ is a $\sigma$-finite measure on the space of curves in $\ol\bbH$ from $0$ to $i$. 
	\end{lemma}
	\begin{proof}
    Define $(\psi, \eta)$ as in Proposition~\ref{prop-quantum-zipper-ns}. By Proposition~\ref{prop-quantum-zipper-ns} the law of $(\psi, \eta)$ is \[ (2 \Im  \tilde g_\tau(0))^{\alpha^2/2} |\tilde g_\tau(0) - \tilde W_\tau|^{\alpha \beta'} \LF_\bbH^{(\alpha, \tilde g_\tau(0)),(\beta', \tilde W_\tau)}(d \psi) 1_{\tilde g_\tau^{-1}(x) \in I} 1_{\tau < \infty} \mathrm{rSLE}_{\kappa,\tilde \rho}^\tau (d \eta),\] where $x = x(\eta)$ is defined by letting $U$ be the unbounded connected component of $\bbH\backslash \eta$ and setting $x = \inf\{ y \in \bbR \: : \: y \not \in \ol U\}$. Here, we are using $q = \tilde g_\tau^{-1}(x)$ where $q$ is as defined in Lemma~\ref{lem-marginal-disk-ns}. 
	Since $(\bbH, \psi', \eta', i, 0)/{\sim_\gamma} = (\bbH, \psi, \eta, \tilde g_\tau(0), \tilde g_\tau(1))/{\sim_\gamma}$, the claim follows from the conformal covariance of the Liouville field (Lemma~\ref{lem:lf-covariance}).
	\end{proof}
	
	\begin{proof}[{Proof of Proposition~\ref{prop:radialzipup-ns}}]
		In the setting of Lemma~\ref{lem-marginal-disk-ns}, by Lemma~\ref{lem-marginal-disk-ns} there is a constant $c$ such that the law of $\cD^f$ is $c \mathcal M^\mathrm{f.d.}_{0,2}(W)$ restricted to the event that the generalized boundary length of the left boundary arc is greater than that of the right boundary arc. Thus, by Lemmas~\ref{lem-disintegrate-3-pt-ns} and~\ref{lem-embed-marked-disk-ns}, the conformal welding $\tilde \cD^f$ of $\cD^f$ to itself has law $c' \int_0^\infty \mathrm{Weld}(\QT^f(W, W, \gamma^2-2; \ell, \ell))\, d\ell$ for some constant $c'$. Comparing this with  Lemma~\ref{lem-quantum-zipper-ns}, we conclude that if $\tilde \cD$ is the quantum surface obtained from $\tilde \cD^f$ by forgetting its boundary foresting, the law of $\tilde \cD$ is $c''\Md_{1,1}(W, W+2) \otimes \mathfrak m(W)$ for some constant $c''$ and for some $\sigma$-finite measure $\mathfrak m(W)$. By Definition~\ref{def-fd} and the fact that the law of $\tilde \cD^f$ is $c' \int_0^\infty \mathrm{Weld}(\QT^f(W, W, \gamma^2-2; \ell, \ell))\, d\ell$, conditioned on $\tilde \cD$ one can resample $\tilde\cD^f$ by foresting the boundary of $\tilde \cD$, so the law of $\tilde \cD^f$ is $c''\mathcal M^\mathrm{f.d.}_{1,1}(W,W+2)\otimes\mathfrak{m}(W)$ as needed. 
	\end{proof}
	
	Finally, given Proposition~\ref{prop:radialzipup-ns}, to complete the proof of Theorem~\ref{thm:w2w-ns} for $W >\frac{\gamma^2}2$, what remains is to identify $\mathfrak m$ as a multiple of radial SLE$_{\kappa'}(\rho)$ with $\rho = W-(\gamma^2-2)$,  where the force point is infinitesimally to the right of the boundary point. 
	
	\begin{proof}[{Proof of Theorem~\ref{thm:w2w-ns} for $W > \frac{\gamma^2}2$}] It suffices to identify the law $\mathfrak m$ of the curve in Proposition~\ref{prop:radialzipup-ns}. 
	Without loss of generality, assume the spine of the surface from welding is embedded in $\bbD$. Let $\mu$ be the law of the left and right boundaries $(\eta^L,\eta^R)$ of the interface $\eta'$ where $\eta'\sim\mathfrak{m}$. By~\cite[Proposition 3.25]{AHSY23} and Theorem~\ref{thm:disk+QT}, $\mu$ is invariant under the Markov chain described in Proposition~\ref{prop:resample-nonsimple} for $\rho_1=0$ and $\rho_2 = \frac{4}{\gamma^2}(2+W-\gamma^2)$, and given $(\eta^L,\eta^R)$, $\eta'$ has the law $\SLE_{\kappa'}(\frac{\kappa'}{2}-4;\frac{\kappa'}{2}-4)$ in each connected component $\bbD\backslash(\eta^L\cup\eta^R)$ between $\eta^L$ and $\eta^R$. Therefore by Proposition~\ref{prop:resample-nonsimple} and Theorem~\ref{thm:SLE-duality}, we conclude that $\eta'$ must be radial $\SLE_{\kappa'}^\beta(0;\frac{4}{\gamma^2}(2+W-\gamma^2))$ for some possibly random $\beta$. Finally we identify $\beta=0$ using the same argument as in the last paragraph in the proof of Lemma~\ref{lem:beta-law-2}.
	\end{proof}

	\subsection{Theorem~\ref{thm:w2w-ns} for $W < \frac{\gamma^2}2$}\label{sec-thin-nonsimple}
	In this section we will prove   Proposition~\ref{prop-ns-thin} below, using an argument parallel to that of Section~\ref{sec:pf-thin}. As we will shortly see, Proposition~\ref{prop-ns-thin} immediately implies Theorem~\ref{thm:w2w-ns} for $W < \frac{\gamma^2}2$.
	
	Let $\gamma \in (\sqrt2,2)$, $\kappa' = 16/\gamma^2$ and $W \in (0, \frac{\gamma^2}2)$. 
	Sample a quantum disk from $\mathcal{M}^\mathrm{f.d.}_{0,2}(W)$, restrict to the event that the generalized boundary length of the left boundary arc is greater than that of its right boundary arc, and conformally weld the entire right boundary arc to the initial segment of the left boundary arc starting from the first marked point. This gives a forested quantum surface $\tilde D^f$ with a bulk marked point and a boundary marked point. Let $\tilde D$ be the connected component  containing the marked points, and embed it as $(\bbH, \phi, \eta, i, \infty)$, so the curve $\eta$ goes from $\infty$ to $i$. Let $M_W^f$ be the law of $\tilde D^f$, and  $M_W$  the law of $(\phi, \eta)$.
	\begin{proposition}\label{prop-ns-thin}
		There is a constant $c(W)\in (0,\infty)$ such that  \[M_W^f = c(W)\mathcal{M}_{1,1}^\mathrm{f.d.} (W_1, W_1 + \gamma^2-2) \otimes \mathrm{raSLE}_{\kappa'}(0; \rho), \qquad \text{where }\rho = \frac4{\gamma^2}W + \frac8{\gamma^2} - 4.\]
	\end{proposition}
	Equivalently, $M_W^f$ can be described via forested quantum triangles:
	
	\begin{lemma}\label{lem-MWf-QTf}
		There is a constant $C= C(W)$ such that 
		\[M_W^f  = C\int_0^\infty \Wd(\QT^f(W,W,\gamma^2-2;\ell,\ell))\,d\ell.\]
	\end{lemma}
	\begin{proof}
		The result is immediate from  Lemmas~\ref{lem-disintegrate-3-pt-ns} and~\ref{lem-embed-marked-disk-ns}.
	\end{proof}
	
	\begin{proof}[{Proof of  Theorem~\ref{thm:w2w-ns} for $W < \frac{\gamma^2}2$}]
		It follows from Proposition~\ref{prop-ns-thin} and Lemma~\ref{lem-MWf-QTf}. 
	\end{proof}
	
	The rest of this section is devoted to the proof of  Proposition~\ref{prop-ns-thin}. To that end, we state and prove the analogs of Propositions~\ref{prop-resample-inside} and~\ref{prop-resample-1}.
	
	\begin{lemma}\label{lem-resample-inside-ns}
		For $(\phi, \eta) \sim M_W$, the conditional law of $\eta$ given $\phi$ is radial $\SLE_{\kappa'}(\rho)$ in $\bbH$ from $\infty$ to $i$ with force point infinitesimally counterclockwise of $\infty$, where $\rho = \frac4{\gamma^2}W + \frac8{\gamma^2} - 4$.  Furthermore, consider a bounded neighborhood $A\subset \bbH$ of $i$ such that $\ol A \cap \partial \bbH = \emptyset$, and let $\alpha = Q - \frac1{2\gamma} (W + 2-\frac{\gamma^2}2)$. Conditioned on $\phi |_{\bbH \backslash A}$, we have $\phi|_A \stackrel d= h + \mathfrak h + \alpha G_A(\cdot, i)$, where $\mathfrak h$ is the harmonic extension of $\phi|_{\bbH \backslash A}$ to $A$, $h$ is a zero boundary GFF on $A$ and $G_A$ is the Green function of $h$. 
	\end{lemma}
	\begin{proof}
		The proof is identical to that of Proposition~\ref{prop-resample-inside}, except that in the argument of Lemma~\ref{lem-decomp-cone}, instead of using \cite[Theorem 1.2.4]{DMS14} which states that cutting a weight $W$ quantum cone by an independent whole-plane $\SLE_\kappa(W-2)$ gives a weight $W$ quantum wedge, we instead use \cite[Theorem 1.4.9]{DMS14} which states that cutting a weight $(W + 2-\frac{\gamma^2}2)$ quantum cone by an independent $\SLE_{\kappa'}(\rho)$ curve gives a weight $W$ forested quantum wedge. 
	\end{proof}
	
	\begin{lemma}\label{lem-resample-1-ns}
		Let $A \subset \bbH$ be a neighborhood of $\infty$ not containing $i$ such that $\bbH \backslash A$ is simply connected and  $\partial \bbH \backslash \ol A$ contains an interval,  and let $\beta = Q - \frac W\gamma$.
        For $(\phi, \eta) \sim M_W$, conditioned on $\phi|_{\bbH \backslash A}$, we have $\phi|_A \stackrel d= h + \mathfrak h + (\frac\beta2 - Q)G_A(\cdot, \infty)$, where $\mathfrak h$ is the harmonic extension of $\phi|_{\bbH \backslash A}$ to $A$ which has normal derivative zero on $\partial A \cap \partial \bbH$,
		$h$ is a mixed boundary GFF on $A$ with zero boundary conditions on $\partial A \cap \bbH$ and free boundary conditions on $\partial A \cap \partial \bbH$, and $G_A$ is the Green function of $h$. 
	\end{lemma}
	To prove Lemma~\ref{lem-resample-1-ns}, we use the following decomposition of a sample from $M_W^f$, see Figure~\ref{fig-sample-infty-nonsimple}. 
	
	\begin{lemma}\label{lem-decomp-qtf}
		There is a constant $C$ such that a sample from $M_W^f$ can be obtained by conformally welding a tuple $(\cD_1, \cD_2, \cD_3, \mathcal T)$ sampled from 
		\eqb\label{eq-decomp-qtf}
		\begin{split}
			C \int 1_{\ell_1 + \ell + \ell_3 > r + r_2} \Big(&\mathcal M^\mathrm{f.d.}_{0,2}(W; \ell_1+\ell+\ell_3 - r - r_2, \ell_1) \times \mathcal  M_{0,2}^\mathrm{f.d.}(W; \cdot, r_2) \\ &\times \mathcal M_{0,2}^\mathrm{f.d.}(\gamma^2-2; \ell_3, \cdot) \times \QT^f(\gamma^2-W,\gamma^2-W,\gamma^2-2; r, \ell)\Big),
		\end{split}
		\eqe
		where in~\eqref{eq-decomp-qtf} we integrate over $(\ell_1, \ell, \ell_3, r, r_2) \in (0,\infty)^5$ with respect to Lebesgue measure.

		Moreover, for a sample from $M_W^f$, under the embedding of $\tilde \cD$ as $(\bbH, \phi, \eta, i, \infty)$, let $\tau$ be the time that $\eta$ disconnects $i$ from $\partial \bbH$, let $G \subset \bbH$ be the connected component of $\bbH \backslash \eta$ such that $\eta(\tau) \in \ol G$ and $\ol G \cap \partial \bbH\neq \emptyset$, and let $[x_2, x_3] = \ol G \cap \partial \bbH$. Let $E$ be the event that $\eta$ hits $x_2$ before hitting $x_3$. Then a sample from $1_E M_W^f$ can be obtained from conformally welding $(\cD_1, \cD_2, \cD_3, \mathcal T)$ sampled from~\eqref{eq-decomp-qtf} but with the integrand containing the additional factor $1_{\ell_3 > r_2}$.
	\end{lemma}
	\begin{proof}
		Since the operation of foresting a quantum surface is Poissonian with respect to quantum boundary length, the first claim follows from Lemma~\ref{lem-MWf-QTf} and Definition~\ref{def-qt-thin}. For the second claim, notice that $x_2$ (resp.\ $x_3$) is the image of the first marked point of $\cD_2$ (resp.\ $\cD_3$) under the embedding, so $E$ corresponds to the entire right boundary arc of $\cD_2$ being conformally welded to the left boundary arc of $\cD_3$, i.e., $r_2 < \ell_3$. 
	\end{proof}
	
	\begin{proof}[{Proof of Lemma~\ref{lem-resample-1-ns}}]
		Recall the event $E = E(\eta)$ defined in Lemma~\ref{lem-decomp-qtf}. Let $\sigma$ be the first time after $\eta$ hits $x_2$ that $\eta$ hits $\partial \bbH$. On the event $E$, let $U$ be the complement of the connected component of $\bbH \backslash \eta([0,\sigma])$ containing $i$.
        We claim that conditioned on $\eta$, $E$ and $\phi|_{\bbH \backslash U}$, we have $\phi|_U \stackrel d= h_0 + \mathfrak h_0  + (\frac \beta2 - Q) G_U(\cdot, \infty)$ where $h_0$ is a GFF on $U$ with zero boundary conditions on $\partial U \cap \bbH$ and free boundary conditions elsewhere, $\mathfrak h_0$ is the harmonic extension of $\phi|_{\bbH \backslash U}$ to $U$ having zero normal derivative on $\partial U \cap \bbH$, and $G_U$ is the Green function for $h_0$. This claim is the analog of Lemma~\ref{lem-resample-infty-2}; using it, the proof of Lemma~\ref{lem-resample-1-ns} is identical to that of Lemma~\ref{lem-resample-bdy} using Lemma~\ref{lem-resample-infty-2}. 
		
		We now turn to the proof of the claim, see Figure~\ref{fig-sample-infty-nonsimple-2}. Let $\mathcal T'$ be $\cD_3$ after adding a third marked point on the left boundary arc at generalized boundary length $\ell_3 - r_2$ from the first vertex, and reorder the vertices of $\mathcal T'$ so that the generalized boundary length from the first to the second (resp.\ third) vertex is $\ell_3 - r_2$ (resp.\ $r_2$). By  Lemmas~\ref{lem-embed-marked-disk-ns} and~\ref{lem-disintegrate-3-pt-ns}, for $\cD_3\sim \cM^\mathrm{f.d.}_{0,2}(\gamma^2-2; \ell_3, \cdot)$ the law of $\mathcal T'$ is $C \QT^f(\gamma^2-2, \gamma^2-2, \gamma^2-2; \ell_3 - r_2, r_2)$ for some constant $C$. By Definition~\ref{def-qt-thin}, if $\cD'$ denotes the forested quantum disk corresponding to the second arm of $\mathcal T'$ and $\mathcal T''$ denotes the rest of $\mathcal T'$, for $\cD_3\sim \cM^\mathrm{f.d.}_{0,2}(\gamma^2-2; \ell_3, \cdot)$ the law of $(\mathcal T'', \cD')$ is 
		\[C\int_0^{\ell_3 - r_2} (\frac4{\gamma^2}-1) \QT^f(\gamma^2-2, 2, \gamma^2-2; a, r_2) \times \mathcal M^\mathrm{f.d.}_{0,2}(\gamma^2-2; \cdot, \ell_3 - r_2 - a) \, da.\]
		Finally, by \cite[Theorem 1.4]{AHSY23} and Lemma~\ref{lem-disk=qt-thin}, for $(\cD_2, \mathcal T'') \sim \int_0^\infty \mathcal M^\mathrm{f.d.}_{0,2}(W; \cdot, r_2)  \times \QT^f(\gamma^2 - 2, 2, \gamma^2 - 2;a, r_2) \, dr_2$, conformally welding $\cD_2$ and $\mathcal T''$ along their boundary arcs of generalized boundary length $r_2$ gives a forested quantum surface decorated by a curve, whose law is $C \QT^f(W + \frac{\gamma^2}2, 2, W+\frac{\gamma^2}2; a) \otimes \SLE_{\kappa'}(\rho; 0)$, where $\rho = \frac4{\gamma^2}(2-\gamma^2 + W)$ and $C$ is a constant. Forgetting the foresting and embedding in $(\bbH, 0, 1, \infty)$ gives a field $\psi$ whose law is a multiple of $\LF_\bbH^{(\beta, 0), (\gamma, 1), (\beta, \infty)}$ by Lemma~\ref{lem-embed-marked-disk}. Exactly as in the proof of Lemma~\ref{lem-resample-infty-2}, this implies the claim. 
	\end{proof}

	\begin{figure}[ht]
		\centering
		\includegraphics[scale=0.4]{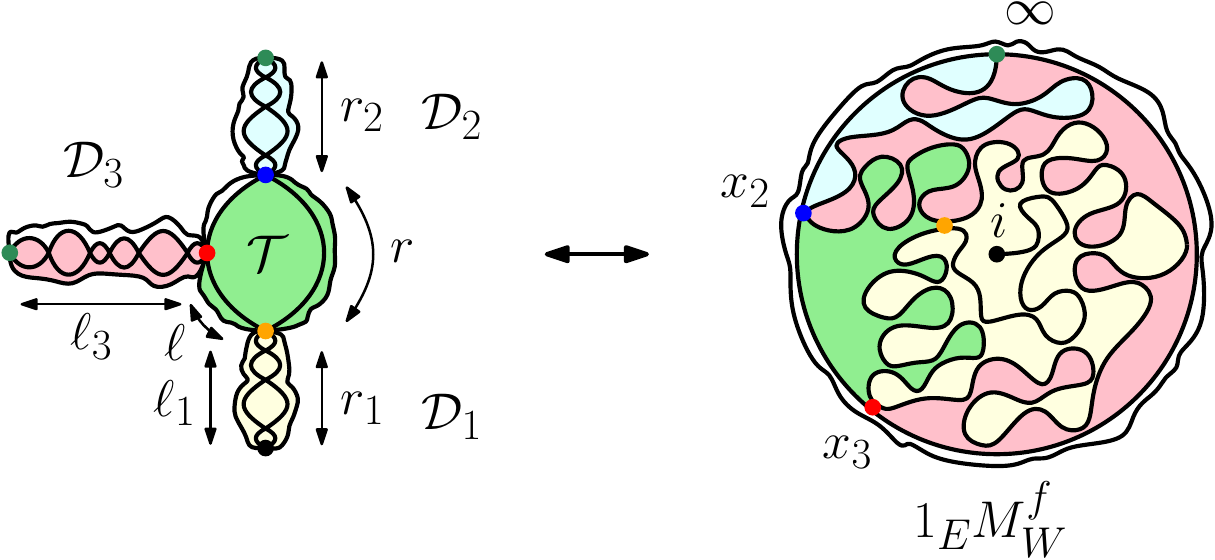}
		\caption{ Diagram for Lemma~\ref{lem-decomp-qtf}. 
			To reduce clutter, we draw wavy boundaries to represent the generalized quantum disks attached to the boundary of a forested quantum surface. 
			Equation~\eqref{eq-decomp-qtf} (with the additional factor $1_{\ell_3>r_2}$) gives a decomposition of $1_E M_W^f$; in the figure we set $r_1 = \ell_1+\ell+\ell_3 - r - r_2$ since the boundary arc of length $\ell_1 + \ell + \ell_3$ is to be conformally welded to the boundary arc of length $r_1+r + r_2$. 
			On the right, the half-plane $\bbH$ is drawn using the unit disk for clarity. The curve $\eta$ from $\infty$ to $i$ disconnects $i$ from $\partial \bbH$ at time $\tau$; the point $\eta(\tau)$ is shown in orange. The boundary of the green region intersects $\partial \bbH$ along an interval which we call $[x_2, x_3]$. The event $E = \{ \eta \text{ hits } x_2 \text{ before }x_3\}$ corresponds to $\{ r_2 < \ell_3\}$. 
		}\label{fig-sample-infty-nonsimple}
	\end{figure}
	
	\begin{figure}[ht]
		\centering
		\includegraphics[scale=0.39]{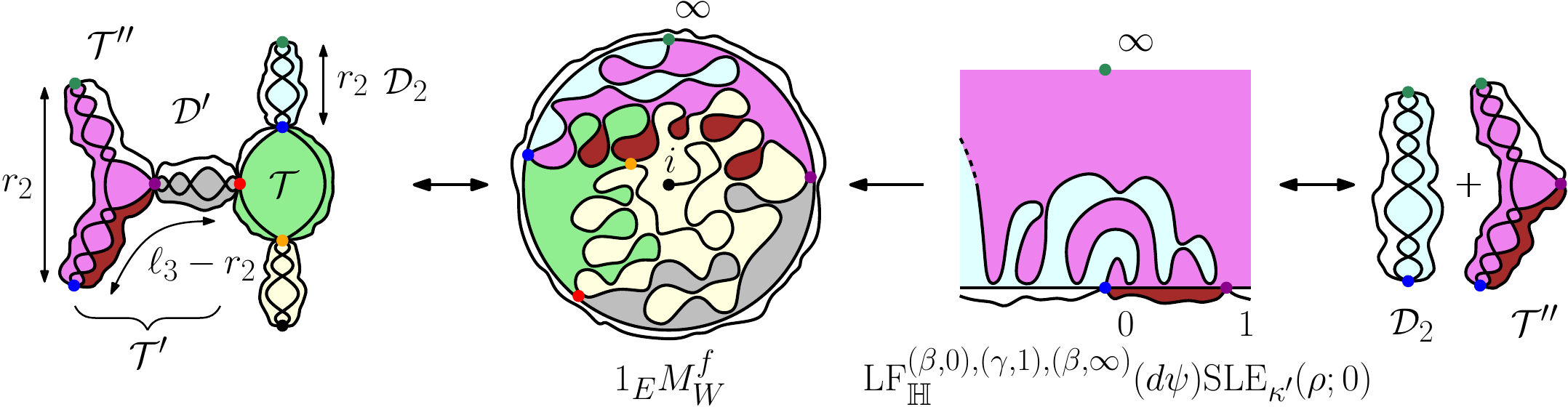}
		\caption{Diagram for the proof of Lemma~\ref{lem-resample-1-ns}. \textbf{Left arrow:} We add a marked point (blue) to $\cD_3$ to obtain a quantum triangle $\mathcal T'$ by Lemmas~\ref{lem-embed-marked-disk-ns} and~\ref{lem-disintegrate-3-pt-ns}. Cutting $\mathcal T'$ gives $\mathcal T''$ and $\cD'$. \textbf{Right arrow:} \cite{AHSY23} implies that the conformal welding of $\cD_2$ and $\mathcal T''$ is described by a Liouville field. \textbf{Middle arrow:} A resampling property of the Liouville field implies a corresponding resampling property for 
			$\phi|_U$, where $U$ is the union of the blue and pink regions.  }\label{fig-sample-infty-nonsimple-2}
	\end{figure}
	
	\begin{proof}[{Proof of Proposition~\ref{prop-ns-thin}}]
		Let $\alpha, \beta, \rho \in \bbR$ take the values given in Lemmas~\ref{lem-resample-inside-ns} and~\ref{lem-resample-1-ns}.
		Arguing exactly as in the proof of Theorem~\ref{thm:w2w} stated in Section~\ref{sec-proof-main-thin}, using Lemmas~\ref{lem-resample-inside-ns} and~\ref{lem-resample-1-ns} instead of Propositions~\ref{prop-resample-inside} and~\ref{prop-resample-1}, we deduce that the $M_W$-law of $(\phi, \eta)$ is $C \LF_\bbH^{(\alpha, i), (\beta, \infty)} \times \mathrm{raSLE}_{\kappa'}(0; \rho)$ for some constant $C$, so we have identified the law of $\tilde \cD$.  Since Lemma~\ref{lem-MWf-QTf} implies that conditioned on $\tilde \cD$, we can resample $\tilde \cD^f$ by foresting the boundary of $\tilde \cD$, we conclude the law of $\tilde \cD^f$ is $\mathcal{M}_{1,1}^\mathrm{f.d.} (W_1 + 2 - \frac{\gamma^2}2, W_1 + \frac{\gamma^2}2) \otimes \mathrm{raSLE}_{\kappa'}(0; \rho)$, as claimed.
	\end{proof}

	\subsection{Proof of Theorem~\ref{thm:main-nonsimple}}\label{sec-proof-nonsimple}
    The proof of  Theorem~\ref{thm:main-nonsimple} given Theorem~\ref{thm:w2w-ns} is mostly identical to the proof of Theorem~\ref{thm:main} given Theorem~\ref{thm:w2w} as in Section~\ref{subsec:weld-qt}, but it further requires the description of counterflowlines in Theorem~\ref{thm:SLE-duality}. This argument has already been given in~\cite[Theorem 5.16]{LSYZ24} for $W_1 = 2- \frac{\gamma^2}2$, in which they bootstrap the $(W_1, W_2, W_3) = (W_1, W_1, \gamma^2-2)$ result of~\cite[Theorem 3.1]{ASYZ24} to the general case ($W_2 + W_3 = W_1 + \gamma^2-2$). The proof of \cite[Theorem 5.16]{LSYZ24} holds identically for general $W_1$, so Theorem~\ref{thm:main-nonsimple} follows from  Theorem~\ref{thm:w2w-ns}.

\bibliographystyle{alpha}
\bibliography{theta1}

\end{document}